\documentclass[a4paper]{amsart}

\usepackage{enumerate, amsmath, amsfonts, amssymb, amsthm, wasysym, graphics, graphicx, xcolor, url, hyperref, hypcap, a4wide, stmaryrd, shuffle, xargs, multicol, multirow, overpic, pdflscape, hvfloat, minibox, accents, pifont, etoolbox, colortbl, calc}
\hypersetup{colorlinks=true, citecolor=darkblue, linkcolor=darkblue, urlcolor=darkblue}
\usepackage[all]{xy}
\usepackage[bottom]{footmisc}
\usepackage{tikz}\usetikzlibrary{trees,snakes,shapes,arrows,matrix,calc}
\graphicspath{{figures/}{figures/nodes/}{figures/relations/}}
\makeatletter
\def\input@path{{figures/}}
\makeatother

\title{The weak order on integer posets}

\thanks{VPi~was partially supported by the French ANR grant SC3A~(15\,CE40\,0004\,01).}

\author{Gr\'egory Chatel}
\author{Vincent Pilaud}
\author{Viviane Pons}

\address[GC]{LIGM, Univ.\,Paris-Est Marne-la-Vall\'ee}
\email{gregory.chatel@univ-paris-est.fr}
\urladdr{\url{http://igm.univ-mlv.fr/~gchatel/}}

\address[VPi]{CNRS \& LIX, \'Ecole Polytechnique, Palaiseau}
\email{vincent.pilaud@lix.polytechnique.fr}
\urladdr{\url{http://www.lix.polytechnique.fr/~pilaud/}}

\address[VPo]{LRI, Univ.\,Paris-Sud}
\email{viviane.pons@lri.fr}
\urladdr{\url{http://www.lri.fr/~pons/}}

\newtheorem{theorem}{Theorem}
\newtheorem{corollary}[theorem]{Corollary}
\newtheorem{proposition}[theorem]{Proposition}
\newtheorem{lemma}[theorem]{Lemma}
\newtheorem{definition}[theorem]{Definition}

\theoremstyle{definition}
\newtheorem{example}[theorem]{Example}
\newtheorem{remark}[theorem]{Remark}

\newtheorem{claim}{Claim}

\newcommand{\R}{\mathbb{R}} 
\newcommand{\N}{\mathbb{N}} 

\newcommand{\set}[2]{\left\{ #1 \;\middle|\; #2 \right\}} 
\newcommand{\bigset}[2]{\big\{ #1 \;\big|\; #2 \big\}} 
\newcommand{\ssm}{\smallsetminus} 
\newcommand{\eqdef}{\mbox{\,\raisebox{0.2ex}{\scriptsize\ensuremath{\mathrm:}}\ensuremath{=}\,}} 
\renewcommand{\implies}{\;\Longrightarrow\;} 

\newcommandx{\rel}[1][1=R]{\mathbin{\mathrm{#1}}} 
\newcommandx{\notrel}[1][1=R]{\mathbin{\!\raisebox{.02cm}{$\not$}\hspace{.02cm}\mathrm{#1}\hspace*{.01cm}}} 
\newcommand{\less}{\vartriangleleft} 
\newcommand{\more}{\vartriangleright} 
\newcommand{\bless}{\blacktriangleleft} 
\newcommand{\bmore}{\blacktriangleright} 
\newcommand{\dashV}{\rotatebox[origin=c]{180}{$\Vdash$}} 

\newcommand{\IRel}{\mathcal{R}} 
\newcommand{\ITrans}{\mathcal{T}} 
\newcommand{\ISemiTrans}{\mathcal{ST}} 
\newcommand{\ISym}{\mathcal{S}} 
\newcommand{\IAntisym}{\mathcal{A}} 
\newcommand{\IEquiv}{\mathcal{E}} 
\newcommand{\IPos}{\mathcal{P}} 
\newcommand{\IInc}{\mathcal{I}} 
\newcommand{\IDec}{\mathcal{D}} 

\newcommand{\X}{\mathsf{X}} 
\newcommand{\WOEP}{\mathsf{WOEP}} 
\newcommand{\WOFP}{\mathsf{WOFP}} 
\newcommand{\WOIP}{\mathsf{WOIP}} 
\newcommand{\IWOIP}{\mathsf{IWOIP}} 
\newcommand{\DWOIP}{\mathsf{DWOIP}} 
\newcommand{\TOEP}{\mathsf{TOEP}} 
\newcommand{\TOFP}{\mathsf{TOFP}} 
\newcommand{\TOIP}{\mathsf{TOIP}} 
\newcommandx{\CO}[1][1=\signature]{\mathsf{CO}(#1)} 
\newcommandx{\COEP}[1][1=\signature]{\mathsf{COEP}(#1)} 
\newcommandx{\COFP}[1][1=\signature]{\mathsf{COFP}(#1)} 
\newcommandx{\COIP}[1][1=\signature]{\mathsf{COIP}(#1)} 
\newcommandx{\ICOIP}[1][1=\signature]{\mathsf{ICOIP}(#1)} 
\newcommandx{\DCOIP}[1][1=\signature]{\mathsf{DCOIP}(#1)} 
\newcommand{\BOEP}{\mathsf{BOEP}} 
\newcommand{\BOFP}{\mathsf{BOFP}} 
\newcommand{\BOIP}{\mathsf{BOIP}} 

\newcommand{\orientation}{\mathbb{O}} 
\newcommand{\positive}{^{+}} 
\newcommand{\negative}{^{-}} 
\newcommandx{\PIP}[1][1=\orientation]{\mathsf{PIP}(#1)} 
\newcommandx{\IPIP}[1][1=\orientation]{\mathsf{IPIP}(#1)} 
\newcommandx{\IPIPp}[1][1=\orientation]{\mathsf{IPIP}\positive(#1)} 
\newcommandx{\IPIPm}[1][1=\orientation]{\mathsf{IPIP}\negative(#1)} 
\newcommandx{\IPIPpm}[1][1=\orientation]{\mathsf{IPIP}^{\pm}(#1)} 
\newcommandx{\IPIPe}[1][1=\orientation]{\mathsf{IPIP}^{\varepsilon}(#1)} 
\newcommandx{\DPIP}[1][1=\orientation]{\mathsf{DPIP}(#1)} 
\newcommandx{\DPIPp}[1][1=\orientation]{\mathsf{DPIP}\positive(#1)} 
\newcommandx{\DPIPm}[1][1=\orientation]{\mathsf{DPIP}\negative(#1)} 
\newcommandx{\DPIPpm}[1][1=\orientation]{\mathsf{DPIP}^{\pm}(#1)} 
\newcommandx{\DPIPe}[1][1=\orientation]{\mathsf{DPIP}^{\varepsilon}(#1)} 
\newcommandx{\incompE}[1][1=\orientation]{\mathsf{sn}(#1)} 
\newcommandx{\PEP}[1][1=\orientation]{\mathsf{PEP}(#1)} 
\newcommandx{\incompF}[1][1=\orientation]{\mathsf{f}(#1)} 
\newcommandx{\PFP}[1][1=\orientation]{\mathsf{PFP}(#1)} 

\newcommandx{\Inc}[1]{#1^{\mathsf{Inc}}} 
\newcommandx{\Dec}[1]{#1^{\mathsf{Dec}}}
\newcommand{\rev}[1]{#1^{\mathsf{rev}}} 
\newcommand{\tc}[1]{#1^{\mathsf{tc}}} 
\newcommand{\tid}[1]{#1^{\mathsf{tid}}} 
\newcommand{\tdd}[1]{#1^{\mathsf{tdd}}} 
\newcommand{\IWOIPid}[1]{#1^{\mathsf{IWOIPid}}} 
\newcommand{\DWOIPdd}[1]{#1^{\mathsf{DWOIPdd}}} 
\newcommand{\WOIPd}[1]{#1^{\mathsf{WOIPd}}} 
\newcommand{\maxle}[1]{#1^{\mathsf{maxle}}} 
\newcommand{\minle}[1]{#1^{\mathsf{minle}}} 
\newcommand{\TOIPd}[1]{#1^{\mathsf{TOIPd}}} 
\newcommandx{\IPIPpid}[2][2=\orientation]{#1^{\mathsf{IPIP^{+}id}(#2)}} 
\newcommandx{\IPIPmid}[2][2=\orientation]{#1^{\mathsf{IPIP^{-}id}(#2)}} 
\newcommandx{\IPIPpmid}[2][2=\orientation]{#1^{\mathsf{IPIP^{\pm}id}(#2)}} 
\newcommandx{\IPIPid}[2][2=\orientation]{#1^{\mathsf{IPIPid}(#2)}} 
\newcommandx{\IPIPeid}[2][2=\orientation]{#1^{\mathsf{IPIP^{\varepsilon}id}(#2)}} 
\newcommandx{\DPIPpdd}[2][2=\orientation]{#1^{\mathsf{DPIP^{+}dd}(#2)}} 
\newcommandx{\DPIPmdd}[2][2=\orientation]{#1^{\mathsf{DPIP^{-}dd}(#2)}} 
\newcommandx{\DPIPpmdd}[2][2=\orientation]{#1^{\mathsf{DPIP^{\pm}dd}(#2)}} 
\newcommandx{\DPIPdd}[2][2=\orientation]{#1^{\mathsf{DPIPdd}(#2)}} 
\newcommandx{\DPIPedd}[2][2=\orientation]{#1^{\mathsf{DPIP^{\varepsilon}dd}(#2)}} 
\newcommandx{\PIPd}[2][2=\orientation]{#1^{\mathsf{PIPd}(#2)}} 
\newcommand{\PFPia}[1]{#1^{\mathsf{PFPia}}} 
\newcommand{\PFPda}[1]{#1^{\mathsf{PFPda}}} 

\newcommand{\meetT}{\wedge_\ITrans} 
\newcommand{\joinT}{\vee_\ITrans} 
\newcommand{\meetST}{\wedge_\ISemiTrans} 
\newcommand{\joinST}{\vee_\ISemiTrans} 
\newcommand{\meetR}{\wedge_\IRel} 
\newcommand{\joinR}{\vee_\IRel} 
\newcommand{\joinWO}{\vee_\fS} 
\newcommand{\meetWO}{\wedge_\fS} 
\newcommand{\joinWOIP}{\vee_\WOIP} 
\newcommand{\meetWOIP}{\wedge_\WOIP} 
\newcommand{\joinIWOIP}{\vee_\IWOIP} 
\newcommand{\meetIWOIP}{\wedge_\IWOIP} 
\newcommand{\joinDWOIP}{\vee_\DWOIP} 
\newcommand{\meetDWOIP}{\wedge_\DWOIP} 
\newcommand{\joinWOFP}{\vee_\WOFP} 
\newcommand{\meetWOFP}{\wedge_\WOFP} 
\newcommand{\joinTO}{\vee_\fB} 
\newcommand{\meetTO}{\wedge_\fB} 
\newcommand{\joinTOIP}{\vee_\TOIP} 
\newcommand{\meetTOIP}{\wedge_\TOIP} 
\newcommand{\joinTOFP}{\vee_\TOFP} 
\newcommand{\meetTOFP}{\wedge_\TOFP} 
\newcommand{\joinO}{\vee_\orientation} 
\newcommand{\meetO}{\wedge_\orientation} 
\newcommand{\joinPIP}{\vee_{\PIP}} 
\newcommand{\meetPIP}{\wedge_{\PIP}} 
\newcommand{\joinPEP}{\vee_{\PEP}} 
\newcommand{\meetPEP}{\wedge_{\PEP}} 
\newcommand{\joinPFP}{\vee_{\PFP}} 
\newcommand{\meetPFP}{\wedge_{\PFP}} 

\DeclareMathOperator{\inv}{inv} 
\DeclareMathOperator{\ver}{ver} 
\newcommand{\wole}{\preccurlyeq} 
\newcommand{\wo}{w_\circ} 
\newcommand{\bt}{\mathrm{bt}} 
\newcommand{\st}{\mathrm{st}} 

\newcommand{\fS}{\mathfrak{S}} 
\newcommand{\fB}{\mathfrak{B}} 

\newcommandx{\conflicts}[1][1=cf]{\mathsf{#1}} 
\newcommand{\free}{\mathcal{F}} 

\DeclareMathOperator{\conv}{conv} 
\newcommandx{\Asso}[1][1=\signature]{\mathsf{Asso}(#1)} 
\newcommandx{\Perm}[1][1=n]{\mathsf{Perm}(#1)} 
\newcommandx{\Para}[1][1=n]{\mathsf{Para}(#1)} 
\newcommandx{\PT}[1][1=\orientation]{\mathsf{PT}(#1)} 

\newcommandx{\graphG}[1][1=G]{\mathrm{#1}} 
\newcommandx{\tree}[1][1=T]{\mathrm{#1}} 
\newcommandx{\tuple}[1][1=T]{\mathcal{#1}} 
\newcommandx{\poset}{{\tree[P]}} 
\newcommand{\ground}{\mathrm{V}} 
\newcommand{\signature}{\varepsilon} 

\newcommandx{\Permutrees}[1][1=\orientation]{\mathrm{PT}(#1)} 
\newcommandx{\SchroderPermutrees}[1][1=\orientation]{\mathrm{SchrPT}(#1)} 
\newcommandx{\surjection}[1][1=\orientation]{\Psi_{#1}} 
\newcommand{\linearExtensions}{\mathcal{L}} 
\newcommand{\none}[1]{\underaccent{\phantom{.}}{\accentset{\phantom{.}}{#1}}}
\newcommand{\up}[1]{\underaccent{\phantom{.}}{\overline{#1}}}
\newcommand{\down}[1]{\accentset{\phantom{.}}{\underline{#1}}}
\newcommand{\updown}[1]{\underline{\overline{#1}}}

\newcommand{\fref}[1]{Figure~\ref{#1}} 
\newcommand{\ie}{\textit{i.e.}~} 
\newcommand{\eg}{\textit{e.g.}~} 
\newcommand{\viceversa}{\textit{vice versa}} 
\definecolor{darkblue}{rgb}{0,0,0.7} 
\definecolor{green}{RGB}{57,181,74} 
\newcommand{\darkblue}{\color{darkblue}} 
\newcommand{\defn}[1]{\emph{\darkblue #1}} 
\usepackage{todonotes}

\newcommand{\para}[1]{\medskip\noindent\framebox{\textsc{#1}}\quad} 

\makeatletter
\def\part{\@startsection{part}{1}%
\z@{.7\linespacing\@plus\linespacing}{.5\linespacing}%
{\LARGE\sffamily\centering}}
\@addtoreset{section}{part}
\makeatother

\makeatletter
\def\l@section{\@tocline{1}{3pt}{0pc}{}{}}
\makeatother
\let\oldtocpart=\tocpart
\renewcommand{\tocpart}[2]{\hspace{0em}\bf\large\oldtocpart{#1}{#2}}
\let\oldtocsection=\tocsection
\renewcommand{\tocsection}[2]{\hspace{0em}\bf\oldtocsection{#1}{#2}}

\begin{document}

\begin{abstract}
We explore lattice structures on integer binary relations (\ie binary relations on the set~$\{1, 2, \dots, n\}$ for a fixed integer~$n$) and on integer posets (\ie partial orders on the set~$\{1, 2, \dots, n\}$ for a fixed integer~$n$).
We first observe that the weak order on the symmetric group naturally extends to a lattice structure on all integer binary relations. We then show that the subposet of this weak order induced by integer posets defines as well a lattice. We finally study the subposets of this weak order induced by specific families of integer posets corresponding to the elements, the intervals, and the faces of the permutahedron, the associahedron, and some recent generalizations of those.
\end{abstract}

\vspace*{-1.3cm}
\maketitle
\vspace*{-.2cm}


The \defn{weak order} is the lattice on the symmetric group~$\fS(n)$ defined as the inclusion order of inversions, where an \defn{inversion} of~$\sigma \in \fS(n)$ is a pair of values~$a < b$ such that~$\sigma^{-1}(a) > \sigma^{-1}(b)$. It is a fundamental tool for the study of the symmetric group, in connection to reduced expressions of permutations as products of simple transpositions. Its Hasse diagram can also be seen as a certain acyclic orientation of the skeleton of the permutahedron (the convex hull of all permutations of~$\fS(n)$ seen as vectors in~$\R^n$).

This paper extends the weak order to all \defn{integer binary relations}, \ie binary relations on the set~$[n] \eqdef \{1, 2, \dots, n\}$ for a fixed integer~$n$.
A permutation~$\sigma \in \fS(n)$ is seen as an binary relation~$\less$ on~$[n]$ where~$u \less v$ when~$u$ appears before~$v$ in~$\sigma$.
Inversions of~$\sigma$ then translates to decreasing relations of~$\less$, \ie elements~$a < b$ such that~$b \less a$.
This interpretation enables to naturally extend the weak order to all binary relations on~$[n]$ as follows.
For any two binary relations~$\rel, \rel[S]$ on~$[n]$, we define
\[
\rel \wole \rel[S] \quad \iff \quad {\Inc{\rel} \supseteq \Inc{\rel[S]}} \text{ and } {\Dec{\rel} \subseteq \Dec{\rel[S]}},
\]
where~$\Inc{\rel} \eqdef \set{(a,b) \in \rel}{a \le b}$ and~$\Dec{\rel} \eqdef \set{(b,a) \in \rel}{a \le b}$ respectively denote the increasing and decreasing subrelations of~$\rel$. We call this order the \defn{weak order} on integer binary relations, see \fref{fig:weakOrderRelations}. The central result of this paper is the following statement, see \fref{fig:weakOrderPosets}.

\begin{theorem}
\label{thm:main}
For any~$n \in \N$, the weak order restricted to the set of all posets on~$[n]$ is a lattice.
\end{theorem}

Our motivation for this result is that many relevant combinatorial objects can be interpreted by specific integer posets, and the subposets of the weak order induced by these specific integer posets often correspond to classical lattice structures on these combinatorial objects. To illustrate this, we study specific integer posets corresponding to the elements, to the intervals, and to the faces in the classical weak order, the Tamari and Cambrian lattices~\cite{TamariFestschrift, Reading-CambrianLattices}, the boolean lattice, and other related lattices defined in~\cite{PilaudPons}. By this systematic approach, we rediscover and shed light on lattice structures studied by G.~Chatel and V.~Pons on Tamari interval posets~\cite{ChatelPons}, by G.~Chatel and V.~Pilaud on Cambrian and Schr\"oder-Cambrian trees~\cite{ChatelPilaud}, by D.~Krob, M.~Latapy, J.-C.~Novelli, H.-D.~Phan and S.~Schwer on pseudo-permutations~\cite{KrobLatapyNovelliPhanSchwer}, and by P.~Palacios and M.~Ronco~\cite{PalaciosRonco} and J.-C.~Novelli and J.-Y.~Thibon~\cite{NovelliThibon-trialgebras} on plane trees.

The research code for experiments and computations along this work is available online~\cite{code}.


\part{The weak order on integer posets}

\section{The weak order on integer binary relations}
\label{sec:weakOrder}

\subsection{Integer binary relations}
\label{subsec:integerBinaryRelations}

Our main object of focus are binary relations on integers. An \defn{integer (binary) relation} of size~$n$ is a binary relation on~$[n] \eqdef \{1, \dots, n\}$, that is, a subset~$\rel$ of~$[n]^2$. As usual, we write equivalently~$(u,v) \in \rel$ or~$u \rel v$, and similarly, we write equivalently~$(u,v) \notin \rel$ or~$u \notrel v$.
Recall that a relation~$\rel \in [n]^2$ is called:
\begin{itemize}
\item \defn{reflexive} if $u \rel u$ for all~$u \in [n]$,
\item \defn{transitive} if ${u \rel v}$ and~${v \rel w}$ implies~${u \rel w}$ for all~$u,v,w \in [n]$,
\item \defn{symmetric} if $u \rel v$ implies~$v \rel u$ for all~$u,v \in [n]$,
\item \defn{antisymmetric} if $u \rel v \text{ and } v \rel u$ implies~$u = v$ for all~$u,v \in [n]$.
\end{itemize}
From now on, we only consider reflexive relations. We denote by~$\IRel(n)$ (resp.~$\ITrans(n)$, resp.~$\ISym(n)$, resp.~$\IAntisym(n)$) the collection of all reflexive (resp.~reflexive and transitive, resp.~reflexive and symmetric, resp.~reflexive and antisymmetric) integer relations of size~$n$. We denote by~$\IEquiv(n)$ the set of \defn{integer equivalences} of size~$n$, that is, reflexive transitive symmetric integer relations, and by~$\IPos(n)$ the collection of \defn{integer posets} of size~$n$, that is, reflexive transitive antisymmetric integer relations. In all these notations, we forget the~$n$ when we consider a relation without restriction on its size.

A \defn{subrelation} of~$\rel \in \IRel(n)$ is a relation~$\rel[S] \in \IRel(n)$ such that~$\rel[S] \subseteq \rel$ as subsets of~$[n]^2$. We say that~$\rel[S]$ \defn{coarsens}~$\rel$ and~$\rel$ \defn{extends}~$\rel[S]$. The extension order defines a graded lattice structure on~$\IRel(n)$ whose meet and join are respectively given by intersection and union. The complementation~${\rel \mapsto \set{(u,v)}{u = v \text{ or } u \notrel v}}$ is an antiautomorphism of~$(\IRel(n), \subseteq, \cap, \cup)$ and makes it an ortho-complemented lattice.

Note that~$\ITrans(n)$, $\ISym(n)$ and~$\IAntisym(n)$ are all stable by intersection, while only~$\ISym(n)$ is stable by~union. In other words, $(\ISym(n), \subseteq, \cap, \cup)$ is a sublattice of~$(\IRel(n), \subseteq, \cap, \cup)$, while~${(\ITrans(n), \subseteq)}$ and~$(\IAntisym(n), \subseteq)$ are meet-semisublattice of~$(\IRel(n), \subseteq, \cap)$ but not sublattices of~$(\IRel(n), \subseteq, \cap, \cup)$. However, $(\ITrans(n), \subseteq)$ is a lattice. To see it, consider the \defn{transitive closure} of a relation~$\rel \in \IRel(n)$~defined~by
\[
\tc{\rel} \eqdef \bigset{(u,w) \in [n]^2}{\exists\; v_1, \dots, v_p \in [n] \text{ such that } u = v_1 \rel v_2 \rel \dots \rel v_{p-1} \rel v_p = w}.
\]
The transitive closure~$\tc{\rel}$ is the coarsest transitive relation containing~$\rel$. It follows that~$(\ITrans(n), \subseteq)$ is a lattice where the meet of~$\rel, \rel[S] \in \IRel(n)$ is given by~$\rel \cap \rel[S]$ and the join of~$\rel, \rel[S] \in \IRel(n)$ is given by~${\tc{(\rel \cup \rel[S])}}$. Since the transitive closure preserves symmetry, the subposet $(\IEquiv(n), \subseteq)$ of integer equivalences is a sublattice of~$(\ITrans(n),\subseteq)$.

\subsection{Weak order}
\label{subsec:weakOrder}

From now on, we consider both a relation~$\rel$ and the natural order~$<$ on~$[n]$ simultaneously. To limit confusions, we try to stick to the following convention throughout the paper. We denote integers by letters~$a,b,c$ when we know that~$a < b < c$ in the natural order. In contrast, we prefer to denote integers by letters~$u, v, w$ when we do not know their relative order. This only helps avoid confusions and is always specified.

Let~$\rel[I]_n \eqdef \set{(a,b) \in [n]^2}{a \le b}$ and~$\rel[D]_n \eqdef \set{(b,a) \in [n]^2}{a \le b}$. Observe that~$\rel[I]_n \cup \rel[D]_n = [n]^2$ while~$\rel[I]_n \cap \rel[D]_n = \set{(a,a)}{a \in [n]}$. We say that the relation~$\rel \in \IRel(n)$ is \defn{increasing} (resp.~\defn{decreasing}) when~$\rel \subseteq \rel[I]_n$ (resp.~$\rel \subseteq \rel[D]_n$). We denote by~$\IInc(n)$ (resp.~$\IDec(n)$) the collection of all increasing (resp.~decreasing) relations on~$[n]$. The \defn{increasing} and \defn{decreasing subrelations} of an integer relation~$\rel \in \IRel(n)$ are the relations defined~by:
\[
\Inc{\rel} \eqdef \rel \cap {}\rel[I]_n{} = \bigset{(a,b) \in \rel}{a \le b} \in \IInc(n)
\quad\text{and}\quad
\Dec{\rel} \eqdef \rel \cap {}\rel[D]_n{} = \bigset{(b,a) \in \rel}{a \le b} \in \IDec(n).
\]
In our pictures, we always represent an integer relation~$\rel \in \IRel(n)$ as follows: we write the numbers $1, \dots, n$ from left to right and we draw the increasing relations of~$\rel$ above in blue and the decreasing relations of~$\rel$ below in red. Although we only consider reflexive relations, we always omit the relations~$(i,i)$ in the pictures (as well as in our explicit examples). See \eg \fref{fig:weakOrderRelations}.

Besides the extension lattice mentioned above in Section~\ref{subsec:integerBinaryRelations}, there is another natural poset structure on~$\IRel(n)$, whose name will be justified in Section~\ref{sec:relevantFamiliesPermutahedron}.

\hvFloat[floatPos=p, capWidth=h, capPos=r, capAngle=90, objectAngle=90, capVPos=c, objectPos=c]{figure}
{\includegraphics[scale=.39]{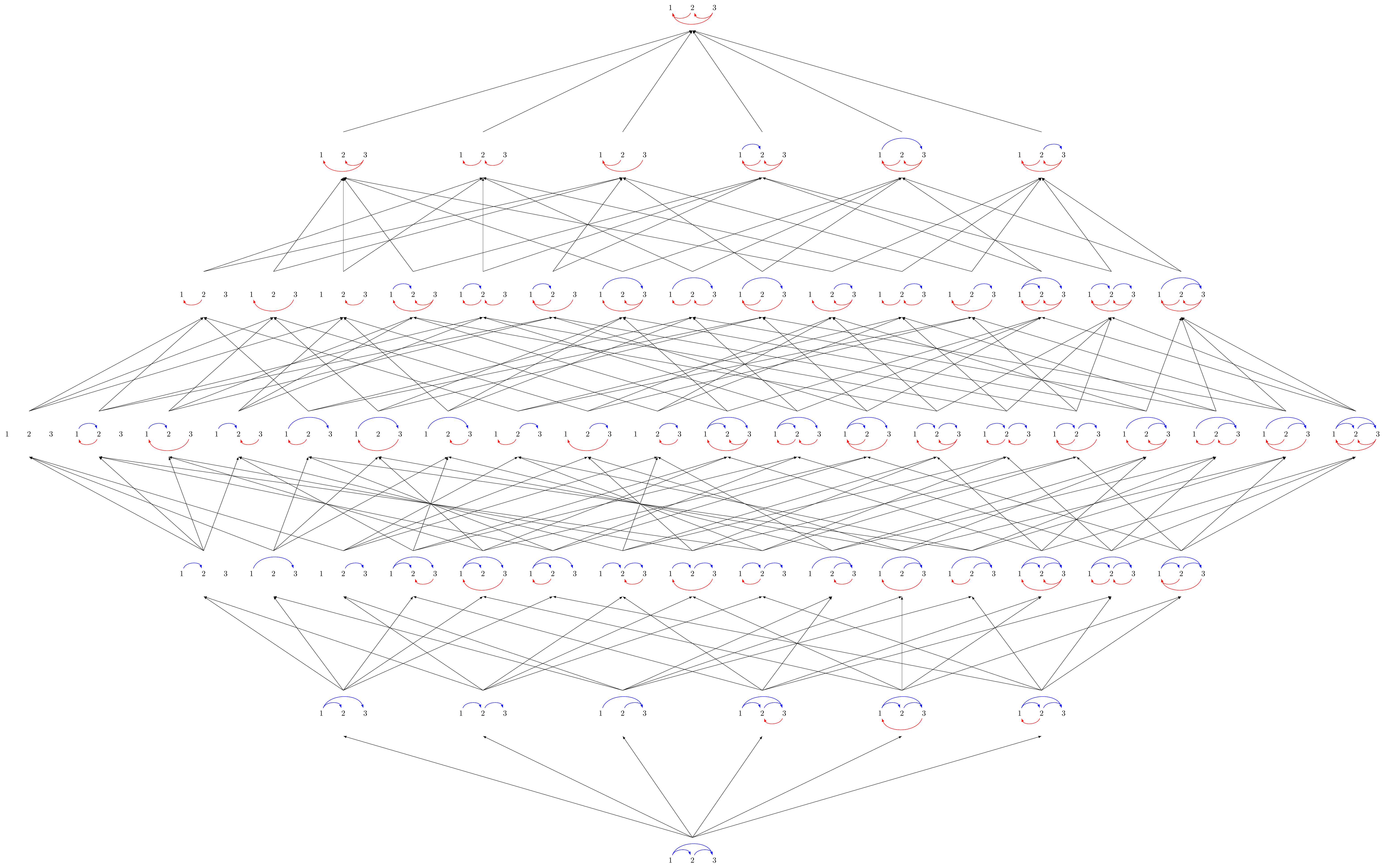}}
{The weak order on (reflexive) integer binary relations of size~$3$. All reflexive relations~$(i,i)$ for~$i \in [n]$ are omitted.}
{fig:weakOrderRelations}

\begin{definition}
\label{def:weakOrder}
The \defn{weak order} on~$\IRel(n)$ is the order defined by~$\rel \wole \rel[S]$ if~${\Inc{\rel} \supseteq \Inc{\rel[S]}}$ and~${\Dec{\rel} \subseteq \Dec{\rel[S]}}$.
\end{definition}

The weak order on~$\IRel(3)$ is illustrated in \fref{fig:weakOrderRelations}. Observe that the weak order is obtained by combining the extension lattice on increasing subrelations with the coarsening lattice on decreasing subrelations. In other words, $\IRel(n)$ is the square of an $\binom{n}{2}$-dimensional boolean lattice. It explains the following statement.

\begin{proposition}
\label{prop:reflexiveLattice}
The weak order~$(\IRel(n), \wole)$ is a graded lattice whose meet and join are given by
\[
{\rel[R]} \meetR {\rel[S]} = ( \Inc{\rel} \cup \Inc{\rel[S]} ) \cup ( \Dec{\rel} \cap \Dec{\rel[S]} )
\qquad\text{and}\qquad
{\rel[R]} \joinR {\rel[S]} = ( \Inc{\rel} \cap \Inc{\rel[S]} ) \cup ( \Dec{\rel} \cup \Dec{\rel[S]} ).
\]
\end{proposition}

\begin{proof}
The weak order is clearly a poset (antisymmetry comes from the fact that~$\rel = \Inc{\rel} \cup \Dec{\rel}$). Its cover relations are all of the form~$\rel \wole \rel \ssm \{(a,b)\}$ for~$a \Inc{\rel} b$ or~$\rel \ssm \{(b,a)\} \wole \rel$ with~$b \Dec{\rel} a$. Therefore, the weak order is graded by~$\rel \mapsto |\Dec{\rel}| - |\Inc{\rel}|$. To check that it is a lattice, consider~$\rel, \rel[S] \in \IRel(n)$. Observe first that ${\rel[R]} \meetR {\rel[S]}$ is indeed below both~$\rel$ and~$\rel[S]$ in weak order. Moreover, if~$\rel[T] \wole \rel$ and~$\rel[T] \wole \rel[S]$, then~$\Inc{\rel[T]} \supseteq \Inc{\rel} \cup \Inc{\rel[S]}$ and~$\Dec{\rel[T]} \subseteq \Dec{\rel} \cap \Dec{\rel[S]}$, so that~$\rel[T] \wole {\rel[R]} \meetR {\rel[S]}$. This proves that~${\rel[R]} \meetR {\rel[S]}$ is indeed the meet of~$\rel$ and~$\rel[S]$. The proof is similar for the join.
\end{proof}

\begin{remark}
\label{remark:reverse}
Define the \defn{reverse} of a relation~$\rel \in \IRel$ as~$\rev{\rel} \eqdef \set{(u,v) \in [n]^2}{(v,u) \in \rel}$. Observe that ${\Inc{(\rev{\rel})} = \rev{(\Dec{\rel})}}$ and~$\Dec{(\rev{\rel})} = \rev{(\Inc{\rel})}$. Therefore, the reverse map~$\rel \mapsto \rev{\rel}$ defines an antiautomorphism of the weak order~$(\IRel(n), \wole, \meetR, \joinR)$. Note that it preserves symmetry, antisymmetry and transitivity.
\end{remark}


\section{The weak order on integer posets}
\label{sec:weakOrderIntegerPosets}

In this section, we show that the three subposets of the weak order~$(\IRel(n), \wole)$ induced by antisymmetric relations, by transitive relations, and by posets are all lattices (although the last two are not sublattices of~$(\IRel(n), \wole, \meetR, \joinR)$).


\subsection{Antisymmetric relations}
\label{subsec:antisymmetricRelations}

We first treat the case of antisymmetric relations. \fref{fig:meetAS} shows the meet and join of two antisymmetric relations, and illustrates the following statement.

\begin{proposition}
\label{prop:antisymmetricLAttice}
The meet~$\meetR$ and the join~$\joinR$ both preserve antisymmetry. Thus, the antisymmetric relations induce a sublattice~$(\IAntisym(n), \wole, \meetR, \joinR)$ of the weak order~$(\IRel(n), \wole, \meetR, \joinR)$.
\end{proposition}

\begin{proof}
Let~$\rel, \rel[S] \in \IAntisym(n)$. Let~$a < b \in [n]$ be such that~$(b,a) \in {\rel[R]} \meetR {\rel[S]}$. Since~$(b,a)$ is decreasing and~$\Dec{({\rel[R]} \meetR {\rel[S]})} = \Dec{\rel} \cap \Dec{\rel[S]}$, we have~$b \Dec{\rel} a$ and~$b \Dec{\rel[S]} a$. By antisymmetry of~$\rel$ and~$\rel[S]$, we obtain that~$a \Inc{\notrel} b$ and $a \Inc{\notrel[S]} b$. Therefore, $(a,b) \notin \Inc{\rel} \cup \Inc{\rel[S]} = \Inc{({\rel[R]} \meetR {\rel[S]})}$. We conclude that~$(b,a) \in {\rel[R]} \meetR {\rel[S]}$ implies~$(a,b) \notin {\rel[R]} \meetR {\rel[S]}$ and thus that~${\rel[R]} \meetR {\rel[S]}$ is antisymetric. The proof is identical for~$\joinR$.
\end{proof}

\begin{figure}[ht]
	\centerline{
    \begin{tabular}{c@{\qquad}c@{\qquad}c@{\qquad}c}
        $\rel[R] \in \IAntisym(4)$ &
        $\rel[S] \in \IAntisym(4)$ &
        $\rel \meetR \rel[S] \in \IAntisym(4)$ &
        $\rel \joinR \rel[S] \in \IAntisym(4)$
        \\[-.3cm]
        \scalebox{0.8}{\input{relations/AS1.tex}} &
        \scalebox{0.8}{\input{relations/AS2.tex}} &
        \scalebox{0.8}{\input{relations/AS3.tex}} &
        \scalebox{0.8}{\input{relations/AS4.tex}}
    \end{tabular}
}
	\vspace{-.5cm}
	\caption{Two antisymmetric relations~$\rel, \rel[S]$ and their meet~${\rel[R]} \meetR {\rel[S]}$ and join~${\rel[R]} \joinR {\rel[S]}$.}
	\label{fig:meetAS}
	\vspace{-.4cm}
\end{figure}

Our next two statements describe all cover relations in~$(\IAntisym(n), \wole)$.

\begin{proposition}
\label{prop:coverRelationsAntisymmetric}
All cover relations in~$(\IAntisym(n), \wole)$ are cover relations in~$(\IRel(n), \wole)$. In particular,~$(\IAntisym(n), \wole)$ is still graded by~$\rel \mapsto |\Dec{\rel}|-|\Inc{\rel}|$.
\end{proposition}

\begin{proof}
Consider a cover relation~$\rel \wole \rel[S]$ in~$(\IAntisym(n), \wole)$. We have~$\Inc{\rel} \supseteq \Inc{\rel[S]}$ and~$\Dec{\rel} \subseteq \Dec{\rel[S]}$ where at least one of the inclusions is strict. Suppose first that~$\Inc{\rel} \ne \Inc{\rel[S]}$. Let~${(a,b) \in \Inc{\rel} \ssm \Inc{\rel[S]}}$ and~$\rel[T] \eqdef \rel \ssm \{(a,b)\}$. Note that~$\rel[T]$ is still antisymmetric as it is obtained by removing an arc from an antisymmetric relation. Moreover, we have~$\rel \ne \rel[T]$ and~$\rel \wole \rel[T] \wole \rel[S]$. Since~$\rel[S]$ covers~$\rel$, this implies that~$\rel[S] = \rel[T] = \rel \ssm \{(a,b)\}$. We prove similarly that if~$\Dec{\rel} \ne \Dec{\rel[S]}$, there exists~$a < b$ such that~$\rel[S] = \rel \cup \{(b,a)\}$. In both cases, $\rel \wole \rel[S]$ is a cover relation in~$(\IRel(n), \wole)$.
\end{proof}

\begin{corollary}
\label{coro:coverRelationsAntisymmetric}
In the weak order~$(\IAntisym(n), \wole)$, the antisymmetric relations that cover a given antisymmetric relation~$\rel \in \IAntisym(n)$ are precisely the relations
\begin{itemize}
\item $\rel \ssm \{(a,b)\}$ for~$a < b$ such that~$a \rel b$,
\item $\rel \cup \{(b,a)\}$ for~$a < b$ such that~$a \notrel b$ and $b \notrel a$.
\end{itemize}
\end{corollary}


\subsection{Transitive relations}
\label{subsec:transitiveRelations}

We now consider transitive relations. Observe first that the subposet $(\ITrans(n), \wole)$ of~$(\IRel(n), \wole)$ is not a sublattice since~$\meetR$ and~$\joinR$ do not preserve transitivity (see \eg \fref{fig:meetT}). When~$\rel$ and~$\rel[S]$ are transitive, we need to transform~${\rel[R]} \meetR {\rel[S]}$ to make it a transitive relation~${\rel[R]} \meetT {\rel[S]}$. We proceed in two steps described below.

\para{Semitransitive relations}
Before dealing with transitive relations, we introduce the intermediate notion of semitransitivity. We say that a relation~$\rel \in \IRel$ is \defn{semitransitive} when both~$\Inc{\rel}$ and~$\Dec{\rel}$ are transitive.
We denote by~$\ISemiTrans(n)$ the collection of all semitransitive relations of size~$n$.
\fref{fig:meetST} illustrates the following statement.

\begin{figure}[t]
	\centerline{
    \begin{tabular}{c@{\qquad}c@{\qquad}c@{\qquad}c}
        $\rel \in \ISemiTrans(4)$ &
        $\rel[S] \in \ISemiTrans(4)$ &
        $\rel \meetR \rel[S] \notin \ISemiTrans(4)$ &
        $\rel \meetST \rel[S] \in \ISemiTrans(4)$
        \\
        \scalebox{0.8}{\input{relations/ST1.tex}} &
        \scalebox{0.8}{\input{relations/ST2.tex}} &
        \scalebox{0.8}{\input{relations/ST3.tex}} & 
        \scalebox{0.8}{\input{relations/ST4.tex}}
    \end{tabular}
}
    \vspace{-.3cm}
	\caption{Two semi-transitive relations~$\rel[R], \rel[S]$ and their meets~${\rel[R]} \meetR {\rel[S]}$ and ${\rel[R]} \meetST {\rel[S]}$.}
	\label{fig:meetST}
	\vspace{-.4cm}
\end{figure}

\begin{proposition}
\label{prop:semitransitiveLattice}
The weak order~$(\ISemiTrans(n), \wole)$ is a lattice whose meet and join are given by
\[
{\rel[R]} \meetST {\rel[S]} = \tc{( \Inc{\rel} \cup \Inc{\rel[S]} )} \cup ( \Dec{\rel} \cap \Dec{\rel[S]} )
\quad\text{and}\quad
{\rel[R]} \joinST {\rel[S]} = ( \Inc{\rel} \cap \Inc{\rel[S]} ) \cup \tc{( \Dec{\rel} \cup \Dec{\rel[S]} )}.
\]
\end{proposition}

\begin{proof}
Let~$\rel, \rel[S] \in \ISemiTrans(n)$. Observe first that~${\rel[R]} \meetST {\rel[S]}$ is indeed semitransitive and below both~$\rel$ and~$\rel[S]$. Moreover, if a semitransitive relation~$\rel[T]$ is such that~$\rel[T] \wole \rel$ and~$\rel[T] \wole \rel[S]$, then~$\Inc{\rel[T]} \supseteq \Inc{\rel} \cup \Inc{\rel[S]}$ and~$\Dec{\rel[T]} \subseteq \Dec{\rel} \cap \Dec{\rel[S]}$. By semitransitivity of~$\rel[T]$, we get~$\Inc{\rel[T]} \supseteq \tc{(\Inc{\rel} \cup \Inc{\rel[S]})}$, so that~$\rel[T] \wole {\rel[R]} \meetST {\rel[S]}$. This proves that~${\rel[R]} \meetST {\rel[S]}$ is indeed the meet of~$\rel$ and~$\rel[S]$. The proof is similar for the join.
\end{proof}

As in the previous section, we describe all cover relations in~$(\ISemiTrans(n), \wole)$.

\begin{proposition}
\label{prop:coverRelationsSemitransitive}
All cover relations in~$(\ISemiTrans(n), \wole)$ are cover relations in~$(\IRel(n), \wole)$. In particular, $(\ISemiTrans(n), \wole)$ is still graded by~$\rel \mapsto |\Dec{\rel}|-|\Inc{\rel}|$.
\end{proposition}

\begin{proof}
Consider a cover relation~$\rel \wole \rel[S]$ in~$(\ISemiTrans(n), \wole)$. We have~$\Inc{\rel} \supseteq \Inc{\rel[S]}$ and~$\Dec{\rel} \subseteq \Dec{\rel[S]}$ where at least one of the inclusions is strict. Suppose first that~$\Inc{\rel} \ne \Inc{\rel[S]}$. Let~$(a,b) \in \Inc{\rel} \ssm \Inc{\rel[S]}$ be such that~$b-a$ is minimal, and let~$\rel[T] \eqdef \rel \ssm \{(a,b)\}$. Observe that there is no~$a < i < b$ such that~$a \rel i \rel b$. Otherwise, by minimality of~$b-a$, we would have~$a \rel[S] i$ and~$i \rel[S] b$ while~$a \notrel[S] b$, contradicting the transitivity of~$\Inc{\rel[S]}$. It follows that~$\Inc{\rel[T]}$ is still transitive. Since~$\Dec{\rel[T]} =\Dec{\rel}$ is also transitive, we obtain that~$\rel[T]$ is semitransitive. Moreover, we have~$\rel \ne \rel[T]$ and~$\rel \wole \rel[T] \wole \rel[S]$. Since~$\rel[S]$ covers~$\rel$, this implies that~$\rel[S] = \rel[T] = \rel \ssm \{(a,b)\}$. We prove similarly that if~$\Dec{\rel} \ne \Dec{\rel[S]}$, there exists~$(b,a)$ such that~$\rel[S] = \rel \cup \{(b,a)\}$: in this case, one needs to pick $(b,a) \in \Dec{\rel[S]} \ssm \Dec{\rel[R]}$ with $b-a$ maximal. In both cases, $\rel \wole \rel[S]$ is a cover relation in~$(\IRel(n), \wole)$.
\end{proof}

\begin{corollary}
\label{coro:coverRelationsSemitransitive}
In the weak order~$(\ISemiTrans(n), \wole)$, the semitransitive relations that cover a given semitransitive relation~$\rel \in \ISemiTrans(n)$ are precisely the relations
\begin{itemize}
\item $\rel \ssm \{(a,b)\}$ for~$a < b$ such that~$a \rel b$ and there is no~$a < i < b$ with~$a \rel i \rel b$,
\item $\rel \cup \{(b,a)\}$ for~$a < b$ such that~$b \notrel a$ and there is no~$i < a$ with~$a \rel i$~but~$b \notrel i$ and similarly no~$b < j$ with~$j \rel b$ but~$j \notrel a$.
\end{itemize}
\end{corollary}


\para{Transitive relations}
Now we consider transitive relations.

Note that~$\ITrans(n) \subseteq \ISemiTrans(n)$ but $\ISemiTrans(n) \not\subseteq \ITrans(n)$. In particular, ${\rel[R]} \meetST {\rel[S]}$ and ${\rel[R]} \joinST {\rel[S]}$ may not be transitive even if~$\rel$ and~$\rel[S]$ are (see \fref{fig:meetT}).
To see that the subposet of the weak order induced by transitive relations is indeed a lattice, we therefore need operations which ensure transitivity and are compatible with the weak order. For~$\rel \in \IRel$, define the \defn{transitive decreasing deletion} of~$\rel$ as
\[
\tdd{\rel} \eqdef {\rel} \ssm \bigset{(b,a) \in \Dec{\rel}}{\exists \; i \le b \text{ and } j \ge a \text{ such that } i \rel b \rel a \rel j \text{ while } i \notrel j},
\]
and the \defn{transitive increasing deletion} of~$\rel$ as
\[
\tid{\rel} \eqdef {\rel} \ssm \bigset{(a,b) \in \Inc{\rel}}{\exists \; i \ge a \text{ and } j \le b \text{ such that } i \rel a \rel b \rel j \text{ while } i \notrel j}.
\]
Note that in these definitions, $i$ and~$j$ may coincide with~$a$ and~$b$ (since we assumed that all our relations are reflexive). \fref{fig:meetT} illustrates the transitive decreasing deletion: the rightmost relation~${\rel[R]} \meetT {\rel[S]}$ is indeed obtained as~$\tdd{\left({\rel[R]} \meetST {\rel[S]}\right)}$. Observe that two decreasing relations have been deleted: $(3,1)$ (take~$i=2$ and~$j=1$, or~$i=3$ and~$j=2$) and $(4,1)$ (take $i=4$ and $j=2$).  

\begin{remark}
\label{rem:tdd}
The idea of the transitive decreasing deletion is to delete all decreasing relations which prevent the binary relation to be transitive. It may thus seem more natural to assume in the definition of~$\tdd{\rel}$ that either~$i = b$ or~$j = a$. However, this would not suffice to rule out all non-transitive relations, consider for example the relation~$[4]^2 \ssm \{(2,3),(3,2)\}$. We would therefore need to iterate the deletion process, which would require to prove a converging property. Our definition of~$\tdd{\rel}$ simplifies the presentation as it requires only one deletion step.
\end{remark}

\begin{lemma}
\label{lem:transDelOrder}
For any relation~$\rel \in \IRel$, we have $\tdd{\rel} \wole \rel \wole \tid{\rel}$.
\end{lemma}

\begin{proof}
$\tdd{\rel}$ is obtained from~$\rel$ by deleting decreasing relations. Therefore~$\Inc{(\tdd{\rel})} = \Inc{\rel}$ and ${\Dec{(\tdd{\rel})} \subseteq \Dec{\rel}}$ and thus~$\tdd{\rel} \wole \rel$ by definition of the weak order. The argument is similar for~$\tid{\rel}$.
\end{proof}

\begin{lemma}
\label{lem:transDelTransitive}
If~$\rel \in \IRel$ is semitransitive, then~$\tdd{\rel}$ and~$\tid{\rel}$ are transitive.
\end{lemma}

\begin{proof}
We prove the result for~$\tdd{\rel}$, the proof being symmetric for~$\tid{\rel}$. Set
\[
\rel[U] \eqdef \set{(b,a) \in \Dec{\rel}}{\exists \; i \le b \text{ and } j \ge a \text{ such that } i \rel b \rel a \rel j \text{ while } i \notrel j},
\]
so that~$\tdd{\rel} = \rel \ssm \rel[U]$ with~$\Inc{(\tdd{\rel})} = \Inc{\rel}$ and~$\Dec{(\tdd{\rel})} = \Dec{\rel} \ssm \rel[U]$.
Let~$u,v,w \in [n]$ be such that $u \tdd{\rel} v$ and~$v \tdd{\rel} w$. We want to prove that~$u \tdd{\rel} w$. We distinguish six cases according to the relative order of~$u,v,w$:
\begin{enumerate}[(i)]
\item If~$u < v < w$, then~$u \Inc{\rel} v$ and~$v \Inc{\rel} w$. Thus $u \Inc{\rel} w$ by transitivity~of~$\Inc{\rel}$, and thus~$u \tdd{\rel} w$.
\item If~$u < w < v$, then~$u \Inc{\rel} v$ and~$v \Dec{\rel} w$. Since~$v \notrel[U] w$, we have~$u \Inc{\rel} w$ and thus~$u \tdd{\rel} w$.
\item If~$v < u < w$, then~$u \Dec{\rel} v$ and~$v \Inc{\rel} w$. Since~$u \notrel[U] v$, we have~$u \Inc{\rel} w$ and thus~$u \tdd{\rel} w$.
\item If~$v < w < u$, then~$u \Dec{\rel} v$ and~$v \Inc{\rel} w$. Since~$u \notrel[U] v$, we have~$u \Dec{\rel} w$. Assume by contradiction that~$u \rel[U] w$. Hence there is~$i \le u$ and~$j \ge w$ such that~$i \rel u \rel w \rel j$ but~$i \notrel j$. Since~$v \Inc{\rel} w$ and~$w \Inc{\rel} j$, the transitivity of~$\Inc{\rel}$ ensures that~$v \rel j$. We obtain that~$u \rel[U] v$, a contradiction. Therefore, $u \notrel[U] w$ and~$u \tdd{\rel} w$.
\item If~$w < u < v$, then~$u \Inc{\rel} v$ and~$v \Dec{\rel} w$. Since~$v \notrel[U] w$, we have~$u \Dec{\rel} w$. Assume by contradiction that~$u \rel[U] w$. Hence there is~$i \le u$ and~$j \ge w$ such that~$i \rel u \rel w \rel j$ but~$i \notrel j$. Since~$i \Inc{\rel} u$ and~$u \Inc{\rel} v$, the transitivity of~$\Inc{\rel}$ ensures that~$i \rel v$. We obtain that~$v \rel[U] w$, a contradiction. Therefore, $u \notrel[U] w$ and~$u \tdd{\rel} w$.
\item If~$w < v < u$, then~$u \Dec{\rel} v$ and~$v \Dec{\rel} w$, so that~$u \Dec{\rel} w$ by transitivity of~$\Dec{\rel}$. Assume by contradiction that~$u \rel[U] w$. Hence there is~$i \le u$ and~$j \ge w$ such that~$i \rel u \rel w \rel j$ but~$i \notrel j$. Since~$u \notrel[U] v$ and~$v \notrel[U] w$, we obtain that~$i \rel v$ and~$v \rel j$. If~$i \le v$, then we have~$i \le v$ and $j \ge w$ with~$i \rel v \rel w \rel j$ and~$i \notrel j$ contradicting the fact that~$v \notrel[U] w$. Similarly, if~$j \ge v$, we have~$i \le u$ and~$j \ge v$ with~$i \rel u \rel v \rel j$ and~$i \notrel j$ contradicting the fact that~$u \notrel[U] v$. Finally, if~$j < v < i$, we have~$i \Dec{\rel} v \Dec{\rel} j$ and~$i \Dec{\notrel} j$ contradicting the transitivity of~$\Dec{\rel}$. \qedhere
\end{enumerate}
\end{proof}

\begin{remark}
We observed earlier that the transitive closure~$\tc{\rel}$ is the coarsest transitive relation containing~$\rel$. For~$\rel \in \ISemiTrans$, Lemmas~\ref{lem:transDelOrder} and~\ref{lem:transDelTransitive} show that~$\tdd{\rel}$ is a transitive relation below~$\rel$ in weak order. However, there might be other transitive relations~$\rel[S]$ with~$\rel[S] \wole \rel$ and which are not comparable to~$\tdd{\rel}$ in weak order. For example, consider~$\rel \eqdef \{(1,3), (3,2)\}$ and~$\rel[S] \eqdef \{(1,2),(1,3),(3,2)\}$. Then~$\rel[S]$ is transitive and~$\rel[S] \wole \rel$ while~$\rel[S]$ is incomparable to~$\tdd{\rel} = \{(1,3)\}$ in weak order.
\end{remark}

We use the maps~$\rel \mapsto \tdd{\rel}$ and~$\rel \mapsto \tid{\rel}$ to obtain the main result of this section. \fref{fig:meetT} illustrates all steps of a meet computation in~$\ITrans(4)$.

\begin{figure}[t]
	\centerline{
    \begin{tabular}{c@{\qquad}c@{\qquad}c@{\quad}c@{\quad}c}
        $\rel[R] \in \ITrans(4)$ &
        $\rel[S] \in \ITrans(4)$ &
        $\rel[R] \meetR \rel[S] \notin \ISemiTrans(4)$ &
        $\rel[R] \meetST \rel[S] \in \ISemiTrans(4) \setminus \ITrans(4)$ &
        $\rel[R] \meetT \rel[S] \in \ITrans(4)$
        \\[-.3cm]
        \scalebox{0.8}{\input{relations/poset1.tex}} &
        \scalebox{0.8}{\input{relations/poset2.tex}} &
        \scalebox{0.8}{\input{relations/R1.tex}} & 
        \scalebox{0.8}{\input{relations/ST.tex}} &
        \scalebox{0.8}{\input{relations/poset3.tex}} 
    \end{tabular}
}
    \vspace{-.3cm}
	\caption{Two transitive relations~$\rel[R], \rel[S]$ and their meets~${\rel[R]} \meetR {\rel[S]}$, ${\rel[R]} \meetST {\rel[S]}$ and~${\rel[R]} \meetT {\rel[S]}$.}
	\label{fig:meetT}
	\vspace{-.4cm}
\end{figure}

\begin{proposition}
\label{prop:transitiveLattice}
The weak order~$(\ITrans(n), \wole)$ is a lattice whose meet and join are given by
\[
{\rel[R]} \meetT {\rel[S]} = \tdd{\big( \tc{(\Inc{\rel} \cup \Inc{\rel[S]})} \cup (\Dec{\rel} \cap \Dec{\rel[S]}) \big)}
\quad\text{and}\quad
{\rel[R]} \joinT {\rel[S]} = \tid{\big( (\Inc{\rel} \cap \Inc{\rel[S]}) \cup \tc{(\Dec{\rel} \cup \Dec{\rel[S]})} \big)}.
\]
\end{proposition}

Before proving the proposition, we state the following technical lemma which we will used repeatedly in our proofs.

\begin{lemma}
\label{lem:simplifytdd}
Let~$\rel[R]$ and~$\rel[S]$ be two transitive relations, let~$\rel[M] = \rel[R] \meetST \rel[S]$, and let~$1 \le a < b \le n$ such that~$b \rel[M] a$ and $b \notrel[\tdd{M}] a$. By definition of $\tdd{\rel[M]}$, there exist~$i \le b$ and~$j \ge a$ such that~$i \rel[M] b \rel[M] a \rel[M] j$ while $i \notrel[M] j$. Then we have
\begin{itemize}
\item either $i \ne b$ or $j \ne a$,
\item if $i \ne b$, there is~$a < k < b$ such that $i \rel[M] k \rel[M] b$ with $(k,b) \in \rel[R] \cup \rel[S]$ and $k \tdd{\notrel[M]} a$,
\item if $j \ne a$, there is~$a < k < b$ such that $a \rel[M] k \rel[M] j$ with $(a,k) \in \rel[R] \cup \rel[S]$ and $b \tdd{\notrel[M]} k$,
\end{itemize}
Besides if $\rel$ and $\rel[S]$ are also antisymmetric, then in both cases, $b \tdd{\notrel[M]} k \tdd{\notrel[M]} a$.
\end{lemma}

\begin{proof}
Since~$b \rel[M] a$ and $i \notrel[M] j$, we cannot have both~$i = b$ and~$j = a$. By symmetry, we can assume that~$i \ne b$. Since~$(i,b) \in \Inc{\rel[M]} = \tc{(\Inc{\rel} \cup \Inc{\rel[S]})}$, there exists~$i \le k < b$ such that~$(i,k) \in \Inc{\rel[M]}$ and~${(k,b) \in \Inc{\rel} \cup \Inc{\rel[S]}}$. Assume without loss of generality that~$k \Inc{\rel} b$. We obtain that~$k \rel b \rel a$ and thus that~$k \rel a$ by transitivity of~$\rel$. We want to prove that $k > a$.

Assume that~$k \le a$. We then have~$(k,a) \in \Inc{\rel} \subseteq \Inc{\rel[M]}$ and thus that~${i \Inc{\rel[M]} k \Inc{\rel[M]} a \Inc{\rel[M]} j}$ while~$i \Inc{\notrel[M]} j$ contradicting the transitivity of~$\Inc{\rel[M]}$. We then have $a < k < b$.  There is left to prove that $k \tdd{\notrel[M]} a$. Suppose that we have $k \tdd{\rel[M]} a$, then we have $i \rel[M] k \rel[M] a \rel[M] j$ which implies $i \rel[M] j$ because $(k,a)$ is not deleted by the transitive decreasing deletion. This contradicts our initial statement $i \notrel[M] j$.

Besides, if $\rel$ is antisymmetric, then $k \rel b$ implies $b \notrel k$ which in turns gives $b \tdd{\notrel[M]} k$.
\end{proof}

\begin{proof}[Proof of Proposition~\ref{prop:transitiveLattice}]
The weak order~$(\ITrans(n), \wole)$ is a subposet of~$(\IRel(n), \wole)$. It is also clearly bounded: the weak order minimal transitive relation is~$\rel[I]_n = \set{(a,b) \in [n]^2}{a \le b}$ while the weak order maximal transitive relation is~$\rel[D]_n = \set{(b,a) \in [n]^2}{a \le b}$. Therefore, we only have to show that any two transitive relations admit a meet and a join. We prove the result for the meet, the proof for the join being symmetric.

Let~$\rel, \rel[S] \in \ITrans(n)$ and~$\rel[M] \eqdef {\rel[R]} \meetST {\rel[S]} = \tc{(\Inc{\rel} \cup \Inc{\rel[S]})} \cup ( \Dec{\rel} \cap \Dec{\rel[S]})$, so that~${\rel[R]} \meetT {\rel[S]} = \tdd{\rel[M]}$. First we have~$\rel[M] \wole \rel$ so that~${\rel[R]} \meetT {\rel[S]} = \tdd{\rel[M]} \wole \rel[M] \wole \rel$ by Lemma~\ref{lem:transDelOrder}. Similarly, ${\rel[R]} \meetT {\rel[S]} \wole \rel[S]$. Moreover, ${\rel[R]} \meetT {\rel[S]}$ is transitive by Lemma~\ref{lem:transDelTransitive}. It thus remains to show that~${\rel[R]} \meetT {\rel[S]}$ is larger than any other transitive relation smaller than both~$\rel$ and~$\rel[S]$.

Consider thus another transitive relation~$\rel[T] \in \ITrans(n)$ such that~$\rel[T] \wole \rel$ and~$\rel[T] \wole \rel[S]$. We need to show that~$\rel[T] \wole {\rel[R]} \meetT {\rel[S]} = \tdd{\rel[M]}$. Observe that~$\rel[T] \wole \rel[M]$ since~$\rel[T]$ is semitransitive and~$\rel[M] = {\rel[R]} \meetST {\rel[S]}$ is larger than any semitransitive relation smaller than both~$\rel$ and~$\rel[S]$. It implies in particular that~$\Inc{\rel[T]} \supseteq \Inc{\rel[M]} = \Inc{(\tdd{\rel[M]})}$ and that~$\Dec{\rel[T]} \subseteq \Dec{\rel[M]}$.

Assume by contradiction that~$\rel[T] \not\wole \tdd{\rel[M]}$. Since~$\Inc{\rel[T]} \supseteq \Inc{(\tdd{\rel[M]})}$, this means that there exist~$(b,a) \in \Dec{\rel[T]} \ssm \tdd{\rel[M]}$. We choose $(b,a) \in \Dec{\rel[T]} \ssm \tdd{\rel[M]}$ such that~$b-a$ is minimal. Since~${\Dec{\rel[T]} \subseteq \Dec{\rel[M]}}$, we have~$(b,a) \in \Dec{\rel[M]} \ssm \tdd{\rel[M]}$. By definition of~$\tdd{\rel[M]}$, there exists~$i \le b$ and $j \ge a$ such that~$i \rel[M] b \rel[M] a \rel[M] j$ while~$i \notrel[M] j$. We use Lemma~\ref{lem:simplifytdd} and assume without loss of generality that there is $a < k < b$ with $(k,b) \in \rel \cup \rel[S]$ and $k \tdd{\notrel[M]} a$. Since~$(k,b) \in ( \Inc{\rel} \cup \Inc{\rel[S]} ) \subseteq \Inc{\rel[T]}$ and~$b \rel[T] a$ we have~$k \rel[T] a$ by transitivity of~$\rel[T]$. Since~$k > a$, we get~$(k,a) \in \Dec{\rel[T]} \subseteq \Dec{\rel[M]}$. But by Lemma~\ref{lem:simplifytdd}, $(k,a) \notin \tdd{\rel[M]}$: it has been deleted by the transitive decreasing deletion and thus contradicts the minimality of $b - a$.
\end{proof}

\begin{remark}
In contrast to Propositions~\ref{prop:coverRelationsAntisymmetric} and~\ref{prop:coverRelationsSemitransitive} and Corollaries~\ref{coro:coverRelationsAntisymmetric} and~\ref{coro:coverRelationsSemitransitive}, the cover relations in~$(\ITrans(n), \wole)$ are more complicated to describe. In fact, the lattice~$(\ITrans(n), \wole)$ is not graded as soon as~$n \ge 3$. Indeed, consider the maximal chains from~$\rel[I]_3$ to~$\rel[D]_3$ in~$(\ITrans(3), \wole)$. Those chains passing through the trivial reflexive relation~$\set{(i,i)}{i \in [n]}$ all have length~$6$, while those passing through the full relation~$[3]^2$ all have length~$4$.
\end{remark}


\subsection{Integer posets}
\label{subsec:integerPosets}

\begin{figure}[t]
	\centerline{\includegraphics[scale=.7]{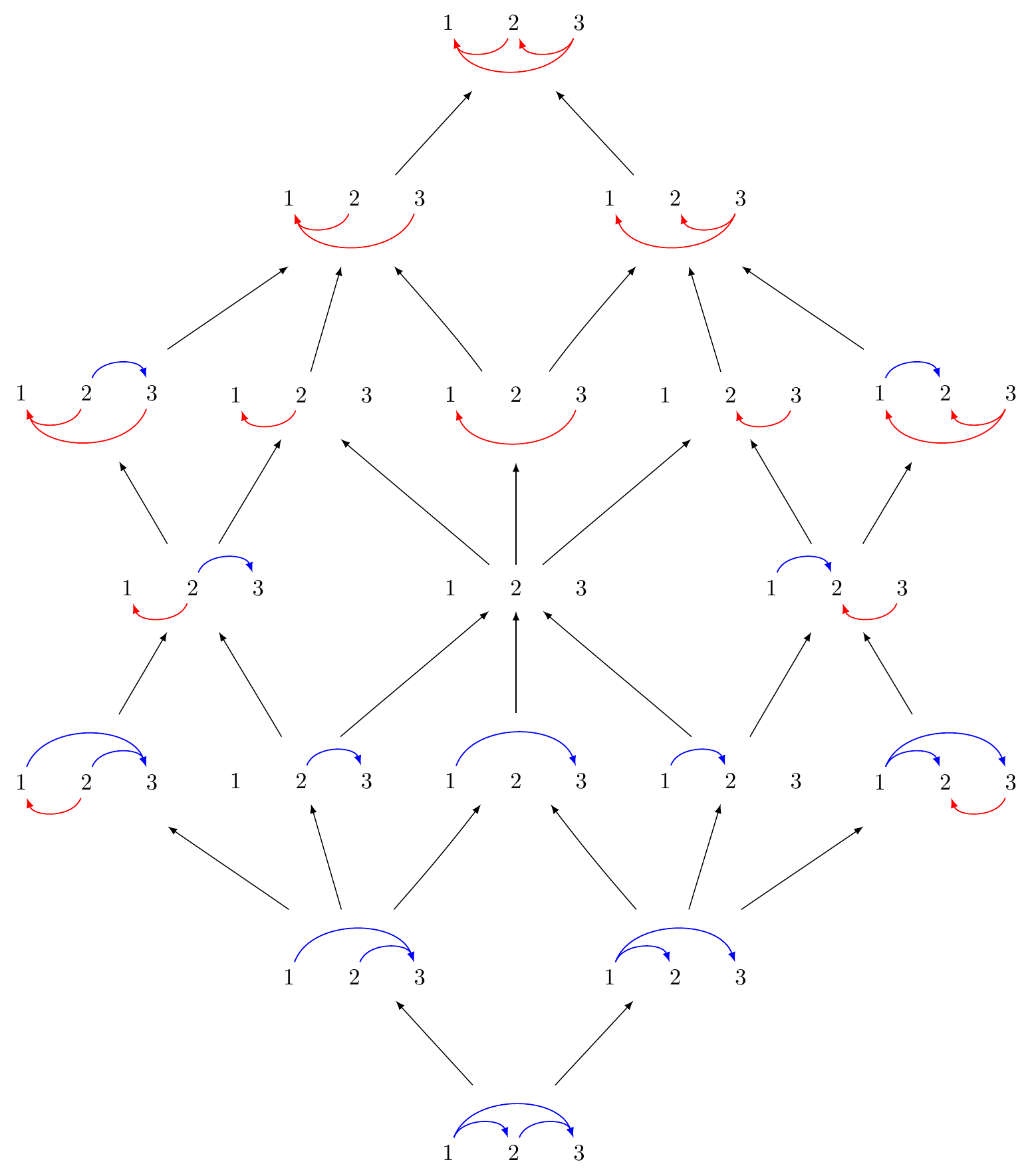}}
	\vspace{-.3cm}
	\caption{The weak order on integer posets of size~$3$.}
	\label{fig:weakOrderPosets}
	\vspace{-.4cm}
\end{figure}

We finally arrive to the subposet of the weak order induced by integer posets. The weak order on~$\IPos(3)$ is illustrated in \fref{fig:weakOrderPosets}. We now have all tools to show Theorem~\ref{thm:main} announced in the introduction.

\begin{proposition}
\label{prop:posetLattice}
The transitive meet~$\meetT$ and the transitive join~$\joinT$ both preserve antisymmetry. In other words, $(\IPos(n), \wole, \meetT, \joinT)$ is a sublattice of~$(\ITrans(n), \wole, \meetT, \joinT)$.
\end{proposition}

\begin{proof}
Let~$\rel, \rel[S] \in \IPos(n)$. Let~$\rel[M] \eqdef \rel \meetST \rel[S] = \tc{(\Inc{\rel} \cup \Inc{\rel[S]})} \cup ( \Dec{\rel} \cap \Dec{\rel[S]})$, so that~${\rel[R]} \meetT {\rel[S]} = \tdd{\rel[M]}$. Assume that~$\tdd{\rel[M]}$ is not antisymmetric. Let~$a < c \in [n]$ be such that~$\{(a,c), (c,a)\} \subseteq \tdd{\rel[M]}$ with $c-a$ minimal. Since~$(c,a) \in \Dec{(\tdd{\rel[M]})} \subseteq \Dec{\rel[M]} = \Dec{\rel} \cap \Dec{\rel[S]}$, we have~$(a,c) \notin \Inc{\rel} \cup \Inc{\rel[S]}$ by antisymmetry of~$\rel$ and~$\rel[S]$. Since~$(a,c) \in \tc{(\Inc{\rel} \cup \Inc{\rel[S]})} \ssm (\Inc{\rel} \cup \Inc{\rel[S]})$, there exists~$a < b < c$ such that~${\{(a,b), (b,c)\} \subseteq \tc{(\Inc{\rel} \cup \Inc{\rel[S]})}}$. Since~$c \tdd{\rel[M]} a \tdd{\rel[M]} b$, we obtain by transitivity of~$\tdd{\rel[M]}$ that~$\{(b,c), (c,b)\} \subseteq \tdd{\rel[M]}$, contradicting the minimality of~$c-a$.
\end{proof}

\begin{remark}
In contrast, there is no guarantee that the semitransitive meet of two transitive antisymmetric relations is antisymmetric. For example in \fref{fig:meetT}, $\rel[R]$ and $\rel[S]$ are antisymmetric but ${\rel[M]} = {\rel[R]} \meetST {\rel[S]}$ is not as it contains both~$(1,3)$ and~$(3,1)$. However, the relation~$(3,1)$ is removed by the transitive decreasing delation and the result $\tdd{\rel[M]} = {\rel[R]} \meetT {\rel[S]}$ is antisymmetric. 
\end{remark}

As in Propositions~\ref{prop:coverRelationsAntisymmetric} and~\ref{prop:coverRelationsSemitransitive} and Corollaries~\ref{coro:coverRelationsAntisymmetric} and~\ref{coro:coverRelationsSemitransitive}, the next two statements describe all cover relations in~$(\IPos(n), \wole)$.

\begin{proposition}
All cover relations in~$(\IPos(n), \wole)$ are cover relations in~$(\IRel(n), \wole)$. In particular, $(\IPos(n), \wole)$ is still graded by~$\rel \mapsto |\Dec{\rel}|-|\Inc{\rel}|$.
\end{proposition}

\begin{proof}
Consider a cover relation~$\rel \wole \rel[S]$ in~$(\IPos(n), \wole)$. We have~$\Inc{\rel} \supseteq \Inc{\rel[S]}$ and~$\Dec{\rel} \subseteq \Dec{\rel[S]}$ where at least one of the inclusions is strict. Suppose first that~$\Inc{\rel} \ne \Inc{\rel[S]}$. Consider the set~$\rel[X] \eqdef \set{(a,b) \in \Inc{\rel} \ssm \Inc{\rel[S]}}{\not\!\exists \; a < i < b \text{ with } a \rel i \rel b}$. This set~$\rel[X]$ is nonempty as it contains any~$(a,b)$ in~$\Inc{\rel} \ssm \Inc{\rel[S]}$ with~$b-a$ minimal. Consider now~$(a,b) \in \rel[X]$ with~$b-a$ maximal and let~$\rel[T] \eqdef \rel \ssm \{(a,b)\}$. We claim that~$\rel[T]$ is still a poset. It is clearly still reflexive and antisymmetric. For transitivity, assume by means of contradiction that there is~$j \in [n] \ssm \{a,b\}$ such that~$a \rel j \rel b$. Since~$(a,b) \in \rel[X]$, we know that~$j < a$ or~$b < j$. As these two options are symmetric, assume for instance that~$j < a$ and choose~$j$ so that~$a-j$ is minimal. We claim that there is no~$j < i < b$ such that~$j \rel i \rel b$. Otherwise, since~$a \rel j \rel i$ and~$\rel$ is transitive, we have~$a \rel i \rel b$. Now, if~$i < a$, we have~$a \rel i \rel b$ and~$j < i < a$ contradicting the minimality of~$a-j$ in our choice of~$j$. If~$i > a$, we have~$a \rel i \rel b$ and~$a < i < b$ contradicting the fact that~$(a,b) \in \rel[X]$. Finally, if~$i = a$, we have~$a \rel j \rel a$ contradicting the antisymmetry of~$\rel$. This proves that there is no~$j < i < b$ such that~$j \rel i \rel b$. By maximality of~$b-a$ in our choice of~$(a,b)$ this implies that~$j \rel[S] b$. Since~$(a,j) \in \Dec{\rel} \subseteq \Dec{\rel[S]}$, we therefore obtain that~$a \rel[S] j \rel[S] b$ while~$a \notrel[S] b$, contradicting the transitivity of~$\rel[S]$. This proves that~$\rel[T]$ is transitive and it is thus a poset. Moreover, we have~$\rel \ne \rel[T]$ and~$\rel \wole \rel[T] \wole \rel[S]$. Since~$\rel[S]$ covers~$\rel$, this implies that~$\rel[S] = \rel[T] = \rel \ssm \{(a,b)\}$. We prove similarly that if~$\Dec{\rel} \ne \Dec{\rel[S]}$, there exists~$(b,a)$ such that~$\rel[S] = \rel \cup \{(b,a)\}$. In both cases, $\rel \wole \rel[S]$ is a cover relation in~$(\IRel(n), \wole)$.
\end{proof}

\begin{corollary}
In the weak order~$(\IPos(n), \wole)$, the posets that cover a given integer poset~$\rel \in \IPos(n)$ are precisely the posets
\begin{itemize}
\item the relations~$\rel \ssm \{(a,b)\}$ for~$a < b$ such that~$a \rel b$ and there is no~$i \in [n]$ with~$a \rel i \rel b$,
\item the relations~$\rel \cup \{(b,a)\}$ for~$a < b$ such that~$a \notrel b$ and~$b \notrel a$ and there is no~$i \ne a$ with~$a \rel i$ but~$b \notrel i$ and similarly no~$j \ne b$ with~$j \rel b$ but~$j \notrel a$.
\end{itemize}
\end{corollary}


\part{Weak order induced by some relevant families of posets}
\label{part:relevantFamilies}

In the rest of the paper, we present our motivation to study Theorem~\ref{thm:main}. We observe that many relevant combinatorial objects (for example permutations, binary trees, binary sequences, ...) can be interpreted by specific integer posets\footnote{A comment on the notations used along this section. We use different notations for the set of permutations~$\fS(n)$ and the set of corresponding posets~$\WOEP$. Although it might look like a complicated notation for a well-known object, we want our notation to clearly distinguish between the combinatorial objects and their corresponding posets.}. Moreover, the subposets of the weak order induced by these specific integer posets often correspond to classical lattice structures on these combinatorial objects (for example the classical weak order, the Tamari lattice, the boolean lattice, ...).  Table~\ref{tab:roadmap} summarizes the different combinatorial objects involved and a roadmap to their properties.

Rather than our previous notations~$\rel[R], \rel[S], \rel[M]$ used for integer binary relations, we will denote integer posets by~$\less, \bless, \dashv$ so that $a \less b$ (resp. $a \bless b$ and $a \dashv v$) means that $a$ is in relation with $b$ for the $\less$ relation. These notations emphasize the notion of \emph{order} and allow us to write~$a \more b$ for~${b \less a}$, in particular when~$a < b$. To make our presentation easier to read, we have decomposed some of our proofs into technical but straightforward claims that are proved separately~in~Appendix~\ref{sec:missingClaims}.

\hvFloat[floatPos=p, capWidth=h, capPos=r, capAngle=90, objectAngle=90, capVPos=c, objectPos=c]{table}
{
\begin{tabular}{l|l||c|c|c|c|c}
\\[-2.2cm]
\multicolumn{2}{c||}{} 	& weak order & Tamari lattice & Cambrian lattices~\cite{Reading-CambrianLattices} & boolean lattice & Permutree lattices \\
\hline\hline
\multirow{7}{*}{\raisebox{-2.9cm}{Elements}}
& \multirow{2}{*}{\raisebox{-1.9cm}{\begin{minipage}[c]{2.5cm} combinatorial \\ objects \end{minipage}}}
& \begin{minipage}[c]{2cm} \begin{center} permutations \\ $2751346$ \end{center} \end{minipage}
& binary trees
& Cambrian trees~\cite{ChatelPilaud}
& \begin{minipage}[c]{3cm} \begin{center} ~\\ binary sequences \\ ${-}{+}{+}{-}{+}{-}$ \end{center} \end{minipage}
& permutrees~\cite{PilaudPons}
\\[.3cm]
& 
& \raisebox{.1cm}{\includegraphics[scale=.7]{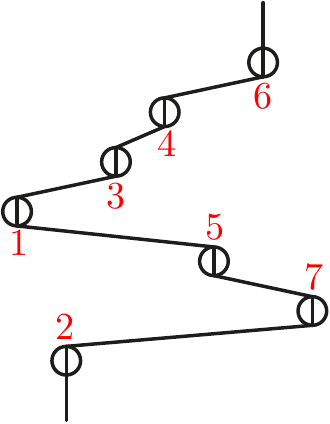}}
& \includegraphics[scale=.7]{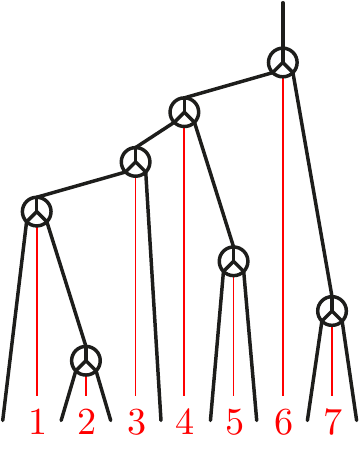}
& \includegraphics[scale=.7]{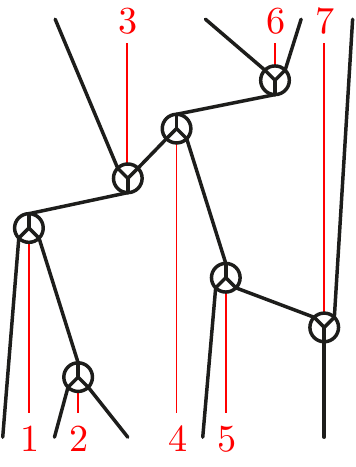}
& \includegraphics[scale=.7]{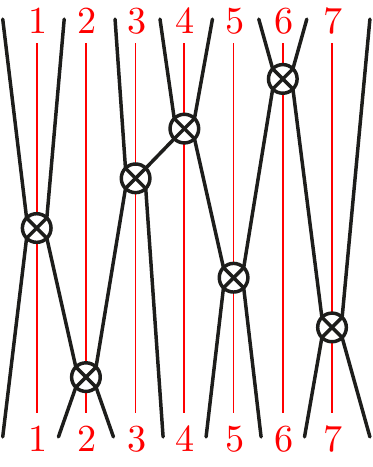}
& \includegraphics[scale=.7]{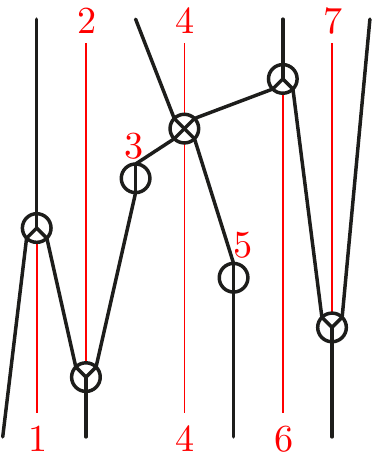}
\\
\arrayrulecolor{lightgray}\cline{2-7}\arrayrulecolor{black}
& notation
& $\WOEP(n)$
& $\TOEP(n)$
& $\COEP$
& $\BOEP(n)$
& $\PEP$
\\
\arrayrulecolor{lightgray}\cline{2-7}\arrayrulecolor{black}
& characterization
& Prop.~\ref{prop:characterizationWOEP}
& Prop.~\ref{prop:characterizationTOEP}
& Prop.~\ref{prop:characterizationPEP}
& Prop.~\ref{prop:characterizationPEP}
& Prop.~\ref{prop:characterizationPEP}
\\
\arrayrulecolor{lightgray}\cline{2-7}\arrayrulecolor{black}
& lattice properties
& Prop.~\ref{prop:WOEPsublattice} \& Coro.~\ref{coro:WOEPsublattice}
& Prop.~\ref{prop:TOEPsublattice} \& Coro.~\ref{coro:TOEPCOEPBOEPsublattices}
& Coro.~\ref{coro:TOEPCOEPBOEPsublattices} & Coro.~\ref{coro:TOEPCOEPBOEPsublattices}
& Thms.~\ref{thm:coveringOrientationElementsSublatticeIPos} \&~\ref{thm:nonCoveringOrientationElementsSublatticeWOIP}
\\
\arrayrulecolor{lightgray}\cline{2-7}\arrayrulecolor{black}
& \multirow{2}{*}{enumeration}
& $1, 2, 6, 24, 120, ...$ 
& $1, 2, 5, 14, 42, ...$
& $1, 2, 5, 14, 42, ...$
& $1, 2, 4, 8, 16, ...$
& depends on
\\
&
& \href{https://oeis.org/A000142}{\cite[A000142]{OEIS}}
& \href{https://oeis.org/A000108}{\cite[A000108]{OEIS}}
& \href{https://oeis.org/A000108}{\cite[A000108]{OEIS}}
& \href{https://oeis.org/A000079}{\cite[A000079]{OEIS}}
& the orientation~$\orientation$
\\[.2cm]
\hline
\multirow{7}{*}{\raisebox{-1.7cm}{Intervals}}
& \multirow{2}{*}{\raisebox{-1cm}{\begin{minipage}[c]{2.5cm} combinatorial \\ objects \end{minipage}}}
& \begin{minipage}[c]{3.3cm} \begin{center} ~\\ weak order intervals \\ $[213, 321]$ \end{center} \end{minipage}
& \begin{minipage}[c]{3cm} \begin{center} ~\\ Tamari lattice \\ intervals \end{center} \end{minipage}
& \begin{minipage}[c]{3cm} \begin{center} ~\\ Cambrian lattice \\ intervals \end{center} \end{minipage}
& \begin{minipage}[c]{3cm} \begin{center} ~\\ Boolean lattice \\ intervals \end{center} \end{minipage}
& \begin{minipage}[c]{3cm} \begin{center} ~\\ Permutree lattice \\ intervals \end{center} \end{minipage}
\\[.3cm]
&
& \raisebox{.15cm}{\includegraphics[scale=.7]{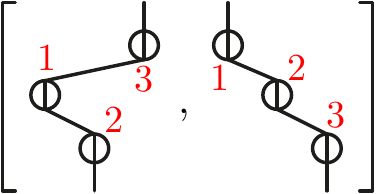}}
& \includegraphics[scale=.7]{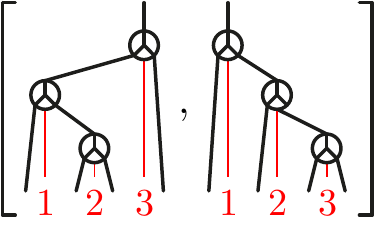}
& \includegraphics[scale=.7]{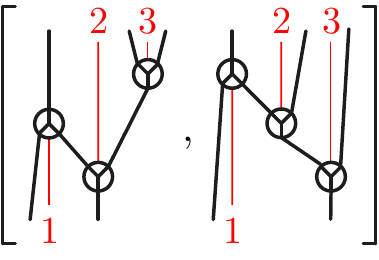}
& \includegraphics[scale=.7]{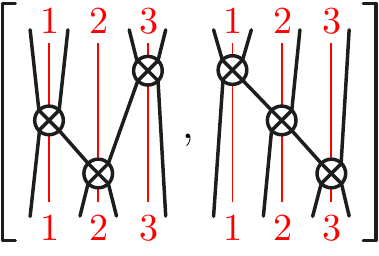}
& \includegraphics[scale=.7]{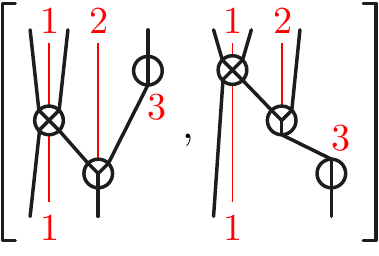}
\\
\arrayrulecolor{lightgray}\cline{2-7}\arrayrulecolor{black}
& notation
& $\WOIP(n)$
& $\TOIP(n)$
& $\COIP$
& $\BOIP(n)$
& $\PIP$
\\
\arrayrulecolor{lightgray}\cline{2-7}\arrayrulecolor{black}
& characterization 
& Prop.~\ref{prop:characterizationWOIP}
& Coro.~\ref{coro:characterizationTOIP}
& Coro.~\ref{coro:characterizationPIP}
& Coro.~\ref{coro:characterizationPIP}
& Coro.~\ref{coro:characterizationPIP}
\\
\arrayrulecolor{lightgray}\cline{2-7}\arrayrulecolor{black}
& lattice properties
& Coro.~\ref{coro:WOIPLattice}
& Coro.~\ref{coro:TOIPLattice} \& Prop.~\ref{prop:TOIPSublattice}
& Coro.~\ref{coro:PIPLattice} \& \ref{coro:TOIPCOIPBOIPsublattices}
& Coro.~\ref{coro:PIPLattice} \& \ref{coro:TOIPCOIPBOIPsublattices}
& Coro.~\ref{coro:PIPLattice}
\\
\arrayrulecolor{lightgray}\cline{2-7}\arrayrulecolor{black}
& \multirow{2}{*}{enumeration}
& $1, 3, 17, 151, 1899, ...$ 
& $1, 3, 13, 68, 399, ...$
& depends on
& $1, 3, 9, 27, 81, ...$
& depends on
\\
&
& \href{https://oeis.org/A007767}{\cite[A007767]{OEIS}}
& \href{https://oeis.org/A000260}{\cite[A000260]{OEIS}}
& the signature~$\signature$
& \href{https://oeis.org/A000244}{\cite[A000244]{OEIS}}
& the orientation~$\orientation$
\\[.2cm]
\hline
\multirow{7}{*}{\raisebox{-1.8cm}{Faces}}
& \multirow{2}{*}{\raisebox{-1.1cm}{\begin{minipage}[c]{2.5cm} combinatorial \\ objects \end{minipage}}}
& \begin{minipage}[c]{3cm} \begin{center} ~\\ ordered partitions \\ $125 | 37 | 46$ \end{center} \end{minipage}
& Schr\"oder trees
& \begin{minipage}[c]{3.5cm} \begin{center} ~\\ Schr\"oder \\ Cambrian trees~\cite{ChatelPilaud} \end{center} \end{minipage}
& \begin{minipage}[c]{3cm} \begin{center} ~\\ ternary sequences \\ ${0}{+}{+}{-}{+}{-}$ \end{center} \end{minipage}
& \begin{minipage}[c]{3cm} \begin{center} ~\\ Schr\"oder \\ permutrees~\cite{PilaudPons} \end{center} \end{minipage}
\\[.3cm]
& 
& \raisebox{.15cm}{\includegraphics[scale=.7]{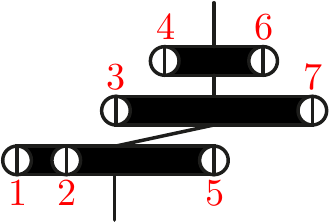}}
& \includegraphics[scale=.7]{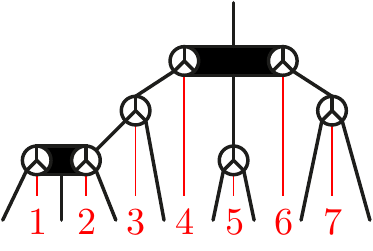}
& \includegraphics[scale=.7]{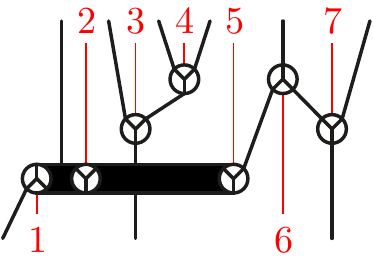}
& \includegraphics[scale=.7]{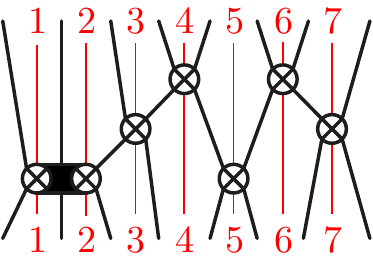}
& \includegraphics[scale=.7]{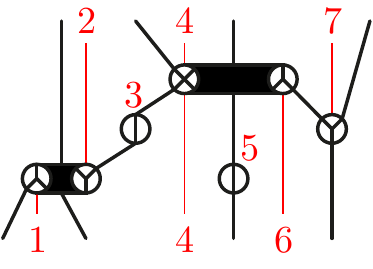}
\\
\arrayrulecolor{lightgray}\cline{2-7}\arrayrulecolor{black}
& notation
& $\WOFP(n)$
& $\TOFP(n)$
& $\COFP$
& $\BOFP(n)$
& $\PFP$
\\
\arrayrulecolor{lightgray}\cline{2-7}\arrayrulecolor{black}
& characterization
& Prop.~\ref{prop:characterizationWOFP}
& Prop.~\ref{prop:characterizationTOFP}
& Prop.~\ref{prop:characterizationPFP}
& Prop.~\ref{prop:characterizationPFP}
& Prop.~\ref{prop:characterizationPFP}
\\
\arrayrulecolor{lightgray}\cline{2-7}\arrayrulecolor{black}
& lattice properties 
& Rem.~\ref{rem:WOFPnotSublattice}
& Rem.~\ref{rem:TOFPnotSublattice}
& Rem.~\ref{rem:PFPnotSublattice}
& Coro.~\ref{coro:PIPLattice} \& \ref{coro:TOIPCOIPBOIPsublattices}
& Rem.~\ref{rem:PFPnotSublattice}
\\
\arrayrulecolor{lightgray}\cline{2-7}\arrayrulecolor{black}
& \multirow{2}{*}{enumeration}
& $1, 3, 13, 75, 541, ...$
& $1, 3, 11, 45, 197, ...$
& $1, 3, 11, 45, 197, ...$
& $1, 3, 9, 27, 81, ...$
& depends on
\\
&
& \href{https://oeis.org/A000670}{\cite[A000670]{OEIS}}
& \href{https://oeis.org/A001003}{\cite[A001003]{OEIS}}
& \href{https://oeis.org/A001003}{\cite[A001003]{OEIS}}
& \href{https://oeis.org/A000244}{\cite[A000244]{OEIS}}
& the orientation~$\orientation$
\end{tabular}
}
{A roadmap through the combinatorial objects considered in Part~\ref{part:relevantFamilies}.}
{tab:roadmap}

\section{From the permutahedron}
\label{sec:relevantFamiliesPermutahedron}

We start with relevant families of posets corresponding to the elements, the intervals, and the faces of the permutahedron. Further similar families of posets will appear in Sections~\ref{sec:relevantFamiliesAssociahedron}~and~\ref{sec:relevantFamiliesPermutreehedra}.

Let~$\fS(n)$ denote the symmetric group on~$[n]$. For~$\sigma \in \fS(n)$, we denote by
\begin{align*}
& \ver(\sigma) \eqdef \bigset{(a,b) \in [n]^2}{a \le b \text{ and } \sigma^{-1}(a) \le \sigma^{-1}(b)} \\
\text{and}\quad
& \inv(\sigma) \eqdef \bigset{(b,a) \in [n]^2}{a \le b \text{ and } \sigma^{-1}(a) \ge \sigma^{-1}(b)}
\end{align*}
the set of \defn{versions} and \defn{inversions} of~$\sigma$ respectively\footnote{Throughout the paper, we only work with versions and inversions of values (sometimes called left inversions, or coinversions). The cover relations of the weak order are thus given by transpositions of consecutive positions (sometimes called right weak order). As there is no ambiguity in the paper, we never specify this convention.}. Inversions are classical (although we order their entries in a strange way), while versions are borrowed from~\cite{KasselLascouxReutenauer}. Clearly, the versions of~$\sigma$ determine the inversions of~$\sigma$ and \viceversa.
The \defn{weak order} on~$\fS(n)$ is defined as the inclusion order of inversions, or as the clusion (reverse inclusion) order of the versions:
\[
\sigma \wole \tau \iff \inv(\sigma) \subseteq \inv(\tau) \iff \ver(\sigma) \supseteq \ver(\tau).
\]
It is known that the weak order~$(\fS(n), \wole)$ is a lattice. We denote by~$\meetWO$ and~$\joinWO$ its meet and join, and by~$e \eqdef [1,2,\dots,n]$ and~$\wo \eqdef [n,\dots,2,1]$ the weak order minimal and maximal permutations.


\subsection{Weak Order Element Posets}
\label{subsec:WOEP}

We see a permutation~$\sigma \in \fS(n)$ as a total order~$\less_\sigma$ on~$[n]$ defined by~$u \less_\sigma v$ if~$\sigma^{-1}(u) \le \sigma^{-1}(v)$ (\ie $u$ is before $v$ in $\sigma$). In other words, $\less_\sigma$ is the chain~$\sigma(1) \less_\sigma \dots \less_\sigma \sigma(n)$ as illustrated in \fref{fig:woep}.

\begin{figure}[t]
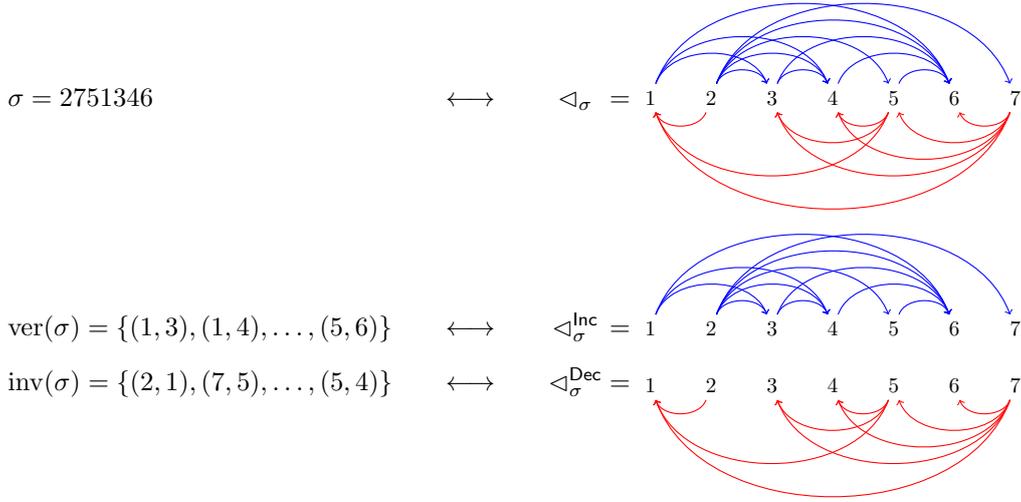

	\vspace{-.3cm}
	\centerline{
    \begin{tabular}{l@{\qquad$\longleftrightarrow$\qquad}c@{\;$=$\;}l}
        $\sigma = 2751346$ & ${\less_\sigma}$ & \raisebox{-1.8cm}{\scalebox{0.8}{\input{relations/poset2751346}}} \\[-.75cm]
        $\ver(\sigma) = \{ (1,3), (1,4), \dots, (5,6)\}$ & ${\Inc{\less_\sigma}}$ & \raisebox{-.1cm}{\scalebox{0.8}{\input{relations/poset2751346inc}}} \\[.3cm]
        $\inv(\sigma) = \{ (2,1), (7,5), \dots, (5,4) \}$ & ${\Dec{\less_\sigma}}$ & \raisebox{-1.85cm}{\scalebox{0.8}{\input{relations/poset2751346dec}}}
    \end{tabular}
}
	\vspace{-.4cm}
	\caption{A Weak Order Element Poset ($\WOEP$).}
	\label{fig:woep}
	\vspace{-.4cm}
\end{figure}

We say that~$\less_\sigma$ is a \defn{weak order element poset}, and we denote by
\[
\WOEP(n) \eqdef \bigset{{\less_\sigma}}{\sigma \in \fS(n)}
\]
the set of all total orders on~$[n]$. The following characterization of these elements is immediate.

\begin{proposition}
\label{prop:characterizationWOEP}
A poset~${\less} \in \IPos(n)$ is in~$\WOEP(n)$ if and only if~$\forall \; u,v \in [n]$, either~${u \less v}$~or~${u \more v}$.
\end{proposition}

In other words, the $\WOEP$ are the maximal posets, with~$\binom{n}{2}$ relations (this should help spotting them on \fref{fig:weakOrderPosets}).
The following proposition connects the weak order on~$\fS(n)$ to that on~$\IPos(n)$. It justifies the term ``weak order'' used in Definition~\ref{def:weakOrder}.

\begin{proposition}
\label{prop:weakOrderWOEP}
For~$\sigma \in \fS(n)$, the increasing (resp.~decreasing) relations of~$\less_\sigma$ are the versions (resp.~inversions) of~$\sigma$: ${\Inc{\less_\sigma}} = \ver(\sigma)$ and~${\Dec{\less_\sigma}} = \inv(\sigma)$. Therefore, for any permutations~${\sigma, \sigma' \in \fS(n)}$, we have~$\sigma \wole \sigma'$ if and only if~${\less_\sigma} \wole {\less_{\sigma'}}$.
\end{proposition}

\begin{proof}
${\Inc{\less_\sigma}} = \set{(a,b) }{a < b \text{ and } a \less_\sigma b} = \set{(a,b)}{a < b \text{ and } \sigma^{-1}(a) < \sigma^{-1}(b)} = \ver(\sigma)$.
\end{proof}

We thus obtain that the subposet of the weak order~$(\IPos(n), \wole)$ induced by the set~$\WOEP(n)$ is isomorphic to the weak order on~$\fS(n)$, and thus is a lattice. To conclude on $\WOEP(n)$, we mention the following stronger statement which will be derived in Corollary~\ref{coro:WOEPsublattice}.

\begin{proposition}
\label{prop:WOEPsublattice}
The set~$\WOEP(n)$ induces a sublattice of the weak order~$(\IPos(n), \wole, \meetT, \joinT)$.
\end{proposition}


\subsection{Weak Order Interval Posets}
\label{subsec:WOIP}

For two permutations~$\sigma, \sigma' \in \fS(n)$ with~$\sigma \wole \sigma'$, we denote by~$[\sigma, \sigma'] \eqdef \set{\tau \in \fS(n)}{\sigma \wole \tau \wole \sigma'}$ the weak order interval between~$\sigma$ and~$\sigma'$. As illustrated in \fref{fig:woip}, we can see such an interval as the set of linear extensions of a poset.

\begin{proposition}
\label{prop:definitionWOIP}
The permutations of~$[\sigma, \sigma']$ are precisely the linear extensions of the poset
\[
\less_{[\sigma,\sigma']} \eqdef \bigcap_{\sigma \wole \tau \wole \sigma'} {\less_\tau} = {\less_{\sigma}} \cap {\less_{\sigma'}} = {\Inc{\less_{\sigma'}}} \cup {\Dec{\less_\sigma}}.
\]
\end{proposition}

\begin{proof}
We first prove that the three expressions for~$\less_{[\sigma,\sigma']}$ coincide. Indeed we have
\[
\bigcap_{\sigma \wole \tau \wole \sigma'} {\less_\tau} = \bigg( \bigcap_{\sigma \wole \tau \wole \sigma'} {\Inc{\less_\tau}} \bigg) \cup \bigg( \bigcap_{\sigma \wole \tau \wole \sigma'} {\Dec{\less_\tau}} \bigg) = {\Inc{\less_{\sigma'}} \cup {\Dec{\less_\sigma}}} = {\less_{\sigma}} \cap {\less_{\sigma'}},
\]
where the first equality is obtained by restriction to the increasing and decreasing relations, the second equality holds since~$\sigma \wole \tau \wole \sigma' \iff {\Inc{\less_\tau}} \supseteq {\Inc{\less_{\sigma'}}}$ and~${\Dec{\less_\sigma}} \subseteq {\Dec{\less_\tau}}$ by Proposition~\ref{prop:weakOrderWOEP}, and the last one follows from~${\Inc{\less_\sigma}} \supseteq {\Inc{\less_{\sigma'}}}$ and~${\Dec{\less_\sigma}} \subseteq {\Dec{\less_{\sigma'}}}$.

Consider now a permutation~$\tau$. By definition, $\less_\tau$ extends~${\Inc{\less_{\sigma'}} \cup {\Dec{\less_\sigma}}}$ if and only if ${\Inc{\less_\tau}} \supseteq {\Inc{\less_{\sigma'}}}$ and~${\Dec{\less_\sigma}} \subseteq {\Dec{\less_\tau}}$, which in turns is equivalent to~$\sigma \wole \tau \wole \sigma'$ by Proposition~\ref{prop:weakOrderWOEP}.
\end{proof}

\begin{figure}[t]
	\vspace{-.7cm}
	\input{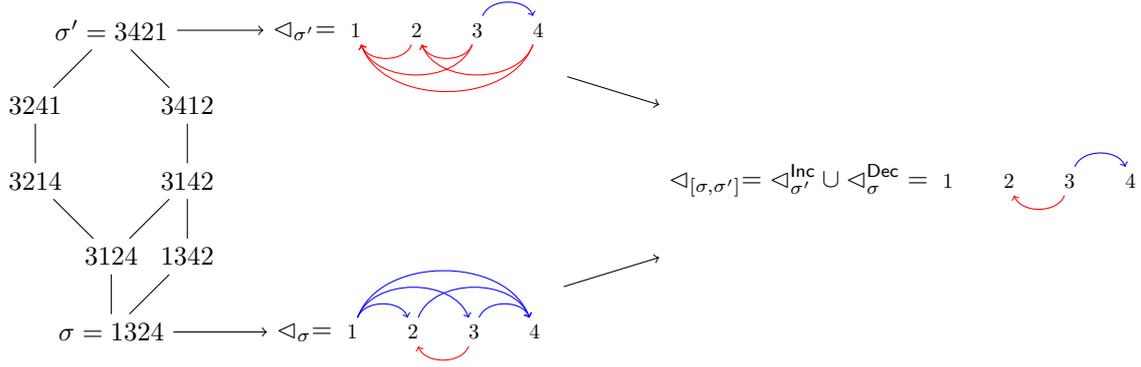}
    \vspace{-.8cm}
	\caption{A Weak Order Interval Poset ($\WOIP$).}
	\label{fig:woip}
	\vspace{-.4cm}
\end{figure}

We say that~$\less_{[\sigma,\sigma']}$ is a \defn{weak order interval poset}, and we denote by
\[
\WOIP(n) \eqdef \bigset{{\less_{[\sigma,\sigma']}}}{\sigma, \sigma' \in \fS(n), \; \sigma \wole \sigma'}
\]
the set of all weak order interval posets on~$[n]$. The following characterization of these posets already appeared in~\cite[Thm.~6.8]{BjornerWachs} and will be discussed in Section~\ref{subsec:IWOIPDWOIP}.

\begin{proposition}[\protect{\cite[Thm.~6.8]{BjornerWachs}}]
\label{prop:characterizationWOIP}
A poset~${\less} \in \IPos(n)$ is in~$\WOIP(n)$ if and only if $\forall \; a < b < c$,
\[
a \less c \implies a \less b \text{ or } b \less c
\qquad\text{and}\qquad
a \more c \implies a \more b \text{ or } b \more c.
\]
\end{proposition}

We now describe the weak order on~$\WOIP(n)$.

\begin{proposition}
\label{prop:weakOrderWOIP}
For any~$\sigma \wole \sigma'$ and~$\tau \wole \tau'$, we have ${\less_{[\sigma,\sigma']}} \wole {\less_{[\tau,\tau']}} \iff \sigma \wole \tau \text{ and } \sigma' \wole \tau'$.
\end{proposition}

\begin{proof}
From the formula of Proposition~\ref{prop:definitionWOIP}, we have
\[
\begin{array}[b]{c@{\;\iff\;}c@{\text{ and }}c}
{\less_{[\sigma,\sigma']}} \wole {\less_{[\tau,\tau']}}
& {\Inc{\less_{[\sigma,\sigma']}}} \supseteq {\Inc{\less_{[\tau,\tau']}}} & {\Dec{\less_{[\sigma,\sigma']}}} \subseteq {\Dec{\less_{[\tau,\tau']}}} \\
& {\Inc{\less_{\sigma'}}} \supseteq {\Inc{\less_{\tau'}}} & {\Dec{\less_{\sigma}}} \subseteq {\Dec{\less_{\tau}}} \\
& \sigma' \wole \tau' & \sigma \wole \tau.
\end{array}
\qedhere
\]
\end{proof}

It follows that $(\WOIP(n), \wole)$ gets the lattice structure of a product, described in the next statement. See also Corollary~\ref{coro:WOIPMeetJoin} for an alternative description of the meet and join in this lattice.

\begin{corollary}
\label{coro:WOIPLattice}
The weak order~$(\WOIP(n), \wole)$ is a lattice whose meet and join are given by
\[
{\less_{[\sigma,\sigma']}} \meetWOIP {\less_{[\tau,\tau']}} = {\less_{[\sigma \meetWO \tau, \, \sigma' \meetWO \tau']}}
\qquad\text{and}\qquad
{\less_{[\sigma,\sigma']}} \joinWOIP {\less_{[\tau,\tau']}} = {\less_{[\sigma \joinWO \tau, \, \sigma' \joinWO \tau']}}.
\]
\end{corollary}

\begin{corollary}
The set~$\WOEP(n)$ induces a sublattice of the weak order~$(\WOIP, \wole, \meetT, \joinT)$.
\end{corollary}

\begin{remark}
\label{rem:WOIPnotSublattice}
$(\WOIP(n), \wole, \meetWOIP, \joinWOIP)$ is not a sublattice of~$(\IPos(n), \wole, \meetT, \joinT)$. For example,
\vspace{-.5cm}
\begin{align*}
& {\less_{[231,321]}} \meetT {\less_{[312,321]}} = \raisebox{-.65cm}{\scalebox{.8}{\input{relations/woip_231_321}}} \meetT \raisebox{-.65cm}{\scalebox{.8}{\input{relations/woip_312_321}}} = \raisebox{-.65cm}{\scalebox{.8}{\input{relations/poset4}}} \\[-.2cm]
\text{while} \quad & {\less_{[231,321]}} \meetWOIP {\less_{[312,321]}} = {\less_{[123,321]}} =
\varnothing \text{ (trivial poset on~$[3]$)}.
\end{align*}
\end{remark}


\subsection{Weak Order Face Posets}
\label{subsec:WOFP}

The permutations of~$\fS(n)$ correspond to the vertices of the permutahedron~$\Perm \eqdef \conv\set{\big(\sigma(1), \dots, \sigma(n) \big)}{\sigma \in \fS(n)}$. We now consider all the faces of the permutahedron. The codimension~$k$ faces of~$\Perm$ correspond to \defn{ordered partitions} of~$[n]$ into~$k$ parts, or equivalently to \defn{surjections} from~$[n]$ to~$[k]$. We see an ordered partition~$\pi$ as a poset~$\less_\pi$ on~$[n]$ defined by~$u \less_\pi v$ if and only if~$u = v$ or~$\pi^{-1}(u) < \pi^{-1}(v)$, that is, the part of~$\pi$ containing~$u$ appears strictly before the part of~$\pi$ containing~$v$. See \fref{fig:wofp}. Note that a permutation~$\sigma$ belongs to the face of the permutahedron~$\Perm$ corresponding to an ordered partition~$\pi$ if and only if~$\less_\sigma$ is a linear extension of~$\less_\pi$. 

\begin{figure}[t]
	\vspace{-.45cm}
	\input{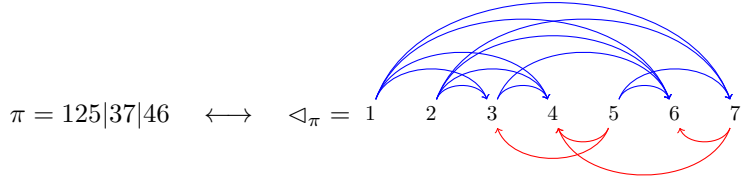}
	\vspace{-1.3cm}
	\caption{A Weak Order Face Poset ($\WOFP$).}
	\label{fig:wofp}
	\vspace{-.4cm}
\end{figure}

We say that~$\less_\pi$ is a \defn{weak order face poset}, and we denote by
\[
\WOFP(n) \eqdef \set{{\less_\pi}}{\pi \text{ ordered partition of } [n]}
\]
the set of all weak order face posets on~$[n]$. We first characterize these posets.

\begin{proposition}
\label{prop:characterizationWOFP}
The following conditions are equivalent for a poset~${\less} \in \IPos(n)$:
\begin{enumerate}[(i)]
\item ${\less} \in \WOFP(n)$,
\item $\forall \; u, v, w \in [n]$, \; $u \less w \implies u \less v \text{ or } v \less w$,
\item ${\less} \in \WOIP(n)$ and $\forall \; a < b < c$ with $a, c$ incomparable, $a \less b \iff b \more c$ and~${a \more b \iff b \less c}$.
\end{enumerate}
\end{proposition}

\begin{proof}
Assume that~${\less} = {\less_\pi} \in \WOFP(n)$ for an ordered partition~$\pi$ of~$[n]$, and let~$u,v,w \in [n]$ such that~$u \less w$. By definition, we have~${\pi^{-1}(u) < \pi^{-1}(w)}$. Therefore, we certainly have~${\pi^{-1}(u) < \pi^{-1}(v)}$ or~${\pi^{-1}(v) < \pi^{-1}(w)}$, and thus~$u \less v$ or~$v \less w$. This proves that (i) $\Longrightarrow$ (ii).

Assume now that~$\less$ satisfies~(ii). It immediately implies that~${\less} \in \WOIP(n)$ by the characterization of Proposition~\ref{prop:characterizationWOIP}. Consider now~$a < b < c$ such that~$a$ and~$c$ are incomparable in~$\less$. If~$a \less b$, then~(ii) implies that either~$a \less c$ or~$c \less b$. Since we assumed that~$a$ and~$c$ are incomparable, we obtain that~$b \more c$. We obtain similarly that~$b \more c \Longrightarrow a \less b$, that~$a \more b \Longrightarrow b \less c$ and that~$b \less c \Longrightarrow a \more b$. This shows that~(ii) $\Longrightarrow$ (iii).

Finally, assume that~$\less$ satisfies~(iii). Consider the $\less$ incomparability relation~$\equiv$ defined by~$u \equiv v$ when~$u$ and~$v$ are incomparable in~$\less$. Condition~(iii) ensures that~$\equiv$ is an equivalence relation. Moreover, the equivalence classes of~$\equiv$ are totally ordered. This shows that~$\less$ defines an ordered partition of~$[n]$ and thus that~(iii) $\Longrightarrow$ (i).
\end{proof}

We now consider the weak order on~$\WOFP(n)$. Since~$\WOFP(n) \subseteq \WOIP(n)$, Proposition~\ref{prop:weakOrderWOIP} shows that we have ${\less} \wole {\bless} \iff {\minle{\less}} \wole {\minle{\bless}} \text{ and } {\maxle{\less}} \wole {\maxle{\bless}}$. This order is precisely the \defn{facial weak order} on the permutahedron~$\Perm$ studied by A.~Dermenjian, C.~Hohlweg and V.~Pilaud in~\cite{DermenjianHohlwegPilaud}. They prove in particular that this order coincides with the \defn{pseudo-permutahedron} originally defined by D.~Krob, M.~Latapy, J.-C.~Novelli, H.-D.~Phan and S.~Schwer~\cite{KrobLatapyNovelliPhanSchwer} on ordered partitions as the transitive closure of the relations
\[
\lambda_1 | \cdots | \lambda_i | \lambda_{i+1} | \cdots | \lambda_k \; \prec \; \lambda_1 | \cdots | \lambda_i\lambda_{i+1} | \cdots | \lambda_k \; \prec \; \lambda_1 | \cdots | \lambda_{i+1} | \lambda_i | \cdots | \lambda_k,
\]
if~$\max(\lambda_i) < \min(\lambda_{i+1})$. This order is known to be a lattice~\cite{KrobLatapyNovelliPhanSchwer,DermenjianHohlwegPilaud}. We will discuss an alternative description of the meet and join in this lattice in Section~\ref{subsec:facesSublattices}.

\begin{remark}
\label{rem:WOFPnotSublattice}
$(\WOFP(n), \wole, \meetWOFP, \joinWOFP)$ is not a sublattice of~$(\IPos(n), \wole, \meetT, \joinT)$, nor a sublattice of~$(\WOIP(n), \wole, \meetWOIP, \joinWOIP)$. For example,
\[
{\less_{2|13}} \meetT {\less_{123}} = {\less_{2|13}} \meetWOIP {\less_{123}}  = \{(2,3)\}
\quad \text{while} \quad {\less_{2|13}} \meetWOFP {\less_{123}} = {\less_{12|3}} = \{(1,3), (2,3)\}.
\]
\end{remark}


\subsection{$\IWOIP(n)$ and $\DWOIP(n)$ and the $\WOIP$ deletion}
\label{subsec:IWOIPDWOIP}

We conclude our section on the permutahedron by introducing some variations on~$\WOIP(n)$ which are needed later and provide a proof of the characterization of~$\WOIP(n)$ given in Proposition~\ref{prop:characterizationWOIP}. 

Since the set of linear extensions of a poset is order-convex, a poset is in $\WOIP(n)$ if and only if it admits weak order minimal and maximal linear extensions. This motivates to consider separately two bigger families of posets. Denote by~$\IWOIP(n)$ (resp.~by~$\DWOIP(n)$) the set of posets of~$\IPos(n)$ which admit a weak order maximal (resp.~minimal) linear extension. Proposition~\ref{prop:characterizationWOIP} follows from the characterization of these posets, illustrated in~\fref{fig:IWOIP-DWOIP}.

\begin{proposition}
\label{prop:characterizationIWOIPDWOIP}
For a poset~${\less} \in \IPos(n)$,
\[
\begin{array}{l@{\quad\iff\quad}l}
{\less} \in \IWOIP(n) & \forall \; a < b < c, \;\; a \less c \implies a \less b \text{ or } b \less c, \\
{\less} \in \DWOIP(n) & \forall \; a < b < c, \;\; a \more c \implies a \more b \text{ or } b \more c.
\end{array}
\]
\end{proposition}

\begin{proof}
\begin{figure}[b]
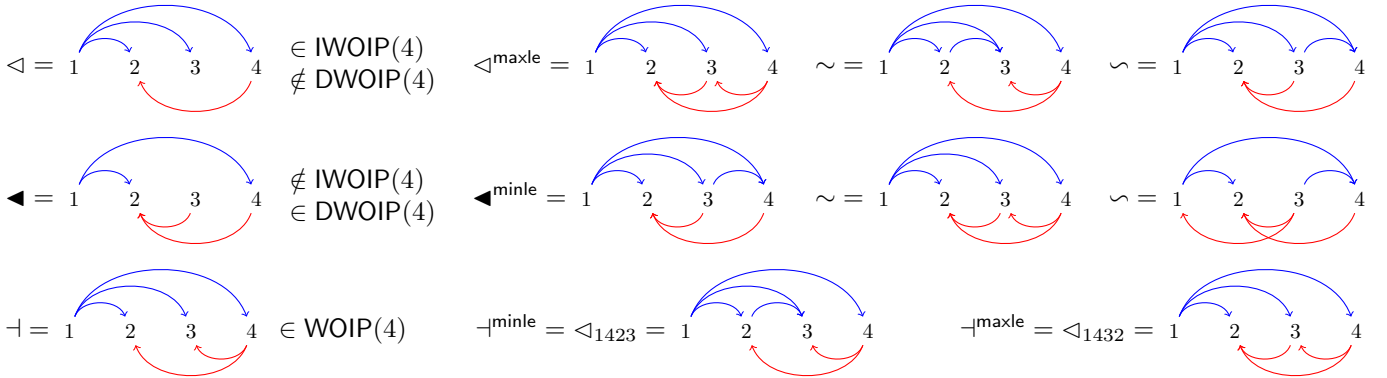

	\vspace{-.2cm}
	\centerline{
	\begin{tabular}{l@{\hspace{.3cm}}l@{\hspace{.3cm}}l@{\hspace{.3cm}}l}
    	${\less} = \raisebox{-.95cm}{\scalebox{0.8}{\input{relations/poset6}}} \begin{array}{l} \in \IWOIP(4) \\ \notin \DWOIP(4) \end{array}$ &
    	${\maxle{\less}} = \raisebox{-.95cm}{\scalebox{0.8}{\input{relations/poset1432}}}$ &
    	${\sim} = \raisebox{-.95cm}{\scalebox{0.8}{\input{relations/poset1423}}}$ &
    	${\backsim} = \raisebox{-.95cm}{\scalebox{0.8}{\input{relations/poset1342}}}$
		\\[-.3cm]
    	${\bless} = \raisebox{-.95cm}{\scalebox{0.8}{\input{relations/poset5}}} \begin{array}{l} \notin \IWOIP(4) \\ \in \DWOIP(4) \end{array}$ &
    	${\minle{\bless}} = \raisebox{-.95cm}{\scalebox{0.8}{\input{relations/poset1342}}}$ &
    	${\sim} = \raisebox{-.95cm}{\scalebox{0.8}{\input{relations/poset1432}}}$ &
    	${\backsim} = \raisebox{-.95cm}{\scalebox{0.8}{\input{relations/poset3142}}}$
		\\[-.3cm]
    	\multicolumn{4}{l}{${\dashv} = \raisebox{-.95cm}{\scalebox{0.8}{\input{relations/woip_1423_1432}}} \in \WOIP(4) \hspace{.9cm}
    	{\minle{\dashv}} = {\less_{1423}} = \raisebox{-.95cm}{\scalebox{.8}{\input{relations/poset1423}}} \hspace{.95cm}
    	{\maxle{\dashv}} = {\less_{1432}} = \raisebox{-.95cm}{\scalebox{.8}{\input{relations/poset1432}}}$}
	\end{tabular}
}
	\vspace{-.6cm}
	\caption{Examples and counterexamples of elements in $\IWOIP(4)$ and $\DWOIP(4)$.}
	\label{fig:IWOIP-DWOIP}
	\vspace{-.4cm}
\end{figure}

By symmetry, we only prove the characterization of~$\IWOIP(n)$.
Assume first that~${\less} \in \IPos(n)$ is such that~$a \less c \implies a \less b \text{ or } b \less c$ for all~$a < b < c$. Let
\[
{\maxle{\less}} \eqdef {\less} \cup \bigset{(b,a)}{a < b \text{ incomparable in } {\less}}
\]
denote the binary relation obtained from~$\less$ by adding a decreasing relation between any two incomparable elements in~$\less$ (see \fref{fig:IWOIP-DWOIP}). The following claim is proved in Appendix~\ref{subsec:appendixIWOIPDWOIP}.

\vspace{-.1cm}
\begin{claim}
\label{claim:maxlePoset}
$\maxle{\less}$ is a poset. 
\end{claim}
\vspace{-.1cm}

\noindent
Moreover~$\maxle{\less}$ is a total order (since any two elements are comparable in~$\maxle{\less}$ by definition) which is a linear extension of~$\less$ (since ${\less} \subseteq {\maxle{\less}}$ by definition). Finally, any other linear extension of~$\less$ is smaller than~$\maxle{\less}$ in weak order (since a linear extension of~$\less$ contains~$\less$ and~$\maxle{\less} \ssm {\less} \subseteq \rel[D]_n$). We conclude that~$\maxle{\less}$ is the maximal linear extension of~$\less$ in weak~order.

Reciprocally, assume now that there exists~$a < b < c$ such that~$a \less c$ while $a \not\less b$ and~$b \not\less c$. The transitivity of~$\less$ implies that~$b \not\less a$ and~$c \not\less b$. Let~${{\sim} \eqdef {\less} \cup \{(a,b),(c,b)\}}$ and~${{\backsim} \eqdef {\less} \cup \{(b,a),(b,c)\}}$. Note that $\sim$ and~$\backsim$ are still acyclic (but not necessary transitive). Indeed any cycle for example in~$\sim$ would involve either~$(a,b)$ or~$(c,b)$, but not both. If~$\sim$ has a cycle involving for example~$(a,b)$, then~$b \less a$ by transitivity of~$\less$, which gives a contradiction. Thus they admit linear extensions, and we consider minimal linear extensions~$\rho$ of~$\sim$ and~$\sigma$ of~$\backsim$. We conclude that~$\rho$ and~$\sigma$ are minimal linear extensions of~$\less$ incomparable in the weak order as illustrated on \fref{fig:IWOIP-DWOIP}. 
\end{proof}

\begin{remark}
\label{rem:coverEnough}
Note that it is enough to check the conditions of Proposition~\ref{prop:characterizationIWOIPDWOIP} only for all cover relations~$a \less c$ and~$a \more c$ of~$\less$. Indeed, consider~$a < b < c$ where~$a \less c$ is not a cover relation, so that there exists~$u \in [n]$ such that~$a \less u \less c$. Assume for example that~$b < u$, the case~$u < b$ being symmetric. Hence~$a < b < u$ and~$a \less u$ implies that either~$a \less b$ or~$b \less u$ (by induction on the length of the minimal chain between~$a$ and~$c$). If~$b \less u$, we obtain that~$b \less u \less c$ so that~$b \less c$.
\end{remark}

We have seen in Corollary~\ref{coro:WOIPLattice} that the weak order $(\WOIP(n), \wole)$ on interval posets forms a lattice. Using the characterization of Proposition~\ref{prop:characterizationIWOIPDWOIP}, we now show that the subposets~${(\IWOIP(n), \wole)}$ and~${(\DWOIP(n), \wole)}$ of the weak order~$(\IPos(n), \wole)$ form lattices --- although there are not sublattices of~${(\IPos(n), \wole, \meetT, \joinT)}$. We define the \defn{$\IWOIP$ increasing deletion}, the \defn{$\DWOIP$ decreasing deletion}, and the \defn{$\WOIP$ deletion}~by
\[
\begin{array}{l@{}l@{\;}l}
{\IWOIPid{\less}} & \eqdef {\less} \ssm \tc{(\rel[I]_n \ssm \Inc{\less})} & = {\less} \ssm \bigset{(a,c)}{\exists \; a < b_1 < \dots < b_k < c, \; a \not\less b_1 \not\less \dots \not\less b_k \not\less c}, \\[.1cm]
{\DWOIPdd{\less}} & \eqdef {\less} \ssm \tc{(\rel[D]_n \ssm \Dec{\less})} & = {\less} \ssm \bigset{(c,a)}{\exists \; a < b_1 < \dots < b_k < c, \; a \not\more b_1 \not\more \dots \not\more b_k \not\more c}, \\[.1cm]
{\WOIPd{\less}} & \eqdef {\IWOIPid{(\DWOIPdd{\less})}} & = {\DWOIPdd{(\IWOIPid{\less})}}.
\end{array}
\]
These operations are illustrated on \fref{fig:IWOIPid/DWOIPdd/WOIPd}.

\begin{figure}[t]
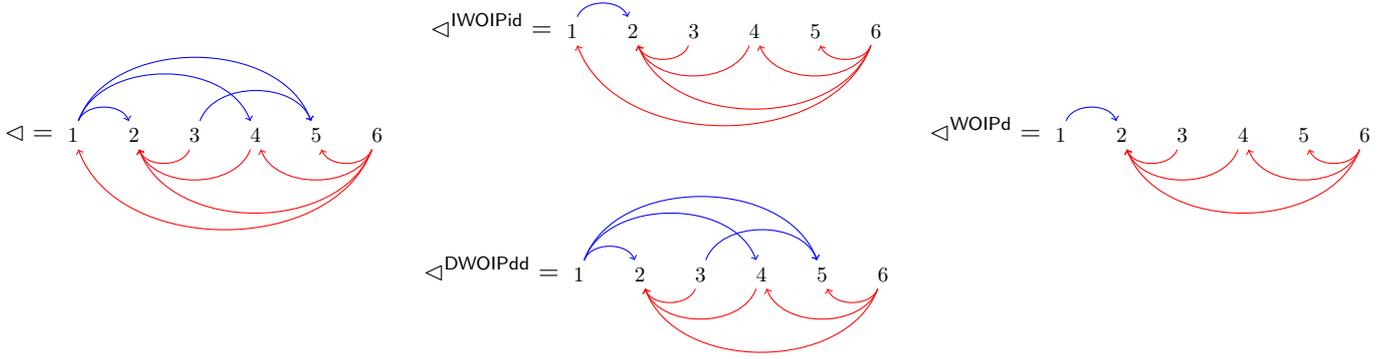

	\vspace{-.6cm}
	\centerline{
	\begin{tabular}{c@{\quad}c@{\quad}c}
		\multirow{ 2}{*}{${\less} = \raisebox{-1.55cm}{\scalebox{0.8}{\input{relations/poset7}}}$} &
		${\IWOIPid{\less}} = \raisebox{-1.55cm}{\scalebox{0.8}{\input{relations/poset8}}}$ &
		\multirow{ 2}{*}{${\WOIPd{\less}} = \raisebox{-1.55cm}{\scalebox{0.8}{\input{relations/posetWOIPlarge}}}$} \\
		& ${\DWOIPdd{\less}} = \raisebox{-1.55cm}{\scalebox{0.8}{\input{relations/poset9}}}$
	\end{tabular}
}
	\vspace{-.5cm}
	\caption{The $\IWOIP$ increasing deletion, the $\DWOIP$ decreasing deletion, and the $\WOIP$ deletion.}
	\label{fig:IWOIPid/DWOIPdd/WOIPd}
	\vspace{-.4cm}
\end{figure}

\begin{remark}
\label{rem:IWOIPidDWOIPdd}
Similar to Remark~\ref{rem:tdd}, the $\IWOIP$ increasing deletion (resp.~$\DWOIP$ decreasing deletion) deletes at once all increasing relations which prevent the poset to be in~$\IWOIP(n)$ (resp.~in~$\DWOIP(n)$). Deleting only the relations~$(a,c)$ (resp.~$(c,a)$) for which there exists~${a < b < c}$ such that~$a \not\less b \not\less c$ (resp.~$a \not\more b \not\more c$) would require several iterations. For example, we would need~$n$ iterations to obtain~$\IWOIPid{\set{(i,j)}{i,j \in [n], \; i+1 < j}} = \varnothing$.
\end{remark}

These functions satisfy the following properties.

\begin{lemma}
\label{lem:IWOIPidDWOIPdd1}
For any poset~${\less} \in \IPos(n)$, we have ${\IWOIPid{\less}} \in \IWOIP(n)$ and~${\DWOIPdd{\less}} \in \DWOIP(n)$. Moreover, ${\less} \in \DWOIP(n) \Longrightarrow {\IWOIPid{\less}} \in \WOIP(n)$ and~${\less} \in \IWOIP(n) \Longrightarrow {\DWOIPdd{\less}} \in \WOIP(n)$.
\end{lemma}

\begin{proof}
We prove the result for~$\IWOIPid{\less}$, the proof for~$\DWOIPdd{\less}$ being symmetric. The details of the following claim are given in Appendix~\ref{subsec:appendixIWOIPDWOIP}.

\vspace{-.15cm}
\begin{claim}
\label{claim:IWOIPidPoset}
$\IWOIPid{\less}$ is a poset.
\end{claim}
\vspace{-.15cm}

\noindent
Thus the characterization of Proposition~\ref{prop:characterizationIWOIPDWOIP} implies that~$\IWOIPid{\less}$ is always in~$\IWOIP(n)$, and even in~$\WOIP(n)$ when~${\less} \in \DWOIP(n)$.
\end{proof}

\begin{lemma}
\label{lem:IWOIPidDWOIPdd2}
For any poset~${\less} \in \IPos(n)$, the poset~$\IWOIPid{\less}$ (resp.~$\DWOIPdd{\less}$) is the weak order minimal (resp.~maximal) poset in~$\IWOIP(n)$ bigger than~${\less}$ (resp.~in~$\DWOIP(n)$ smaller than~${\less}$).
\end{lemma}

\begin{proof}
We prove the result for~$\IWOIPid{\less}$, the proof for~$\DWOIPdd{\less}$ being symmetric. Observe first that~${\less} \wole {\IWOIPid{\less}}$ since~$\IWOIPid{\less}$ is obtained from~$\less$ by deleting increasing relations. Consider now~${\bless} \in \IWOIP(n)$ such that~${\less} \wole {\bless}$. By definition, we have~${\Inc{\less}} \supseteq {\Inc{\bless}}$ and~${\Dec{\less}} \subseteq {\Dec{\bless}}$. Since~$\Dec{(\IWOIPid{\less})} = {\Dec{\less}}$, it just remains to show that for any~$(a,c) \in {\Inc{\bless}}$, there exist no $a < b_1 < \dots < b_k < c$ with~$a \not\less b_1 \not\less b_2 \not\less \dots \not\less b_k \not\less c$. Assume otherwise and choose such a pair~$(a,c)$ with~$c-a$ minimal. Since~${\bless} \in \IWOIP(n)$ and~$a < b_1 < c$ are such that~$a \bless c$ while~$a \not\bless b_1$ (because~$a \not\Inc{\less} b_1$ and~${\Inc{\bless}} \subset {\Inc{\less}}$), we have~$b_1 \bless c$. But this assertion contradicts the minimality of~$c-a$.
\end{proof}

\begin{proposition}
\label{prop:IWOIPDWOIPLattices}
The subposets of the weak order~$(\IPos(n), \wole)$ induced by~$\IWOIP(n)$ and $\DWOIP(n)$ are lattices whose meets and joins are given by
\[
\begin{array}{l@{\;=\;}l@{\qquad}l@{\;=\;}l}
{\less} \meetIWOIP {\bless} & {\less} \meetT {\bless} 						& {\less} \joinIWOIP {\bless} & \IWOIPid{\big( {\less} \joinT {\bless} \big)} \\
{\less} \meetDWOIP {\bless} & \DWOIPdd{\big( {\less} \meetT {\bless} \big)}	& {\less} \joinDWOIP {\bless} & {\less} \joinT {\bless} \; . \\
\end{array}
\]
\end{proposition}

\begin{proof}
We prove the result for~$\IWOIP(n)$, the proof for~$\DWOIP(n)$ being symmetric. Consider ${\less}, {\bless} \in \IWOIP(n)$. We first prove that~${\dashv} \eqdef {\less} \meetT {\bless} = \tdd{\big( \tc{(\Inc{\less} \cup \Inc{\bless})} \cup (\Dec{\less} \cap \Dec{\bless}) \big)}$ is also in~$\IWOIP(n)$ (see also Proposition~\ref{prop:STMCIDConflictFunction} and Example~\ref{exm:IWOIPDWOIPmeetsemisublattices} for a more systematic approach). For any cover relation~$a \dashv c$ and~$a < b < c$, we have~$a \Inc{\dashv} c$ so that~$a \Inc{\less} c$ or~$a \Inc{\bless} c$ (since we have a cover relation). Since~${\less}, {\bless} \in \IWOIP(n)$, we obtain that~$a \Inc{\less} b$, or~$b \Inc{\less} c$, or~$a \Inc{\bless} b$, or~$b \Inc{\bless} c$. Thus, $a \dashv b$ or~$b \dashv c$ for any cover relation~$a \dashv c$ and any~$a < b < c$. Using Remark~\ref{rem:coverEnough}, we conclude that~${\dashv} \in \IWOIP(n)$.

On the other hand, Lemma~\ref{lem:IWOIPidDWOIPdd2} asserts that~$\IWOIPid{\big( {\less} \joinT {\bless} \big)}$ is the weak order minimal poset in~$\IWOIP(n)$ bigger than~${{\less} \joinT {\bless}}$. Any poset in~$\IWOIP(n)$ bigger than~$\less$ and~$\bless$ is also bigger than~${{\less} \joinT {\bless}}$, and thus bigger than~${{\IWOIPid{\big( {\less} \joinT {\bless} \big)}}}$. We conclude that~$\IWOIPid{\big( {\less} \joinT {\bless} \big)}$ is indeed the join of~$\less$~and~$\bless$.
\end{proof}

We finally deduce from Proposition~\ref{prop:IWOIPDWOIPLattices} and Lemma~\ref{lem:IWOIPidDWOIPdd1} an alternative formula for the meet and join in the weak order~$(\WOIP(n), \wole)$. See also Corollary~\ref{coro:WOIPLattice}.

\begin{corollary}
\label{coro:WOIPMeetJoin}
The meet and join in the weak order on~$\WOIP(n)$ are given by
\[
{\less} \meetWOIP {\bless} = \DWOIPdd{\big( {\less} \meetT {\bless} \big)} \qquad\text{and}\qquad {\less} \joinWOIP {\bless} = \IWOIPid{\big( {\less} \joinT {\bless} \big)}.
\]
\end{corollary}


\section{From the associahedron}
\label{sec:relevantFamiliesAssociahedron}

Similarly to the previous section, we now briefly discuss some relevant families of posets corresponding to the elements, the intervals, and the faces of the associahedron.  Further similar families of posets arising from permutreehedra~\cite{PilaudPons} will be discussed in Section~\ref{sec:relevantFamiliesPermutreehedra}. This section should just be considered as a simplified prototype to the next section. We therefore omit the proofs which will appear in a more general context in Sections~\ref{sec:relevantFamiliesPermutreehedra} and~\ref{sec:sublattices}.

We denote by~$\fB(n)$ the set of \defn{planar rooted binary trees} with~$n$ nodes, that we simply call binary trees here for short. We label the vertices of a binary tree~$\tree \in \fB(n)$ through an \defn{inorder traversal}, \ie such that all vertices in the left (resp.~right) child of a vertex~$v$ of~$\tree$ receive a label smaller (resp.~larger) than the label of~$v$. From now on, we identify a vertex and its label.

There is a fundamental surjection from permutations to binary trees. Namely, a permutation~$\sigma \eqdef \sigma_1 \dots \sigma_n \in \fS(n)$ is mapped to the binary tree~$\bt(\sigma) \in \fB(n)$ obtained by successive insertions of~$\sigma_n, \dots, \sigma_1$ in a binary (search) tree. The fiber of a tree~$\tree$ is precisely the set of linear extensions of~$\tree$. It is an interval of the weak order whose minimal and maximal elements respectively avoid the patterns~$312$ and~$132$. Moreover, the fibers of~$\bt$ define a lattice congruence of the weak order. Thus, the set~$\fB(n)$ of binary trees is endowed with a lattice structure~$\wole$ defined~by
\[
\tree \wole \tree' \iff \exists \; \sigma, \sigma' \in \fS(n) \text{ such that } \bt(\sigma) = \tree, \; \bt(\sigma') = \tree' \text{ and } \sigma \wole \sigma' \\
\]
whose meet~$\meetTO$ and join~$\joinTO$ are given by
\[
\tree \meetTO \tree' = \bt(\sigma \meetWO \sigma')
\quad\text{and}\quad
\tree \joinTO \tree' = \bt(\sigma \joinWO \sigma')
\]
for any representatives~$\sigma, \sigma' \in \fS(n)$ such that~$\bt(\sigma) = \tree$ and~$\bt(\sigma') = \tree'$. Note that in particular, $\tree \wole \tree'$ if and only if~$\sigma \wole \sigma'$ where~$\sigma$ and~$\sigma'$ denote the minimal (resp.~maximal) linear extensions of~$\tree$ and~$\tree'$ respectively. For example, the minimal (resp.~maximal) tree is the left (resp.~right) comb whose unique linear extension is~$e \eqdef [1,2,\dots,n]$ (resp.~$\wo \eqdef [n,\dots,2,1]$). This lattice structure is the \defn{Tamari lattice} whose cover relations are given by \defn{right rotations} on binary trees. It was introduced by D.~Tamari~\cite{TamariFestschrift} on Dyck paths, our presentation is a more modern perspective~\cite{BjornerWachs, Reading-CambrianLattices}.


\subsection{Tamari Order Element Posets}
\label{subsec:TOEP}

We consider the tree~$\tree$ as a poset~$\less_{\tree}$, defined by~${i \less_{\tree} j}$ when~$i$ is a descendant of~$j$ in~$\tree$. In other words, the Hasse diagram of~$\less_{\tree}$ is the tree~$\tree$ oriented towards its root. An illustration is provided in \fref{fig:toep}. Note that the increasing (resp.~decreasing) subposet of~$\less_{\tree}$ is given by~$i \Inc{\less_{\tree}} j$ (resp.~$i \Dec{\less_{\tree}} j$) if and only if $i$ belongs to the left (resp.~right) subtree of~$j$ in~$\tree$.

\begin{figure}[h]
	\input{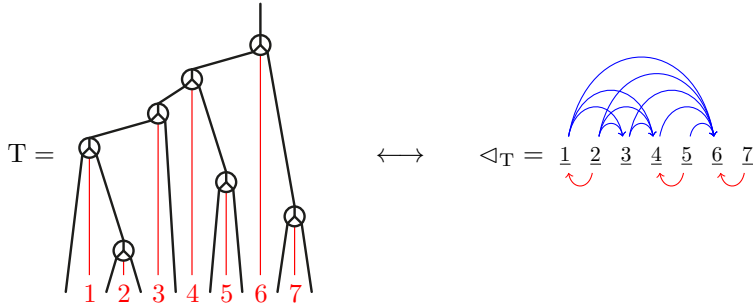}
	\vspace{-.3cm}
	\caption{A Tamari Order Element Poset ($\TOEP$).}
	\label{fig:toep}
	\vspace{-.4cm}
\end{figure}

We say that~$\less_{\tree}$ is a \defn{Tamari order element poset}, and we denote~by
\[
\TOEP(n) \eqdef \bigset{{\less_{\tree}}}{\tree \in \fB(n)}
\]
the set of all Tamari order element posets on~$[n]$. We first characterize them (see Proposition~\ref{prop:characterizationPEP}).

\begin{proposition}
\label{prop:characterizationTOEP}
A poset~${\less} \in \IPos(n)$ is in~$\TOEP(n)$ if and only if
\begin{itemize}
\item $\forall \; a < b < c, \; a \less c \implies b \less c$ and $a \more c \implies a \more b$, 
\item for all~$a < c$ incomparable in~$\less$, there exists ${a < b < c}$ such that~$a \less b \more c$.
\end{itemize}
\end{proposition}

Now we establish the relationship between the Tamari lattice on~$\fB(n)$ and the weak order on~$\TOEP(n)$ (see Proposition~\ref{prop:weakOrderPEP}).

\begin{proposition}
\label{prop:weakOrderTOEP}
For any binary trees~$\tree, \tree' \in \fB(n)$, we have~$\tree \wole \tree'$ in the Tamari lattice if and only if~${\less_{\tree}} \wole {\less_{\tree'}}$ in the weak order on posets.
\end{proposition}

It follows that the subposet of the weak order~$(\IPos, \wole)$ induced by the set~$\TOEP(n)$ is isomorphic to the Tamari lattice on~$\fB(n)$, and is thus a lattice. We conclude on $\TOEP(n)$ with the following stronger statement (see Theorem~\ref{thm:coveringOrientationElementsSublatticeIPos}).

\begin{proposition}
\label{prop:TOEPsublattice}
The set~$\TOEP(n)$ induces a sublattice of the weak order~$(\IPos(n), \wole, \meetT, \joinT)$.
\end{proposition}


\subsection{Tamari Order Interval Posets}
\label{subsec:TOIP}

For two binary trees~$\tree, \tree' \in \fB(n)$ with~$\tree \wole \tree'$, we denote by~$[\tree, \tree'] \eqdef \set{\tree[S] \in \fB(n)}{\tree \wole \tree[S] \wole \tree'}$ the Tamari order interval between~$\tree$ and~$\tree'$. We can see this interval as the poset
\[
{\less_{[\tree, \tree']}} \eqdef \bigcap_{\tree \wole \tree[S] \wole \tree'} {\less_{\tree}} = {\less_{\tree}} \cap {\less_{\tree'}} = {\Inc{\less_{\tree'}}} \cap {\Dec{\less_{\tree}}}.
\]
See \fref{fig:toip} for an example.
\begin{figure}[ht]
    \input{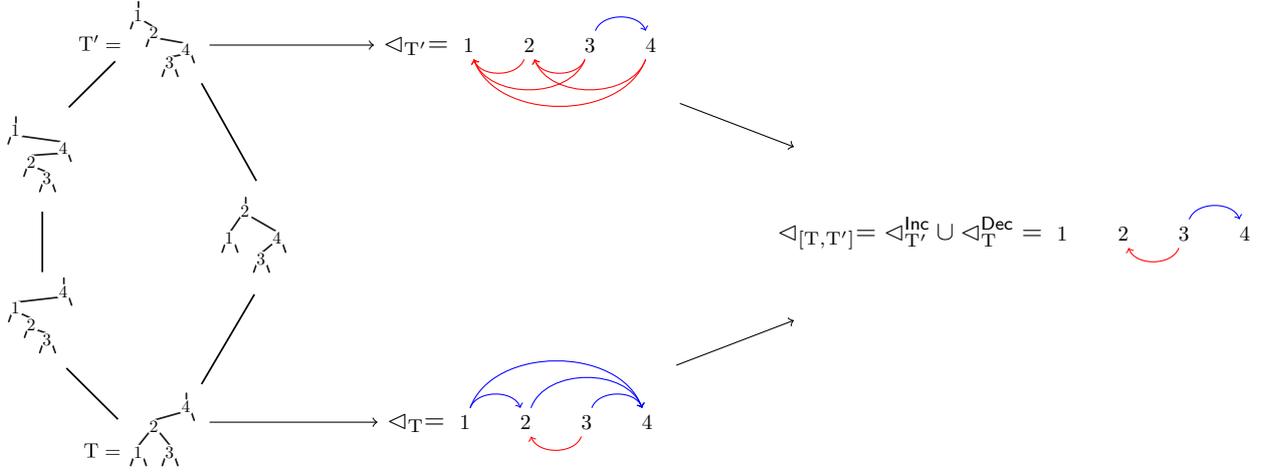}
    \vspace{-.5cm}
    \caption{A Tamari Order Interval Poset ($\TOIP$).}
    \label{fig:toip}
	\vspace{-.4cm}
\end{figure}

This poset~${\less_{[\tree, \tree']}}$ was introduced in~\cite{ChatelPons} with the motivation that its linear extensions are precisely the linear extensions of all binary trees in the interval~$[\tree, \tree']$. We say that~$\less_{[\tree,\tree']}$ is a \defn{Tamari order interval poset}, and we denote by
\[
\TOIP(n) \eqdef \bigset{{\less_{[\tree,\tree']}}}{\tree, \tree' \in \fB(n), \tree \wole \tree'}
\]
the set of all Tamari order interval posets on~$[n]$. The following characterization of these posets (see Proposition~\ref{coro:characterizationPIP}) already appeared in~\cite[Thm.~2.8]{ChatelPons}.

\begin{corollary}[\protect{\cite[Thm.~2.8]{ChatelPons}}]
\label{coro:characterizationTOIP}
A poset~${\less} \in \IPos(n)$ is in~$\TOIP(n)$ if and only if $\forall \; a < b < c$, 
\[
a \less c \implies b \less c
\qquad\text{and}\qquad
a \more c \implies a \more b.
\]
\end{corollary}

Now we describe the weak order on~$\TOIP(n)$ (see Proposition~\ref{prop:weakOrderPIP}, Corollary~\ref{coro:PIPLattice}).

\begin{proposition}
\label{prop:weakOrderTOIP}
For any~$\tree[S] \wole \tree[S]'$ and~$\tree \wole \tree'$, we have~${\less_{[\tree[S], \tree[S]']}} \wole {\less_{[\tree, \tree']}} \iff \tree[S] \wole \tree$ and~$\tree[S]' \wole \tree'$.
\end{proposition}

\begin{corollary}
\label{coro:TOIPLattice}
The weak order~$(\TOIP(n), \wole)$ is a lattice whose meet and join are given by
\[
{\less_{[\tree[S],\tree[S]']}} \meetTOIP {\less_{[\tree,\tree']}} = {\less_{[\tree[S] \meetTO \tree, \tree[S]' \meetTO \tree']}}
\qquad\text{and}\qquad
{\less_{[\tree[S],\tree[S]']}} \joinTOIP {\less_{[\tree,\tree']}} = {\less_{[\tree[S] \joinTO \tree, \tree[S]' \joinTO \tree']}}.
\]
\end{corollary}

\begin{corollary}
The set~$\TOEP(n)$ induces a sublattice of the weak order~$(\TOIP, \wole, \meetT, \joinT)$.
\end{corollary}

In fact, we will derive the following statement (see Corollary~\ref{coro:TOIPCOIPBOIPsublattices}).

\begin{proposition}
\label{prop:TOIPSublattice}
The set~$\TOIP(n)$ induces a sublattice of the weak order~$(\IPos(n), \wole, \meetT, \joinT)$.
\end{proposition}


\subsection{Tamari Order Face Posets}
\label{subsec:TOFP}

The binary trees of~$\fB(n)$ correspond to the vertices of the associahedron~$\Asso[n]$ constructed \eg by J.-L.~Loday in~\cite{Loday}. We now consider all the faces of the associahedron~$\Asso[n]$ which correspond to \defn{Schr\"oder trees}, \ie planar rooted trees where each node has either none or at least two children. Given a Schr\"oder tree~$\tree[S]$, we label the angles between consecutive children of the vertices of~$\tree[S]$ in inorder, meaning that each angle is labeled after the angles in its left child and before the angles in its right child.
Note that a binary tree~$\tree$ belongs to the face of the associahedron~$\Asso[n]$ corresponding to a Schr\"oder tree~$\tree[S]$ if and only if~$\tree[S]$ is obtained by edge contractions in~$\tree$. The set of such binary trees is an interval~$[\tree^{\min}(\tree[S]), \tree^{\max}(\tree[S])]$ in the Tamari lattice, where the minimal (resp.~maximal) tree~$\tree^{\min}(\tree[S])$ (resp.~$\tree^{\max}(\tree[S])$) is obtained by replacing the nodes of~$\tree[S]$ by left (resp.~right) combs as illustrated~in~\fref{fig:tofp}.
\begin{figure}[h]
	\vspace{-.2cm}
	\input{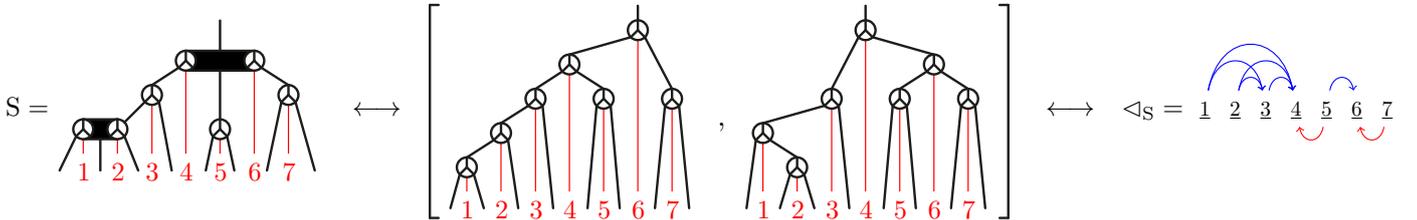}
	\vspace{-.3cm}
	\caption{A Tamari Order Face Poset ($\TOFP$).}
	\label{fig:tofp}
	\vspace{-.4cm}
\end{figure}

We associate to a Schr\"oder tree~$\tree[S]$ the poset~$\less_{\tree[S]} \eqdef \less_{[\tree^{\min}(\tree[S]), \tree^{\max}(\tree[S])]}$. Equivalently, $i \less_{\tree[S]} j$ if and only if the angle~$i$ belongs to the left or the right child of the angle~$j$. See \fref{fig:tofp}. Note~that
\begin{itemize}
\item a binary tree~$\tree$ belongs to the face of the associahedron~$\Asso[n]$ corresponding to a Schr\"oder tree~$\tree[S]$ if and only if~$\less_{\tree}$ is an extension of~$\less_{\tree[S]}$, and
\item the linear extensions of~$\less_{\tree[S]}$ are precisely the linear extensions of~$\less_{\tree}$ for all binary trees~$\tree$ which belong to the face of the associahedron~$\Asso[n]$ corresponding to~$\tree[S]$.
\end{itemize}
We say that~$\less_{\tree[S]}$ is a \defn{Tamari order face poset}, and we denote by
\[
\TOFP(n) \eqdef \bigset{{\less_{\tree[S]}}}{\tree[S] \text{ Schr\"oder tree on $[n]$}}
\]
the set of all Tamari order face posets. We first characterize these posets (see Proposition~\ref{prop:characterizationPFP}).

\begin{proposition}
\label{prop:characterizationTOFP}
A poset~${\less} \in \IPos(n)$ is in~$\TOFP(n)$ if and only if ${\less} \in \TOIP(n)$ (see characterization in Corollary~\ref{coro:characterizationTOIP}) and for all~$a < c$ incomparable in~$\less$, either there exists~$a < b < c$ such that~$a \not\more b \not\less c$, or for all~$a < b < c$ we have~$a \more b \less c$.
\end{proposition}

Consider now the weak order on~$\TOFP(n)$. It turns out (see Proposition~\ref{prop:weakOrderPFP}) that this order on Schr\"oder trees coincides with the \defn{facial weak order} on the associahedron~$\Asso[n]$ studied in~\cite{PalaciosRonco, NovelliThibon-trialgebras, DermenjianHohlwegPilaud}. This order is a quotient of the facial weak order on the permutahedron by the fibers of the Schr\"oder tree insertion~$\st$. In particular, the weak order on~$\TOFP(n)$ is a lattice.

\begin{remark}
\label{rem:TOFPnotSublattice}
The example of Remark~\ref{rem:WOFPnotSublattice} shows that~$(\TOFP(n), \wole, \meetTOFP, \joinTOFP)$ is not a sublattice of~$(\IPos(n), \wole, \meetT, \joinT)$, nor a sublattice of~$(\WOIP(n), \wole, \meetWOIP, \joinWOIP)$, nor a sublattice of~$(\TOIP(n), \wole, \meetTOIP, \joinTOIP)$.  
\end{remark}


\subsection{$\TOIP$ deletion}
\label{subsec:TOIPdeletion}

We finally define a projection from all posets of~$\IPos(n)$ to~$\TOIP(n)$. We call \defn{$\TOIP$ deletion} the map defined by
\[
{\TOIPd{\less}} \eqdef {\less} \ssm \big( \bigset{(a,c)}{\exists \; a < b < c, \; b \not\less c} \cup \bigset{(c,a)}{\exists \; a < b < c, \; a \not\more b} \big).
\]
Then~${\TOIPd{\less}} \in \TOIP(n)$ for any poset~${\less} \in \IPos(n)$ (see Lemma~\ref{lem:IPIPidDPIPdd1}). We compare this map with the binary search tree and Schr\"oder tree insertions described earlier (see Proposition~\ref{prop:PIPd/bst}, Corollary~\ref{coro:PIPd/bstIntervals} and Proposition~\ref{prop:PIPd/st}).

\begin{proposition}
\label{prop:TOIPd/bst}
For any permutation~$\sigma \in \fS(n)$, for any permutations~$\sigma, \sigma' \in \fS(n)$ with~$\sigma \wole \sigma'$, and for any ordered partition~$\pi$ of~$[n]$, we have
\[
{\TOIPd{(\less_\sigma)}} = {\less_{\bt(\sigma)}},
\qquad
{\TOIPd{(\less_{[\sigma,\sigma']})}} = {\less_{[\bt(\sigma), \bt(\sigma')]}}
\qquad\text{and}\qquad
{\TOIPd{(\less_\pi)}} = {\less_{\st(\pi)}}.
\]
\end{proposition}

\begin{example}
Compare Figures~\ref{fig:woep} and~\ref{fig:toep}, Figures~\ref{fig:woip} and~\ref{fig:toip}, and Figures~\ref{fig:wofp} and~\ref{fig:tofp}.
\end{example}


\section{From permutreehedra}
\label{sec:relevantFamiliesPermutreehedra}

Extending Sections~\ref{sec:relevantFamiliesPermutahedron} and~\ref{sec:relevantFamiliesAssociahedron}, we describe further relevant families of posets corresponding to the elements, the faces, and the intervals in the permutreehedra introduced in~\cite{PilaudPons}. This provides a wider range of examples and uniform proofs, at the cost of increasing the technicalities.


\subsection{Permutree Element Posets}
\label{subsec:PEP}

We first recall from~\cite{PilaudPons} the definition of permutrees.

\begin{definition}[\cite{PilaudPons}]
\label{def:permutree}
A \defn{permutree} is a directed tree~$\tree$ with vertex set~$\ground$ endowed with a bijective vertex labeling $p : \ground \to [n]$ such that for each vertex~$v \in \ground$,
\begin{enumerate}[(i)]
\item $v$ has one or two parents (outgoing neighbors), and one or two children (incoming neighbors);
\item if $v$ has two parents (resp.~children), then all labels in the left ancestor (resp.~descendant) subtree of~$v$ are smaller than~$p(v)$ while all labels in the right ancestor (resp.~descendant) subtree of~$v$ are larger than~$p(v)$.
\end{enumerate}
The \defn{orientation} of a permutree~$\tree$ is~$\orientation(\tree) = (n, \orientation\positive, \orientation\negative)$ where~$\orientation\positive$ is the set of labels of the nodes with two parents while~$\orientation\negative$ is the set of labels of the nodes with two children. Note that there is a priori no conditions on these sets~$\orientation\positive$ and~$\orientation\negative$: they can be empty, they can intersect, etc. For a given orientation~$\orientation = (n, \orientation\positive, \orientation\negative)$, we denote by~$\Permutrees$ the set of permutrees with orientation~$\orientation$.
\end{definition}

\fref{fig:permutrees} gives five examples of permutrees. While the first is generic, the other four show that specific permutrees encode relevant combinatorial objects, depending on their orientations:

\medskip
\centerline{
\begin{tabular}{c||c|c|c|c}
orientation~$(n, \orientation\positive, \orientation\negative)$ & $(n, \varnothing, \varnothing)$ & $(n, \varnothing, [n])$ & $\orientation\positive \sqcup \orientation\negative = [n]$ & $(n, [n], [n])$ \\
\hline
combinatorial objects & permutations & binary trees & Cambrian trees~\cite{ChatelPilaud} & binary sequences
\end{tabular}
}

\medskip
\noindent
See~\cite{PilaudPons} for more details on the interpretation of these combinatorial objects as permutrees. We use drawing conventions from~\cite{PilaudPons}: nodes are ordered by their labels from left to right, edges are oriented from bottom to top, and we draw a red wall separating the two parents or the two children of a node. Condition~(ii) in Definition~\ref{def:permutree} says that no black edge can cross~a~red~wall.

\begin{figure}[h]
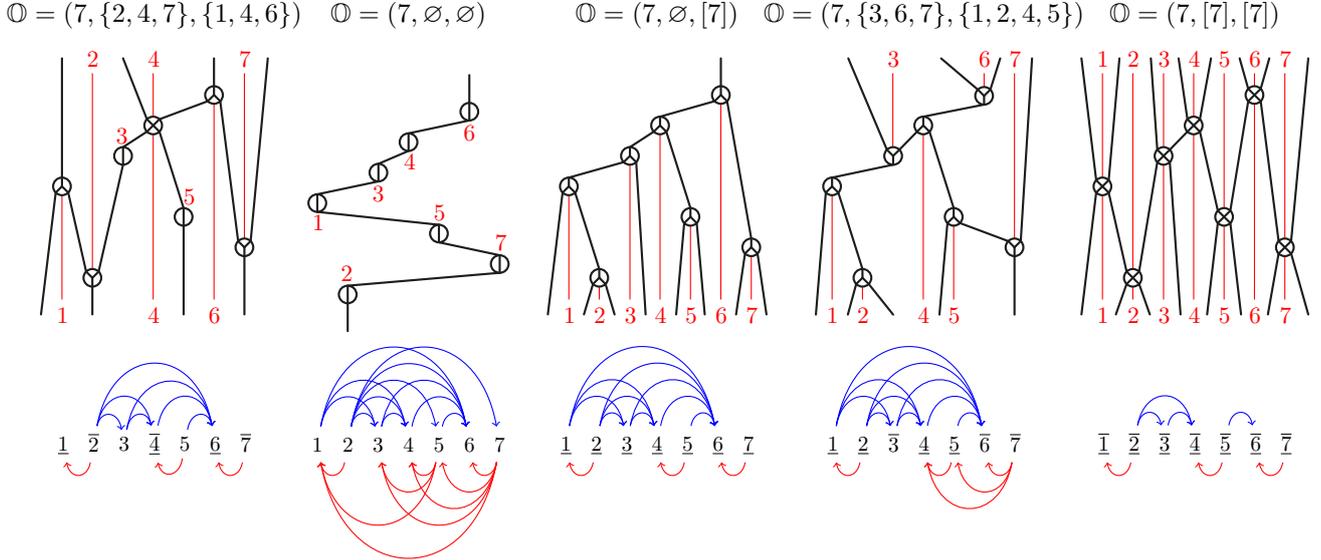

	\centerline{
	\begin{tabular}{c@{\!}cc@{\!}c@{\!}c}
		  $\orientation = (7, \{2,4,7\}, \{1,4,6\})$
		& $\orientation = (7, \varnothing, \varnothing)$
		& $\orientation = (7, \varnothing, [7])$
		& $\orientation = (7, \{3,6,7\}, \{1,2,4,5\})$
		& $\orientation = (7, [7], [7])$
		\\[.2cm]
		  \includegraphics[scale=.8]{permutreeGeneric}
		& \includegraphics[scale=.8]{permutreePermutation}
		& \includegraphics[scale=.8]{permutreeBinaryTree}
		& \includegraphics[scale=.8]{permutreeCambrianTree}
		& \includegraphics[scale=.8]{permutreeBinarySequence}
		\\[-.6cm]
		  \;\scalebox{.8}{\input{relations/posetPEP.tex}}
		& \;\scalebox{.8}{\input{relations/posetWOEP.tex}}
		& \;\scalebox{.8}{\input{relations/posetTOEP.tex}}
		& \;\scalebox{.8}{\input{relations/posetCOEP.tex}}
		& \;\scalebox{.8}{\input{relations/posetBOEP.tex}}
	\end{tabular}
}
	\vspace{-.8cm}
	\caption{Five examples of permutrees~$\tree$ (top) with their posets~$\less_{\tree}$~(bottom). While the first is generic, the last four illustrate specific orientations corresponding to permutations, binary trees, Cambrian trees, and binary sequences.}
	\label{fig:permutrees}
	\vspace{-.4cm}
\end{figure}

For a permutree~$\tree$, we denote by~$\less_{\tree}$ the transitive closure of~$\tree$. That is to say, $i \less_{\tree} j$ if and only if there is an oriented path from~$i$ to~$j$ in~$\tree$. See \fref{fig:permutrees} for illustrations. To visualize the orientation~$\orientation$ in the poset~$\less_{\tree}$, we overline (resp.~underline) the elements of~$\orientation\positive$ (resp.~of~$\orientation\negative$).

We say that~$\less_{\tree}$ is a \defn{permutree element poset} and we denote by
\[
\PEP \eqdef \bigset{{\less_{\tree}}}{\tree \in \Permutrees}
\]
the set of all permutree element posets for a given orientation~$\orientation$. These posets will be characterized in Proposition~\ref{prop:characterizationPEP}. For the moment, we need the following properties from~\cite{PilaudPons}.

\begin{proposition}[\cite{PilaudPons}]
\label{prop:permutrees}
Fix an orientation~$\orientation = (n, \orientation\positive, \orientation\negative)$ of~$[n]$.
\\[-.3cm]
\begin{enumerate}
\item For a permutree~$\tree \in \Permutrees$, the set of linear extensions~$\linearExtensions(\tree)$ of~$\less_{\tree}$ is an interval in the weak order on~$\fS(n)$ whose minimal element avoids the pattern~$ca-b$ with $a < b < c$ and~$b \in \orientation\negative$ (denoted~$31\down{2}$) and the pattern~$b-ca$ with~$a < b < c$ and~$b \in \orientation\positive$ (denoted~$\up{2}31$), and whose maximal element avoids the pattern~$ac-b$ with $a < b < c$ and~$b \in \orientation\negative$ (denoted~$13\down{2}$) and the pattern~$b-ac$ with~$a < b < c$ and~$b \in \orientation\positive$ (denoted~$\up{2}31$) .
\\[-.2cm]
\item The collection of sets~$\linearExtensions(\tree) \eqdef \{\text{linear extensions of } {\less_{\tree}}\}$ for all permutrees~$\tree \in \Permutrees$ forms a partition of~$\fS(n)$. This defines a surjection~$\surjection$ from~$\fS(n)$ to~$\Permutrees$, which sends a permutation~$\sigma \in \fS(n)$ to the unique permutree~$\tree \in \Permutrees$ such that~$\sigma \in \linearExtensions(\tree)$. This surjection can be described directly as an insertion algorithm (we skip this description and refer the interested reader to~\cite{PilaudPons} as it is not needed for the purposes of this paper).
\\[-.2cm]
\item This partition defines a lattice congruence of the weak order (see~\cite{Reading-latticeCongruences, Reading-CambrianLattices, PilaudPons} for details). Therefore, the set of permutrees~$\Permutrees$ is endowed with a lattice structure~$\wole$, called \defn{permutree lattice}, defined by
\[
\tree \wole \tree' \iff \exists \; \sigma, \sigma' \in \fS(n) \text{ such that } \surjection(\sigma) = \tree, \; \surjection(\sigma') = \tree' \text{ and } \sigma \wole \sigma' \\
\]
whose meet~$\meetO$ and join~$\joinO$ are given by
\[
\tree \meetO \tree' = \surjection(\sigma \meetWO \sigma')
\quad\text{and}\quad
\tree \joinO \tree' = \surjection(\sigma \joinWO \sigma')
\]
for any representatives~$\sigma, \sigma' \in \fS(n)$ such that~$\surjection(\sigma) = \tree$ and~$\surjection(\sigma') = \tree'$. In particular, $\tree \wole \tree'$ if and only if~$\sigma \wole \sigma'$ where~$\sigma$ and~$\sigma'$ denote the minimal (resp.~maximal) linear extensions of~$\tree$ and~$\tree'$ respectively.
\\[-.2cm]
\item This lattice structure can equivalently be described as the transitive closure of right rotations in permutrees as described in~\cite{PilaudPons}.
\\[-.2cm]
\item The minimal (resp.~maximal) permutree in the permutree lattice is a left (resp.~right) \mbox{$\orientation$-comb}: it is a chain where each vertex in~$\orientation\positive$ has an additional empty left (resp.~right) parent, while each vertex in~$\orientation\negative$ has an additional empty right (resp.~left) child.
\end{enumerate}
\end{proposition}

For example, we obtain well-known lattice structures for specific orientations:

\medskip
\centerline{
\begin{tabular}{c||c|c|c|c}
orientation~$(n, \orientation\positive, \orientation\negative)$ & $(n, \varnothing, \varnothing)$ & $(n, \varnothing, [n])$ & $\orientation\positive \sqcup \orientation\negative = [n]$ & $(n, [n], [n])$ \\
\hline
\multirow{2}{*}{permutree lattice} & classical & Tamari & Cambrian & boolean \\
& weak order & lattice~\cite{TamariFestschrift} & lattices~\cite{Reading-CambrianLattices} & lattice
\end{tabular}
}

\bigskip
Now we establish the relationship between the permutree lattice on~$\Permutrees$ and the weak order on~$\PEP$.

\begin{proposition}
\label{prop:weakOrderPEP}
For any permutrees~$\tree, \tree' \in \Permutrees$, we have~$\tree \wole \tree'$ in the permutree lattice if and only if~${\less_{\tree}} \wole {\less_{\tree'}}$ in the weak order on posets.
\end{proposition}

\begin{proof}
By Proposition~\ref{prop:permutrees}\,(2), a permutree admits both a minimal and a maximal linear extensions. It follows that~$\PEP \subseteq \WOIP(n)$ and the weak order on~$\PEP$ is therefore given~by
\[
{\less_{\tree}} \wole {\less_{\tree'}} \iff {\minle{\less_{\tree}}} \wole {\minle{\less_{\tree'}}} \text{ and } {\maxle{\less_{\tree}}} \wole {\maxle{\less_{\tree'}}}
\]
according to Proposition~\ref{prop:weakOrderWOIP}. However, we have already mentioned in Proposition~\ref{prop:permutrees}\,(3) that the two conditions on the right are both equivalent to~$\tree \wole \tree'$ in the permutree lattice.
\end{proof}

\begin{remark}
\label{rem:weakOrderPEP}
In fact, we have that $\tree \wole \tree' \iff {\less_{\tree}} \wole {\less_{\tree'}} \iff {\Inc{\less_{\tree}}} \supseteq {\Inc{\less_{\tree'}}} \iff {\Dec{\less_{\tree}}} \subseteq {\Dec{\less_{\tree'}}}.$
\end{remark}

\begin{remark}
Proposition~\ref{prop:weakOrderPEP} affirms that the subposet of the weak order~$(\IPos, \wole)$ induced by the set~$\PEP$ is isomorphic to the permutree lattice on~$\Permutrees$, and is thus a lattice. We will see in Remark~\ref{rem:PEPnotSublatticePos} that the set~$\PEP$ does not always induce a sublattice of~${(\IPos(n), \wole, \meetT, \joinT)}$. Theorem~\ref{thm:coveringOrientationElementsSublatticeIPos} will provide a sufficient condition on the orientation~$\orientation$ for this property. In contrast, we will see in Theorem~\ref{thm:nonCoveringOrientationElementsSublatticeWOIP} that~$\PEP$ always induces a sublattice of~$(\PIP, \wole, \meetPIP, \joinPIP)$ and thus of~$(\WOIP(n), \wole, \meetWOIP, \joinWOIP)$.
\end{remark}


\subsection{Permutree Interval Posets}
\label{subsec:PIP}

For two permutrees~$\tree, \tree' \in \Permutrees$ with~$\tree \wole \tree'$, we denote by~$[\tree, \tree'] \eqdef \set{\tree[S] \in \Permutrees}{\tree \wole \tree[S] \wole \tree'}$ the permutree lattice interval between~$\tree$ and~$\tree'$. As in Proposition~\ref{prop:definitionWOIP}, we can see this interval as the poset
\[
{\less_{[\tree, \tree']}} \eqdef \bigcap_{\tree \wole \tree[S] \wole \tree'} {\less_{\tree}} = {\less_{\tree}} \cap {\less_{\tree'}} = {\Inc{\less_{\tree'}}} \cap {\Dec{\less_{\tree}}}.
\]

We say that~$\less_{[\tree,\tree']}$ is a \defn{permutree interval poset}, and we denote by
\[
\PIP \eqdef \bigset{{\less_{[\tree,\tree']}}}{\tree, \tree' \in \Permutrees, \tree \wole \tree'}
\]
the set of all permutree interval posets for a given orientation~$\orientation$.

We first aim at a concrete characterization of the posets of~$\PIP$. Note that a poset is in~$\PIP$ if and only if it admits a weak order minimal linear extension avoiding the patterns $31\down{2}$ and $\up{2}31$, and a weak order maximal linear extension avoiding the patterns $13\down{2}$ and $\up{2}13$. Similar to our study of~$\IWOIP(n)$ and~$\DWOIP(n)$ in Section~\ref{subsec:IWOIPDWOIP}, it is practical to consider these conditions separately. We thus define the set~$\IPIP$ (resp.~$\DPIP$) of posets which admit a maximal (resp.~minimal) linear extension that avoids the patterns~$\up{2}13$ and~$13\down{2}$ (resp.~$\up{2}31$ and~$31\down{2}$). In order to characterize these posets, we define
\begin{align*}
\IPIPp & \eqdef \set{{\less} \in \IPos(n)}{\forall \; a < b < c \text{ with } b \in \orientation\positive, \; a \less c \implies a \less b}, \\
\IPIPm & \eqdef \set{{\less} \in \IPos(n)}{\forall \; a < b < c \text{ with } b \in \orientation\negative, \; a \less c \implies b \less c}, \\
\IPIPpm & \eqdef \IPIPp \cap \IPIPm,
\end{align*}
and similarly
\begin{align*}
\DPIPp & \eqdef \set{{\less} \in \IPos(n)}{\forall \; a < b < c \text{ with } b \in \orientation\positive, \; a \more c \implies b \more c}, \\
\DPIPm & \eqdef \set{{\less} \in \IPos(n)}{\forall \; a < b < c \text{ with } b \in \orientation\negative, \; a \more c \implies a \more b}, \\
\DPIPpm & \eqdef \IPIPp \cap \IPIPm.
\end{align*}

\begin{proposition}
\label{prop:characterizationIPIPDPIP}
For any orientation~$\orientation$ of~$[n]$, we have
\[
\IPIP = \IWOIP(n) \cap \IPIPpm
\qquad\text{and}\qquad
\DPIP = \DWOIP(n) \cap \DPIPpm.
\]
\end{proposition}

\begin{proof}
Consider~${\less} \in \IWOIP$ and let~$\maxle{\less} = {\less} \cup \set{(b,a)}{a < b \text{ incomparable in } {\less}}$ be its maximal linear extension (see the proof of Proposition~\ref{prop:characterizationIWOIPDWOIP}). Assume first that there is~$a < b < c$ with~$b \in \orientation\positive$ such that~$a \less c$ while~$a \not\less b$. Then we obtain~$b \maxle{\less} a \maxle{\less} c$ which is a~$\up{2}13$-pattern in~$\maxle{\less}$. Reciprocally, if~$\maxle{\less}$ contains a~$\up{2}13$-pattern~$b \maxle{\less} a \maxle{\less} c$ with~$a < b < c$ and~$b \in \orientation\positive$, then $a \less c$ while~$a \not\less b$ by definition of~$\maxle{\less}$. We conclude that~$\maxle{\less}$ avoids the pattern~$\up{2}13$ if and only if~${\less} \in \IPIPp$. The proof for the other patterns is~similar. 
\end{proof}

\begin{corollary}
\label{coro:characterizationPIP}
A poset~$\less$ is in~$\PIP$ if and only if it is in~$\WOIP(n)$ (see characterization in Proposition~\ref{prop:characterizationWOIP}) and satisfies the conditions of $\IPIPp$, $\IPIPm$, $\DPIPp$ and~$\DPIPm$.
\end{corollary}

\begin{remark}
Similarly to Remark~\ref{rem:coverEnough}, note that it suffices to check these conditions only for all cover relations~$a \less c$ and~$a \more c$ in~$\less$.
\end{remark}

Some illustrations are given in \fref{fig:pip}. The leftmost poset is not in~$\PIP$: $\{1,2,3\}$ does not satisfy $\IPIPm$, $\{2,3,5\}$ does not satisfy $\IPIPp$, $\{3,4,6\}$ does not satisfy $\DWOIP(6)$, and $\{3,5,6\}$ does not satisfy $\DPIPm$. The other two posets of \fref{fig:pip} are both in~$\PIP$.

\begin{figure}[ht]
	\centerline{
    \begin{tabular}{c@{\qquad}c@{\qquad}c}
        \scalebox{0.8}{\input{relations/pip_counter_example.tex}}
        &
        \scalebox{0.8}{\input{relations/pip_example.tex}}
        &
        \scalebox{0.8}{\input{relations/pep_example.tex}}
        \\[-.5cm]
        $\notin \PIP$
        &
        $\in \PIP \ssm \PEP$
        &
    $\in \PEP$
    \end{tabular}
}
	\caption{Examples and counterexamples of elements in $\PIP$ and $\PEP$, where~$\orientation = (6, \{2,3\}, \{2,5\})$.}
\label{fig:pip}
	\vspace{-.4cm}
\end{figure}

Now we describe the weak order on~$\PIP$.

\begin{proposition}
\label{prop:weakOrderPIP}
For any~$\tree[S] \wole \tree[S]'$ and~$\tree \wole \tree'$, we have~${\less_{[\tree[S], \tree[S]']}} \wole {\less_{[\tree, \tree']}} \iff \tree[S] \wole \tree$ and~$\tree[S]' \wole \tree'$.
\end{proposition}

\begin{proof}
The proof is similar to that of Proposition~\ref{prop:weakOrderWOIP}.
\end{proof}

We immediately derive that~$(\PIP, \wole)$ has the lattice structure of a product.

\begin{corollary}
\label{coro:PIPLattice}
The weak order~$(\PIP, \wole)$ is a lattice whose meet and join are given by
\[
{\less_{[\tree[S],\tree[S]']}} \meetPIP {\less_{[\tree,\tree']}} = {\less_{[\tree[S] \meetO \tree, \, \tree[S]' \meetO \tree']}}
\qquad\text{and}\qquad
{\less_{[\tree[S],\tree[S]']}} \joinPIP {\less_{[\tree,\tree']}} = {\less_{[\tree[S] \joinO \tree, \, \tree[S]' \joinO \tree']}}.
\]
\end{corollary}

\begin{remark}
\label{rem:WOIPnotSublattice}
As illustrated by~$\WOIP(n)$, the set~$\PIP$ does not always induce a sublattice of~${(\IPos(n), \wole, \meetT, \joinT)}$. Theorem~\ref{thm:coveringOrientationConflictFunctionsSublatticesIPos} will provide a sufficient condition on the orientation~$\orientation$ for this property. In contrast, we will see in Theorem~\ref{thm:orientationConflictFunctionsSublatticesWOIP} that~$\PIP$ always induces a sublattice of~$(\WOIP(n), \wole, \meetWOIP, \joinWOIP)$.
\end{remark}


\subsection{Characterization of~$\PEP$}
\label{subsec:characterizationPEP}

We are now ready to give a characterization of the posets of~$\PEP$ left open in Section~\ref{subsec:PEP}. We need one additional definition. For an orientation~$\orientation$ of~$[n]$, an \defn{$\orientation$-snake} in a poset~$\less$ is a sequence~$x_0 < x_1 < \dots < x_k < x_{k+1}$ such that
\begin{itemize}
\item either~$x_0 \less x_1 \more x_2 \less x_3 \more \cdots$ with~$\set{x_i}{i \in [k] \text{ odd}} \subseteq \orientation\negative$ and $\set{x_i}{i \in [k] \text{ even}} \subseteq \orientation\positive$, 
\item or \quad\; $x_0 \more x_1 \less x_2 \more x_3 \less \cdots$ with~$\set{x_i}{i \in [k] \text{ odd}} \subseteq \orientation\positive$ and $\set{x_i}{i \in [k] \text{ even}} \subseteq \orientation\negative$,
\end{itemize}
as illustrated in \fref{fig:snakes}.

\begin{figure}[h]
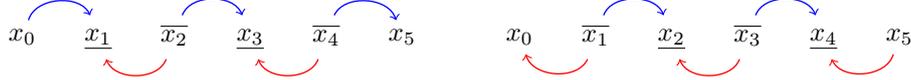

	\centerline{\scalebox{1}{\input{relations/snake1.tex}} \hspace{.6cm} \scalebox{1}{\input{relations/snake2.tex}}}
	\vspace{-.3cm}
	\caption{Two $\orientation$-snakes joining~$x_0$ to~$x_5$. The set~$\orientation\positive$ (resp.~$\orientation\negative$) must contain at least the overlined (resp.~underlined) integers.}
	\label{fig:snakes}
	\vspace{-.4cm}
\end{figure}

\noindent
We say that the $\orientation$-snake~$x_0 < x_1 < \dots < x_k < x_{k+1}$ joins~$x_0$ to~$x_{k+1}$ and has length~$k$. Note that, by definition, we consider the relations~$x \less y$ or~$x \more y$ themselves as (degenerate, length~$0$) $\orientation$-snakes between~$x$ and~$y$.

\begin{proposition}
\label{prop:characterizationPEP}
A poset~${\less} \in \IPos(n)$ is in~$\PEP$ if and only if it is in~$\PIP$ (see characterization in Corollary~\ref{coro:characterizationPIP}) and it admits an $\orientation$-snake between any two values of~$[n]$.
\end{proposition}

\fref{fig:pip} illustrates this proposition: the middle poset is not in~$\PEP$, since there is no $\orientation$-snake between~$1$ and~$4$ nor between~$1$ and~$6$. In contrast, the rightmost poset is in~$\PEP$, as $1 \less \down{2} \more 4$ is an $\orientation$-snake between~$1$ and~$4$ and $1 \less \down{5} \more 6$ is an $\orientation$-snake between~$1$ and~$6$.

\begin{proof}[Proof or Proposition~\ref{prop:characterizationPEP}]
Assume that~${\less} \in \PEP$, and let~$\tree \in \Permutrees$ be the permutree such that~${\less} = {\less_{\tree}}$. Then $\less$ is certainly in~$\PIP$. Now any two values~$x,y \in [n]$ are connected by a (non-oriented) path in~$\tree$, and recalling the local optima along this path provides an $\orientation$-snake joining~$x$ and~$y$.

Reciprocally, consider ${\less} \in \PIP$ such that there is an $\orientation$-snake between any two values of~$[n]$. We need the following two intermediate claims, proved in detail in Appendix~\ref{subsec:appendixCharacterizationPEP}.

\vspace{-.1cm}
\begin{claim}
\label{claim:snake1}
For any~$u,v,w \in [n]$ with~$u < w$,
\begin{itemize}
\item if~$u \less v$ and~$v \more w$ are cover relations of~$\less$, then~$u < v < w$ and~$v \in \orientation\negative$;
\item if~$u \more v$ and~$v \less w$ are cover relations of~$\less$, then~$u < v < w$ and~$v \in \orientation\positive$.
\end{itemize} 
\end{claim}

\vspace{-.3cm}
\begin{claim}
\label{claim:snake2}
Let~$x_0, \dots, x_p \in [n]$ be a path in the Hasse diagram of~$\less$ (\ie~$x_{i-1} \less x_i$ or~$x_{i-1} \more x_i$ are cover relations in~$\less$ for any~$i \in [p]$). Assume moreover that~$x_0 \in \orientation\negative$ and~$x_0 \more x_1$, or that~$x_0 \in \orientation\positive$ and~$x_0 \less x_1$. Then all $x_i$ are on the same side of~$x_0$, \ie~$x_0 < x_1 \iff x_0 < x_i$ for all~$i \in [p]$.
\end{claim}
\vspace{-.1cm}

\noindent
Claims~\ref{claim:snake1} and~\ref{claim:snake2} show that the Hasse diagram of~$\less$ is~a~permutree:
\begin{itemize}
\item it is connected since any two values are connected by a snake,
\item it cannot contain a cycle (otherwise, since this cycle cannot be oriented, there exist three distinct vertices~$u,v,w$ in this cycle with~$u \less v \more w$. Claim~\ref{claim:snake1} ensures that~$u < v < w$ and~$v \in \orientation\negative$. Since there is a path~$v = x_0, w = x_1, x_2, \dots, x_p = u$ in the Hasse diagram of~$\less$ with~$v \in \orientation\negative$ and~$v \more w$, Claim~\ref{claim:snake2} affirms that~$u$ and~$w$ are on the same side of~$v$, a contradiction), and 
\item it fulfills the local conditions of Definition~\ref{def:permutree} to be a permutree (Claim~\ref{claim:snake1} shows Condition~(i) of Definition~\ref{def:permutree}, and Claim~\ref{claim:snake2} shows Condition~(ii) of Definition~\ref{def:permutree}). \qedhere
\end{itemize}
\end{proof}

For further purposes, we will need the following lemma to check the existence of $\orientation$-snakes.

\begin{lemma}
\label{lem:orientationIncompCharacterization}
Let~${\less} \in \IPos(n)$ and~$a < c$ be incomparable in~$\less$. The following are equivalent:
\begin{enumerate}[(i)]
\item There is an $\orientation$-snake between~$a$ and~$c$,
\item $\exists \; a < b < c$ such that there is an $\orientation$-snake between~$a$ and~$b$, and either~$b \in \orientation\negative$ and $b \more c$, or $b \in \orientation\positive$~and~$b \less c$,
\item $\exists \; a < b < c$ such that there is an $\orientation$-snake between~$b$ and~$c$, and either~$b \in \orientation\negative$ and $a \less b$, or $b \in \orientation\positive$~and~$a \more b$.
\end{enumerate}
\end{lemma}

\begin{proof}
The implication~(i) $\Rightarrow$ (ii) is immediate, considering~$b = x_k$. Assume now that~$\less$ and~$\{a,c\}$ satisfy~(ii). Let~$b$ be given by~(ii) and let~$a < x_1 < \dots < x_k < b$ be an $\orientation$-snake between~$a$ and~$b$. If~$x_k \less b \more c$ (or similarly if~$x_k \more b \less c$), then~$a < x_1 < \dots < b < c$ is a $\orientation$-snake between~$a$ and~$c$. In contrast, if~$x_k \less b \less c$ (or similarly if~$x_k \more b \more c$), then~$x_k \less c$ (resp.~$x_k \more c$) by transitivity of~$\less$, so that~$a < x_1 < \dots < x_k < c$ is an $\orientation$-snake between~$a$ and~$c$. Therefore, (i) $\iff$ (ii). The proof of (i) $\iff$ (iii) is identical.
\end{proof}


\subsection{Permutree Face Posets}
\label{subsec:PFP}

The permutrees of~$\Permutrees$ correspond to the vertices of the \defn{\mbox{$\orientation$-permutreehedron}}~$\PT$ constructed in~\cite{PilaudPons}. The precise definition of these polytopes is not needed here. Following~\fref{fig:permutrees}, we illustrate in \fref{fig:permutreehedra} classical polytopes that arise as permutreehedra for specific orientations:

\medskip
\centerline{
\begin{tabular}{c||c|c|c|c}
orientation~$(n, \orientation\positive, \orientation\negative)$ & $(n, \varnothing, \varnothing)$ & $(n, \varnothing, [n])$ & $\orientation\positive \sqcup \orientation\negative = [n]$ & $(n, [n], [n])$ \\
\hline
\multirow{2}{*}{permutreehedron} & \multirow{2}{*}{permutahedron} & Loday's & Hohlweg-Lange's & \multirow{2}{*}{parallelepiped} \\
& & associahedron~\cite{Loday} & associahedra~\cite{HohlwegLange, LangePilaud} &
\end{tabular}
}

\begin{figure}[h]
	\centerline{\includegraphics[width=1.1\textwidth]{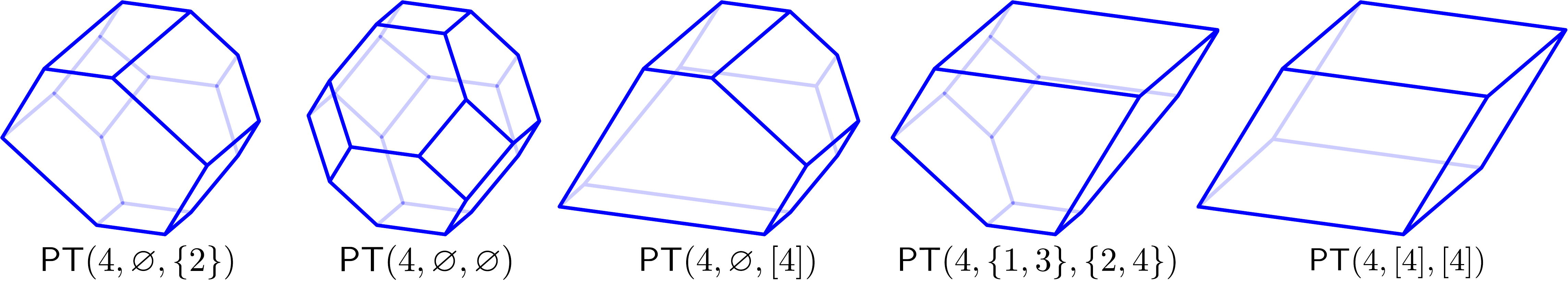}}
	\vspace{-.1cm}
	\caption{Five examples of permutreehedra. While the first is generic, the last four are a permutahedron, some associahedra~\cite{Loday, HohlwegLange, LangePilaud}, and a parallelepiped.}
	\label{fig:permutreehedra}
	\vspace{-.4cm}
\end{figure}

We now consider all the faces of the $\orientation$-permutreehedron. As shown in~\cite{PilaudPons}, they correspond to \defn{Schr\"oder $\orientation$-permutrees}, defined as follows.

\begin{definition}[\cite{PilaudPons}]
\label{def:SchroderPermutree}
For an orientation~$\orientation$ on~$[n]$ and a subset~$X \subseteq [n]$, we let~$X\negative \eqdef X \cap \orientation\negative$ and~$X\positive \eqdef X \cap \orientation\positive$. A \defn{Schr\"oder $\orientation$-permutree} is a directed tree~$\tree[S]$ with vertex set~$\ground$ endowed with a vertex labeling~$p : \ground \to 2^{[n]} \ssm \varnothing$ such that
\begin{enumerate}[(i)]
\item the labels of~$\tree[S]$ partition~$[n]$, \ie $v \ne w \in \ground \Longrightarrow p(v) \cap p(w) = \varnothing$ and~$\bigcup_{v \in \ground} p(v) = [n]$;
\item each vertex~$v \in \ground$ has one incoming (resp.~outgoing) subtree~$\tree[S]_v^I$ (resp.~$\tree[S]^v_I$) for each interval~$I$ of~$[n] \ssm p(v)\negative$ (resp.~of~$[n] \ssm p(v)\positive$) and all labels of~$\tree[S]_v^I$ (resp.~of~$\tree[S]^v_I$) are subsets of~$I$.
\end{enumerate}
We denote by~$\SchroderPermutrees$ the set of Schr\"oder $\orientation$-permutrees.
\end{definition}

For example, in the leftmost Schr\"oder permutree of \fref{fig:SchroderPermutrees}, the vertices are labeled by the sets~$\{1,2\}$, $\{3\}$, $\{4,6\}$, $\{5\}$, and $\{6, 7\}$. The vertex~$v$ labeled by $p(v) = \{4,6\}$ has $3$ incoming subtrees included in the $3$ intervals of~$[n] \ssm p(v)\negative = [n] \ssm \{4,6\} = \{1,2,3\} \sqcup \{5\} \sqcup \{7\}$ and $2$ (empty) outgoing subtrees included in the $2$ intervals of~$[n] \ssm p(v)\positive = [n] \ssm \{4\} = \{1,2,3\} \sqcup \{5,6,7\}$.

\begin{figure}[t]
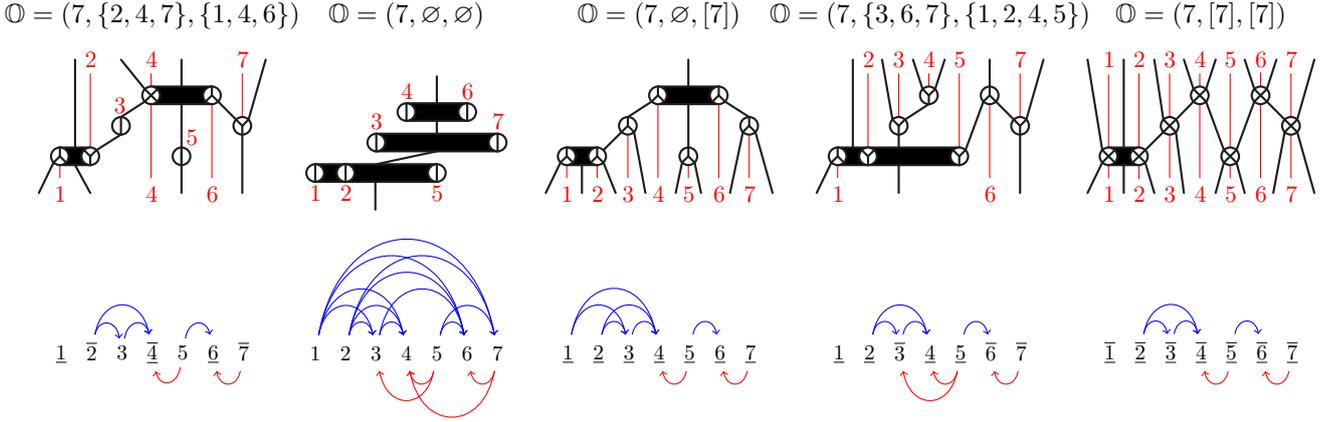

	\centerline{
	\begin{tabular}{c@{\!}cc@{\!}c@{\!}c}
		  $\orientation = (7, \{2,4,7\}, \{1,4,6\})$
		& $\orientation = (7, \varnothing, \varnothing)$
		& $\orientation = (7, \varnothing, [7])$
		& $\orientation = (7, \{3,6,7\}, \{1,2,4,5\})$
		& $\orientation = (7, [7], [7])$
		\\[.2cm]
		  \includegraphics[scale=.8]{SchroderPermutreeGeneric}
		& \includegraphics[scale=.8]{SchroderPermutreeOrderedPartition}
		& \includegraphics[scale=.8]{SchroderPermutreeSchroderTree}
		& \includegraphics[scale=.8]{SchroderPermutreeCambrianTree}
		& \includegraphics[scale=.8]{SchroderPermutreeTernarySequence}
		\\[-.2cm]
		  \;\scalebox{.8}{\input{relations/posetPFP.tex}}
		& \;\scalebox{.8}{\input{relations/posetWOFP.tex}}
		& \;\scalebox{.8}{\input{relations/posetTOFP.tex}}
		& \;\scalebox{.8}{\input{relations/posetCOFP.tex}}
		& \;\scalebox{.8}{\input{relations/posetBOFP.tex}}
	\end{tabular}
}
	\vspace{-1.3cm}
	\caption{Five examples of Schr\"oder permutrees~$\tree[S]$ (top) with their posets~$\less_{\tree[S]}$~(bottom). While the first is generic, the last four illustrate specific orientations corresponding to ordered partitions, Schr\"oder trees, Schr\"oder Cambrian trees, and ternary sequences.}
	\label{fig:SchroderPermutrees}
	\vspace{-.4cm}
\end{figure}

Following~\fref{fig:permutrees}, we have represented in \fref{fig:SchroderPermutrees} five Schr\"oder permutrees, where the last four encode relevant combinatoiral objects obtained for specific orientations:

\medskip
\centerline{
\begin{tabular}{c||c|c|c|c}
orientation~$(n, \orientation\positive, \orientation\negative)$ & $(n, \varnothing, \varnothing)$ & $(n, \varnothing, [n])$ & $\orientation\positive \sqcup \orientation\negative = [n]$ & $(n, [n], [n])$ \\
\hline
\multirow{2}{*}{combinatorial objects} & \multirow{2}{*}{ordered partitions} & \multirow{2}{*}{Schr\"oder trees} & Schr\"oder Cambrian & ternary \\
& & & trees~\cite{ChatelPilaud} & sequences
\end{tabular}
}

\medskip
\noindent
We refer again to~\cite{PilaudPons} for more details on the interpretation of these combinatorial objects as permutrees, and we still use the drawing conventions of~\cite{PilaudPons}.

An $\orientation$-permutree~$\tree$ belongs to a face of the permutreehedron~$\PT$ corresponding to a Schr\"oder $\orientation$-permutree~$\tree[S]$ if and only if~$\tree[S]$ is obtained by edge contractions in~$\tree$. The set of such $\orientation$-permutrees is the interval~$[\tree^{\min}(\tree[S]), \tree^{\max}(\tree[S])]$ of the $\orientation$-permutree lattice, where the minimal (resp.~maximal) tree~$\tree^{\min}(\tree[S])$ (resp.~$\tree^{\max}(\tree[S])$) is obtained by replacing the nodes of~$\tree[S]$ by left (resp.~right) combs as illustrated in \fref{fig:pfp}. To be more precise, we need additional notations. For an interval~$I = [i,j]$ of integers, define~$\bar I \eqdef [i-1, j+1]$. For each edge~$v \to w$ in~$\tree[S]$, we let~$I_v^w$ (resp.~$J_v^w$) denote the interval of~$[n] \ssm p(v)\positive$ (resp.~of~$[n] \ssm p(w)\negative$) such that $p(w) \subseteq I_v^w$ (resp.~$p(v) \subseteq J_v^w$). The minimal and maximal permutrees in the face corresponding to the Schr\"oder permutree~$\tree[S]$ are then described as follows:
\begin{enumerate}[(i)]
\item $\tree^{\min}(\tree[S])$ is the $\orientation$-permutree obtained from the left combs on the subsets~$p(v)$ for~$v$ vertex of~$\tree[S]$ by adding the edges connecting~$\max(p(v) \cap \bar I_v^w)$ with~$\min(p(w) \cap \bar J_v^w)$ for all edges~$v \to w$~in~$\tree[S]$.
\item $\tree^{\max}(\tree[S])$ is the $\orientation$-permutree obtained from the right combs on the subsets~$p(v)$ for~$v$ \mbox{vertex~of~$\tree[S]$} by adding the edges connecting~$\min(p(v) \cap \bar I_v^w)$ with~$\max(p(w) \cap \bar J_v^w)$ for all edges~${v \to w}$~in~$\tree[S]$.
\end{enumerate}

For example, consider the edge $v = \{5\} \to \{4, 6\} = w$ in the Schr\"oder permutree of \fref{fig:pfp}. We have $I_v^w = [n]$ and~$J_v^w = \{5\}$. In $\tree^{\min}(\tree[S])$, we create the left comb $4 \to 6$ and we add the edge $5 = \max(\{5\} \cap [n]) \to \min(\{4, 6\} \cap \{4,5,6\}) = 4$. Similarly, in $\tree^{\max}(\tree[S])$, we create the right comb $6 \to 4$ and we add the edge $5 = \min(\{5\} \cap [n]) \to \max(\{4, 6\} \cap \{4,5,6\}) = 6$.

\begin{figure}[h]
	\vspace{-.2cm}
	\input{example_pfp}
	\vspace{-.5cm}
	\caption{A Permutree Face Poset ($\PFP$).}
	\label{fig:pfp}
	\vspace{-.4cm}
\end{figure}

For a Schr\"oder permutree~$\tree[S]$, we define~${\less_{\tree[S]}} = {\less_{[\tree^{\min}(\tree[S]), \tree^{\max}(\tree[S])]}}$. Examples are given in \fref{fig:SchroderPermutrees}. One easily checks that $\less_{\tree[S]}$ could also be defined as the transitive closure of all relations~$i \less_{\tree[S]} j$ for all~$i \in p(v) \cap \bar I_v^w$ and~$j \in p(w) \cap \bar J_v^w$ for all edges~$v \to w$ in~$\tree[S]$. For edge $\{5\} \to \{4,6\}$ in the Schr\"oder permutree of \fref{fig:pfp}, this corresponds to the relations $4 \more 5 \less 6$ of the poset. Note that
\begin{itemize}
\item an $\orientation$-permutree~$\tree$ belongs to the face of the permutreehedron~$\PT$ corresponding to a Schr\"oder $\orientation$-permutree~$\tree[S]$ if and only if~$\less_{\tree}$ is an extension of~$\less_{\tree[S]}$, and
\item the linear extensions of~$\less_{\tree[S]}$ are precisely the linear extensions of~$\less_{\tree}$ for all $\orientation$-permutrees~$\tree$ which belong to the face of the permutreehedron~$\PT$ corresponding to~$\tree[S]$.
\end{itemize}

We say that~$\less_{\tree[S]}$ is a \defn{permutree face poset}, and we denote by
\[
\PFP \eqdef \bigset{{\less_{\tree[S]}}}{\tree[S] \in \SchroderPermutrees}
\]
the set of all permutree face posets. We now characterize these posets.

\begin{proposition}
\label{prop:characterizationPFP}
A poset~${\less} \in \IPos(n)$ is in~$\PFP$ if and only if it is in~$\PIP$ and for any~$a < c$ incomparable in~$\less$,
\[
\begin{array}{cccccc@{\qquad}c}
& \exists \; a < b < c & \text{such that} & b \in \orientation\positive \text{ and } a \not\less b \not\more c & \text{or} & b \in \orientation\negative \text{ and } a \not\more b \not\less c, \qquad & (\spadesuit) \\
\text{ or} \qquad 
& \forall \; a < b < c & \text{we have} & a \less b \iff b \more c & \text{and} & a \more b \iff b \less c. & (\clubsuit)
\end{array}
\]
\end{proposition}

This property is illustrated on the poset of \fref{fig:pfp}. For example, $1$~and~$2$ are neighbors and thus satisfy~$(\clubsuit)$, $1$~and~$3$ satisfy~$(\spadesuit)$ with~$b = 2$, $4$~and~$6$ satisfy~$(\clubsuit)$, etc.

\begin{proof}[Proof of Proposition~\ref{prop:characterizationPFP}]
Assume that~${\less} \in \PFP$, and consider the Schr\"oder permutree~$\tree[S]$ such that~${\less} = {\less_{\tree[S]}}$. Then~${{\less} = {{\less_{\tree[S]}} = {\less_{[\tree^{\min}(\tree[S]), \tree^{\max}(\tree[S])]}}}}$ belongs to~$\PIP$. Moreover, any~$a < c$ incomparable in~$\less_{\tree[S]}$
\begin{enumerate}[(i)]
\item satisfy~$(\spadesuit)$ when either~$a, c$ belong to distinct vertices of~$\tree[S]$ separated by a wall, or~$a, c$ belong to the same vertex~$v$ of~$\tree[S]$ but there is another~$b$ in~$v$ with~$b \in \orientation\positive \cup \orientation\negative$ and~$a < b < c$,
\item satisfy~$(\clubsuit)$ when $a, c$ belong to the same vertex~$v$ of~$\tree[S]$ and~$b \notin \orientation\positive \cup \orientation\negative$ for any~${a < b < c}$~in~$v$.
\end{enumerate}
This shows one implication of the statement. Before proving the reciprocal implication, let us comment a little more to give some useful intuition. Note that two consecutive elements~$a < c$ in a vertex~$v$ of~$\tree[S]$ satisfy~$(\clubsuit)$ and not~$(\spadesuit)$. In particular, if all~$a < c$ incomparable in~$\less_{\tree[S]}$ satisfy~$(\spadesuit)$, then~$\tree[S]$ is just a permutree. In general, the posets~$\less_{\tree^{\min}(\tree[S])}$ and~$\less_{\tree^{\min}(\tree[S])}$ corresponding to the minimal and maximal $\orientation$-permutrees in the face corresponding to~$\tree[S]$ are given by
\begin{align*}
{\less_{\tree^{\min}(\tree[S])}} & = \tc{\big( {\less} \cup \set{(a,c)}{a < c \text{ incomparable in~$\less_{\tree[S]}$ and not satisfying~$(\spadesuit)$}} \big)}
\\ \text{and}\qquad
{\less_{\tree^{\max}(\tree[S])}} & = \tc{\big( {\less} \cup \set{(c,a)}{a < c \text{ incomparable in~$\less_{\tree[S]}$ and not satisfying~$(\spadesuit)$}} \big)}.
\end{align*}

\medskip
Consider now an arbitrary poset~${\less} \in \PIP$ such that any~$a < c$ incomparable in~$\less$ satisfy~$(\spadesuit)$ or~$(\clubsuit)$. The previous observation motivates the following claim (see Appendix~\ref{subsec:appendixCharacterizationPFP} for the proof).

\vspace{-.05cm}
\begin{claim}
\label{claim:PEP}
If any~$a < c$ incomparable in~$\less$ satisfy~$(\spadesuit)$, then~${\less} \in \PEP \subset \PFP$.
\end{claim}
\vspace{-.1cm}

\noindent
Suppose now that some~$a < c$ incomparable in~$\less$ do not satisfy~$(\spadesuit)$. The idea of our proof is to return to the previous claim by considering the auxiliary poset 
\[
{\bless} \eqdef \tc{\big( {\less} \cup \set{(a,c)}{a < c \text{ incomparable in~$\less$ and not satisfying~$(\spadesuit)$}}\big)}.
\]

\vspace{-.15cm}
\begin{claim}
\label{claim:decontractionCompare}
We have~${\Inc{\less}} \subseteq {\Inc{\bless}}$ and~${\Dec{\less}} = {\Dec{\bless}}$.
\end{claim}

\vspace{-.25cm}
\begin{claim}
\label{claim:decontraction}
If~${\less} \in \PIP$ and any~$a < c$ incomparable in~$\less$ satisfy~$(\spadesuit)$ or~$(\clubsuit)$, then~${\bless} \in \PIP$ and any~$a < c$ incomparable in~$\bless$ satisfy~$(\spadesuit)$.
\end{claim}
\vspace{-.1cm}

\noindent
Combining Claims~\ref{claim:PEP} and~\ref{claim:decontraction}, we obtain that there exists a permutree~$\tree$ such that~${\bless} = {\less_{\tree}}$. Intuitively, $\tree$ is the minimal permutree in the face that will correspond to~$\less$. To find the Schr\"oder permutree of this face, we thus just need to contract some edges in~$\tree$. We therefore consider the Schr\"oder permutree~$\tree[S]$ obtained from~$\tree$ by contracting all edges that appear in the Hasse diagram of~$\bless$ but are not in~$\less$.

\vspace{-.1cm}
\begin{claim}
\label{claim:contraction}
We have ${\less} = {\less_{\tree[S]}}$, so that~${\less} \in \PFP$.
\end{claim}
\vspace{-.1cm}

\noindent
The detailed proofs of Claims~\ref{claim:PEP} to~\ref{claim:contraction} are given in Appendix~\ref{subsec:appendixCharacterizationPFP}. This concludes the proof of Proposition~\ref{prop:characterizationPFP}.
\end{proof}

We now consider the weak order on~$\PFP$. Let us first recall from~\cite{PilaudPons} the definition of the Schr\"oder permutree lattice.

\begin{proposition}[\cite{PilaudPons}]
\label{prop:SchroderPermutrees}
Fix an orientation~$\orientation = (n, \orientation\positive, \orientation\negative)$ of~$[n]$.
\\[-.3cm]
\begin{enumerate}
\item Each Schr\"oder $\orientation$-permutree corresponds to a face of the permutreehedron~$\PT$, and thus to a cone of its normal fan. Moreover, the normal fan of the permutahedron~$\Perm$ refines that of the permutreehedron~$\PT$. This defines a surjection~$\surjection$ from the set of ordered partitions of~$[n]$ to the set of Schr\"oder permutrees of~$\SchroderPermutrees$, which sends an ordered partition~$\pi$ to the unique Schr\"oder permutree~$\tree[S]$ satisfying that the interior of the normal cone of the face of~$\PT$ corresponding to~$\tree[S]$ contains the interior of the normal cone of the face of~$\Perm$ corresponding~to~$\pi$.
\\[-.2cm]
\item The fibers of this surjection~$\surjection$ define a lattice congruence of the facial weak order discussed in Section~\ref{subsec:WOFP} (see~\cite{PilaudPons} for details). Therefore, the set of Schr\"oder permutrees~$\SchroderPermutrees$ is endowed with a lattice structure~$\wole$, called \defn{Schr\"oder permutree lattice}, defined by
\[
\tree[S] \wole \tree[S]' \iff \exists \; \pi, \pi' \text{ such that } \surjection(\pi) = \tree[S], \; \surjection(\pi') = \tree[S]' \text{ and } \pi \wole \pi' \\
\]
\item The contraction of an edge~$e = v \to w$ in a Schr\"oder permutree $\tree[S]$ is called \defn{increasing} if~${\max(p(v)) < \min(p(w))}$ and \defn{decreasing} if~$\max(p(w)) < \min(p(v))$. The Schr\"oder permutree lattice is the transitive closure of the relations~$\tree[S] \prec \tree[S]/e$ (resp.~$\tree[S]/e \prec \tree[S]$) for any Schr\"oder permutree~$\tree[S]$ and edge~$e \in \tree[S]$ defining an increasing (resp.~decreasing) contraction.
\end{enumerate}
\end{proposition}

Now we establish the relationship between the permutree order on~$\Permutrees$ and the weak order on~$\PEP$.

\begin{proposition}
\label{prop:weakOrderPFP}
For any Schr\"oder permutrees~$\tree[S], \tree[S]' \in \SchroderPermutrees$, we have~$\tree[S] \wole \tree[S]'$ in the Schr\"oder permutree lattice if and only if~${\less_{\tree[S]}} \wole {\less_{\tree[S]'}}$ in the weak order on posets.
\end{proposition}

\begin{proof}
We can identify the Schr\"oder $\orientation$-permutree~$\tree[S]$ with:
\begin{enumerate}[(i)]
\item the interval~$[\tree^{\min}(\tree[S]), \tree^{\max}(\tree[S])]$ of $\orientation$-permutrees that belong to the face of~$\PT$ given by~$\tree[S]$,
\item the interval~$[\pi^{\min}(\tree[S]), \pi^{\max}(\tree[S])]$ of ordered partitions~$\pi$ such that the interior of the normal cone of~$\Perm$ corresponding to~$\pi$ is included in the interior of the normal cone of~$\PT$ corresponding~to~$\tree[S]$,
\item the interval~$[\sigma^{\min}(\tree[S]), \sigma^{\max}(\tree[S])]$ of~$\fS(n)$ between the minimal and maximal extensions of~$\less_{\tree[S]}$.
\end{enumerate}
It is immediate to check that~$\sigma^{\min}(\tree[S])$ is the minimal linear extension of~$\tree^{\min}(\tree[S])$ and of~$\pi^{\min}(\tree[S])$ and that~$\sigma^{\max}(\tree[S])$ is the maximal linear extension of~$\tree^{\max}(\tree[S])$ and of~$\pi^{\max}(\tree[S])$. We conclude that
\begin{align*}
\tree[S] \wole \tree[S]'
& \iff \pi^{\min}(\tree[S]) \wole \pi^{\min}(\tree[S]') \text{ and } \pi^{\max}(\tree[S]) \wole \pi^{\max}(\tree[S]') \\
& \iff \sigma^{\min}(\tree[S]) \wole \sigma^{\min}(\tree[S]') \text{ and } \sigma^{\max}(\tree[S]) \wole \sigma^{\max}(\tree[S]') \\
& \iff \tree^{\min}(\tree[S]) \wole \tree^{\min}(\tree[S]') \text{ and } \tree^{\max}(\tree[S]) \wole \tree^{\max}(\tree[S]')
\iff {\less_{\tree[S]}} \wole {\less_{\tree[S]'}}.
\end{align*}
The first line holds by definition of the Schr\"oder permutree lattice (as~$\surjection^{-1}(\tree[S]) = [\pi^{\min}(\tree[S]), \pi^{\max}(\tree[S])]$), while the last holds by definition of the weak order on~$\PFP$ (as~${\less_{\tree[S]}} = {\less_{[\tree^{\min}(\tree[S]), \tree^{\max}(\tree[S])]}}$).
\end{proof}

\begin{remark}
\label{rem:PFPnotSublattice}
Although the weak order on~$\PFP$ is a lattice, the example of Remark~\ref{rem:WOFPnotSublattice} shows that it is not a sublattice of~$(\IPos(n), \wole, \meetT, \joinT)$, nor a sublattice of~$(\WOIP(n), \wole, \meetWOIP, \joinWOIP)$, nor a sublattice of~$(\PIP, \wole, \meetPIP, \joinPIP)$. We will discuss an alternative description of the meet and join in~$\PFP$ in Section~\ref{subsec:facesSublattices}.
\end{remark}


\subsection{$\PIP$ deletion}
\label{subsec:PIPdeletion}

Similar to the projection maps of Sections~\ref{subsec:IWOIPDWOIP} and~\ref{subsec:TOIPdeletion}, we define the \defn{$\IPIPp$} (resp.~\defn{$\IPIPm$}, \defn{$\IPIPpm$}, \defn{$\IPIP$}) \defn{increasing deletion} by
\[
\begin{array}{l@{\;}c@{\;}l}
{\IPIPpid{\less}} & \eqdef & {\less} \ssm \set{(a,c)}{\exists \; a < b < c, \; b \in \orientation\positive \text{ and } a \not\less b}, \\[.1cm]
{\IPIPmid{\less}} & \eqdef & {\less} \ssm \set{(a,c)}{\exists \; a < b < c, \; b \in \orientation\negative \text{ and } b \not\less c}, \\[.1cm]
{\IPIPpmid{\less}} & \eqdef & {\less} \ssm \set{(a,c)}{\exists \; a \le n < p \le c \text{ with } n \in \{a\} \cup \orientation\negative \text{ while } p \in \{c\} \cup \orientation\positive \text{ and } n \not\less p}, \\[.1cm]
{\IPIPid{\less}} & \eqdef & {\IPIPpmid{(\IWOIPid{\less})}},
\end{array}
\]
and similarly the \defn{$\DPIPp$} (resp.~\defn{$\DPIPm$}, \defn{$\DPIPpm$}, \defn{$\DPIP$}) \defn{decreasing deletion} by
\[
\begin{array}{l@{\;}c@{\;}l}
{\DPIPpdd{\less}} & \eqdef & {\less} \ssm \set{(c,a)}{\exists \; a < b < c, \; b \in \orientation\positive \text{ and } b \not\more c}, \\[.1cm]
{\DPIPmdd{\less}} & \eqdef & {\less} \ssm \set{(c,a)}{\exists \; a < b < c, \; b \in \orientation\negative \text{ and } a \not\more b}, \\[.1cm]
{\DPIPpmdd{\less}} & \eqdef & {\less} \ssm \set{(c,a)}{\exists \; a \le p < n \le c \text{ with } p \in \{a\} \cup \orientation\positive \text{ while } n \in \{c\} \cup \orientation\negative \text{ and } p \not\more n}, \\[.1cm]
{\DPIPdd{\less}} & \eqdef & {\DPIPpmdd{(\DWOIPdd{\less})}}.
\end{array}
\]
These operations are illustrated on \fref{fig:IPIPpmid/DPIPpmdd/PIPd}.

\begin{figure}[h]
	\vspace{-.5cm}
	\centerline{
	\begin{tabular}{c@{\quad}c@{\quad}c}
		\multirow{ 2}{*}{${\less} = \raisebox{-1.55cm}{\scalebox{0.8}{\input{relations/poset10}}}$} &
		${\IPIPpmid{\less}} = \raisebox{-1.55cm}{\scalebox{0.8}{\input{relations/poset11}}}$ &
		\multirow{ 2}{*}{${\PIPd{\less}} = \raisebox{-1.55cm}{\scalebox{0.8}{\input{relations/posetPIPlarge}}}$} \\
		& ${\DPIPpmdd{\less}} = \raisebox{-1.55cm}{\scalebox{0.8}{\input{relations/poset12}}}$
	\end{tabular}
}
	\vspace{-.5cm}
	\caption{The $\IPIPpm$ increasing deletion, the $\DPIPpm$ decreasing deletion, and the $\PIP$ deletion.}
	\label{fig:IPIPpmid/DPIPpmdd/PIPd}
	\vspace{-.4cm}
\end{figure}

\begin{remark}
\label{rem:IPIPplidDPIPpmdd}
Similar to Remarks~\ref{rem:tdd} and~\ref{rem:IWOIPidDWOIPdd}, for any~$\varepsilon \in \{\varnothing, -, +, \pm\}$, the $\IPIPe$ increasing deletion (resp.~$\DPIPe$ decreasing deletion) deletes at once all increasing relations which prevent the poset to be in~$\IPIPe$ (resp.~in~$\DPIPe$). Note that we have
\[
{\IPIPpid{\less}} = {\IPIPpmid{\less}[\orientation\positive, \varnothing]} \subseteq {\IPIPpmid{\less}}
\qquad\text{and}\qquad
{\IPIPmid{\less}} = {\IPIPpmid{\less}[\varnothing, \orientation\negative]} \subseteq {\IPIPpmid{\less}}.
\]
However, we do not necessarily have~${\IPIPpmid{\less}} = {\IPIPpid{\less}} \cap {\IPIPmid{\less}}$. Consider for example the poset~${\less} \eqdef \{(1,3), (2,4), (1,4)\}$ and the orientation~$(4, \{3\},\{2\})$. Then~${{\IPIPpid{\less}} = \{(2,4), (1,4)\}}$, ${\IPIPmid{\less}} = \{(1,3), (1,4)\}$ so that~${\IPIPpid{\less}} \cap {\IPIPmid{\less}} = \{(1,4)\} \ne \varnothing = {\IPIPpmid{\less}}$. In other words, we might have to iterate several times the maps~${{\less} \mapsto {\IPIPpid{\less}}}$ and~${{\less} \mapsto {\IPIPmid{\less}}}$ to obtain the map~${\less} \mapsto {\IPIPpmid{\less}}$. This explains the slightly more intricate definition of the map~${{\less} \mapsto {\IPIPpmid{\less}}}$. The same remark holds for the map~${\less} \mapsto {\DPIPpmdd{\less}}$.
\end{remark}

\begin{lemma}
\label{lem:IPIPidDPIPdd1}
For any poset~${\less} \in \IPos(n)$ and any~$\varepsilon \in \{\varnothing, -, +, \pm\}$, we have~${\IPIPeid{\less}} \in \IPIPe$ and ${\DPIPedd{\less}} \in \DPIPe$.
\end{lemma}

\begin{proof}
We split the proof into three technical claims whose proofs are given in Appendix~\ref{subsec:appendixPIPdeletion}.

\vspace{-.1cm}
\begin{claim}
\label{claim:IPIPpmidPoset}
$\IPIPpmid{\less}$ is a poset.
\end{claim}

\vspace{-.2cm}
\begin{claim}
\label{claim:IPIPpmidIPIPpm}
$\IPIPpmid{\less}$ is in~$\IPIPpm$.
\end{claim}
\vspace{-.1cm}

\noindent
This proves the result for~$\IPIPpmid{\less}$. Note that it already contains the result for~$\IPIPpid{\less}$, since~${\IPIPpid{\less}} = {\IPIPpmid{\less}[\orientation\positive, \varnothing]} \in \IPIPpm[\orientation\positive, \varnothing] = \IPIPp$, and similarly for~$\IPIPmid{\less}$. 

\vspace{-.05cm}
\begin{claim}
\label{claim:IPIPidIPIP}
${\IPIPid{\less}}$ is in~$\IPIP$.
\end{claim}
\vspace{-.05cm}

\noindent
Finally, the result for~$\DPIPedd{\less}$ with~$\varepsilon \in \{\varnothing, -, +, \pm\}$ follows by symmetry.
\end{proof}

\begin{lemma}
\label{lem:IPIPidDPIPdd2}
For any poset~${\less} \in \IPos(n)$ and any~$\varepsilon \in \{\varnothing, -, +, \pm\}$, the poset~$\IPIPeid{\less}$ (resp.~$\DPIPedd{\less}$) is the weak order minimal (resp.~maximal) poset in~$\IPIPe$ bigger than~$\less$ (resp.~in~$\DPIPe$ smaller than~$\less$).
\end{lemma}

\begin{proof}
We prove the result for~$\IPIPeid{\less}$, the proof for~$\DPIPedd{\less}$ being symmetric. Observe that~${\less} \wole {\IPIPeid{\less}}$ since~$\IPIPeid{\less}$ is obtained from~$\less$ by deleting increasing relations. Consider now~${\bless} \in \IPIPe$ such that~${\less} \wole {\bless}$. The following claim is proved in Appendix~\ref{subsec:appendixPIPdeletion}.

\begin{claim}
\label{claim:IPIPidDPIPdd2}
${\Inc{\bless}} \subseteq {\Inc{(\IPIPeid{\less})}}$.
\end{claim}
\vspace{-.1cm}

\noindent
This conclude the proof since~${\Inc{(\IPIPeid{\less})}} \supseteq {\Inc{\bless}}$ and~${{\Dec{\bless}} \subseteq {\Dec{\less}} = {\Dec{(\IPIPeid{\less})}}}$ implies that~${\IPIPeid{\less}} \wole {\bless}$.
\end{proof}

Consider now the \defn{$\PIP$ deletion} defined by
\[
{\PIPd{\less}} \eqdef {\IPIPid{(\DPIPdd{\less})}} = {\DPIPdd{(\IPIPid{\less})}}.
\]
See \fref{fig:IPIPpmid/DPIPpmdd/PIPd}.
It follows from Lemma~\ref{lem:IPIPidDPIPdd1} that~${\PIPd{\less}} \in \PIP$ for any poset~${\less} \in \IPos(n)$. We now compare this map with the permutree insertion~$\surjection$ defined in Proposition~\ref{prop:permutrees}.

\begin{proposition}
\label{prop:PIPd/bst}
For any permutation~$\sigma \in \fS(n)$, we have~${\PIPd{\less_\sigma}} = {\less_{\surjection(\sigma)}}$.
\end{proposition}

\begin{proof}
Let~$\sigma$ be a permutation of~$\fS(n)$ and let~${\bless} \eqdef {\PIPd{\less_\sigma}}$. We already know that~${\bless} \in \PIP$. The following claim is proved in Appendix~\ref{subsec:appendixPIPdeletion}.

\vspace{-.1cm}
\begin{claim}
\label{claim:blessSnakes}
${\bless}$ has an $\orientation$-snake between any two values of~$[n]$.
\end{claim}
\vspace{-.1cm}

\noindent
By Proposition~\ref{prop:characterizationPEP}, we thus obtain that~${\bless} \in \PEP$. Since moreover~$\less_\sigma$ is a linear extension of~$\bless$, we conclude that~${\bless} = {\less_{\surjection(\sigma)}}$.
\end{proof}

To obtain a similar statement for~$\WOIP(n)$, we first need to observe that the map~${\less} \mapsto {\PIPd{\less}}$ commutes with intersections. This straightforward proof is left to the reader.

\begin{proposition}
\label{prop:PIPdIntersection}
For any posets~${\less}, {\bless} \in \IPos(n)$, we have~$\PIPd{({\less} \cap {\bless})} = {\PIPd{\less}} \cap {\PIPd{\bless}}$.
\end{proposition}

\begin{corollary}
\label{coro:PIPd/bstIntervals}
For any permutations~$\sigma \wole \sigma'$, we have~${\PIPd{\less_{[\sigma,\sigma']}}} = {\less_{[\surjection(\sigma), \surjection(\sigma')]}}$.
\end{corollary}

\begin{proof}
Applying Propositions~\ref{prop:PIPd/bst} and~\ref{prop:PIPdIntersection}, we obtain
\[
{\PIPd{\less_{[\sigma,\sigma']}}} = {\PIPd{({\less_\sigma} \cap {\less_{\sigma'}})}} = {\PIPd{\less_\sigma}} \cap {\PIPd{\less_{\sigma'}}} = {\less_{\surjection(\sigma)}} \cap {\less_{\surjection(\sigma')}} = {\less_{[\surjection(\sigma), \surjection(\sigma')]}}.
\qedhere
\]
\end{proof}

Finally, we compare the $\PIP$ deletion with the Schr\"oder permutree insertion defined in Proposition~\ref{prop:SchroderPermutrees}.

\begin{proposition}
\label{prop:PIPd/st}
For any ordered partition~$\pi$ of~$[n]$, we have~${\PIPd{\less_\pi}} = {\less_{\surjection(\pi)}}$.
\end{proposition}

\begin{proof}
Let~$\pi$ be an ordered partition and let~${\bless} \eqdef {\PIPd{\less_\pi}}$. We already know that~${\bless} \in \PIP$. The following claim is proved in Appendix~\ref{subsec:appendixPIPdeletion}.

\vspace{-.1cm}
\begin{claim}
\label{claim:blessSpadeClub}
Any~$a < c$ incomparable in~${\bless}$ satisfy at least one of the conditions~$(\spadesuit)$ and~$(\clubsuit)$ of Proposition~\ref{prop:characterizationPFP}.
\end{claim}
\vspace{-.1cm}

\noindent
By Proposition~\ref{prop:characterizationPFP}, we thus obtain that~${\bless} \in \PFP$. Since moreover any linear extension of~$\bless$ extends~$\less_\pi$, we conclude that~${\bless} = {\less_{\surjection(\pi)}}$.
\end{proof}


\section{Sublattices}
\label{sec:sublattices}

The previous sections were dedicated to the characterization of various specific families of posets coming from permutreehedra and to the description of the weak order induced by these families. In this final section, we investigate which of these families induce sublattices of the weak order on posets~$(\IPos(n), \wole, \meetT, \joinT)$. We first introduce some additional notations based on conflict functions which will simplify later the presentation.


\subsection{Conflict functions}
\label{subsec:conflictFunctions}

A \defn{conflict function} is a function~$\conflicts$ which maps a poset~${{\less} \in \IPos(n)}$ to a conflict set~$\conflicts(\less) \subseteq \binom{[n]}{2}$. A poset~$\less$ is \defn{$\conflicts$-free} if~$\conflicts(\less) = \varnothing$, and we denote the set of $\conflicts$-free posets on~$[n]$ by $\free(\conflicts, n) \eqdef \set{{\less} \in \IPos(n)}{\conflicts(\less) = \varnothing}$. Intuitively, the set~$\conflicts(\less)$ gathers the \defn{conflicting pairs} that prevent~$\less$ to be a poset in the family~$\free(\conflicts,n)$.

\begin{example}
\label{exm:conflictFunctionsIWOIPDWOIP}
The characterizations of the families of posets discussed in Sections~\ref{sec:relevantFamiliesPermutahedron}, \ref{sec:relevantFamiliesAssociahedron} and~\ref{sec:relevantFamiliesPermutreehedra} naturally translate to conflict functions. For example, the posets in~$\IWOIP(n)$ and in~$\DWOIP(n)$ are the conflict-free posets for the conflict functions respectively given by
\[
\begin{array}{l@{\,=\,}l}
\conflicts_{\IWOIP}(\less) & \bigset{\{a,c\}}{a \less c \text{ and } \exists \; a < b < c, \;\; a \not\less b \not\less c}, \\
\conflicts_{\DWOIP}(\less) & \bigset{\{a,c\}}{a \more c \text{ and } \exists \; a < b < c, \;\; a \not\more b \not\more c}.
\end{array}
\]
The reader can derive from the characterizations of the previous sections other relevant conflict functions. In general, we denote by~$\conflicts_{\X}$ the conflict function defining a family~$\X$, \ie such that~$\free(\conflicts_{\X}, n) = \X(n)$.
\end{example}

For a poset~$\less$, we denote by~$[\less] \eqdef \bigset{\{i,j\}}{i \less j} \subseteq \binom{[n]}{2}$ the support of~$\less$, \ie the set of pairs of comparable elements in~$\less$. We say that a conflict function~$\conflicts$ is:
\begin{enumerate}[(i)]
\item \defn{local} if~$\{a,b\} \in \conflicts(\less) \iff \{a,b\} \in \conflicts(\less \cap \; [a,b]^2)$ for any~$a < b$ and any poset~$\less$, \ie a conflict~$\{a,b\}$ only depends on the relations in the interval~$[a,b]$,
\item \defn{increasing} if~$\conflicts(\less) \subseteq [\Inc{\less}]$ for any poset~$\less$, \ie only increasing relations are conflicting, \\ 
	  \defn{decreasing} if~$\conflicts(\less) \subseteq [\Dec{\less}]$ for any poset~$\less$, \ie only decreasing relations are conflicting, \\
	  \defn{incomparable} if~$\conflicts(\less) \subseteq \binom{[n]}{2} \ssm [\less]$ for any poset~$\less$, \ie only incomparable pairs are conflicting,
\item \defn{consistent} if~$\conflicts(\less) \cap [\Inc{\less}] = \conflicts(\Inc{\less})$ and~$\conflicts(\less) \cap [\Dec{\less}] = \conflicts(\Dec{\less})$ for any poset~$\less$, \ie increasing (resp.~decreasing) conflicts only depends on increasing (resp.~decreasing) relations,
\item \defn{monotone} if~${\less} \subseteq {\bless} \; \implies \; {\less} \ssm \, \conflicts(\less) \subseteq {\bless} \ssm \, \conflicts(\bless)$,
\item \defn{semitransitive} if~$\less \ssm \, \conflicts(\less)$ is semitransitive, \ie both increasing and decreasing subrelations of~$\less \ssm \, \conflicts(\less)$ are transitive. In other words, if~$a < b < c$ are such that the relations~$a \less b \less c$ are not conflicts for~$\conflicts$, then the relation~$a \less c$ is not a conflict for~$\conflicts$ (and similarly for~$\more$).
\end{enumerate}

\begin{example}
\label{exm:propertiesConflictFunctionsIWOIPDWOIP}
The conflict functions~$\conflicts_\IWOIP$ and~$\conflicts_\DWOIP$ are both local, consistent, monotone and semitransitive. Moreover, $\conflicts_\IWOIP$ is increasing while $\conflicts_\DWOIP$ is decreasing. Indeed, all these properties but the semitransitivity follow directly from the definitions. For the semitransitivity, consider~$a < b < c$ with~$ a \less b \less c$ and~$\{a,c\} \in \conflicts_\IWOIP(\less)$. Then there is~$a < d < c$ such that~$a \not\less d \not\less c$. Assume for example that~$a < d < b$. By transitivity of~$\less$, we have~$d \not\less b$, and thus~$\{a,b\} \in \conflicts_\IWOIP(\less)$.
\end{example}

\begin{remark}
\label{rem:unionConflictFunctions}
If~$\conflicts$ and~$\conflicts'$ are two conflict functions, then $\conflicts \cup \conflicts'$ is as well a conflict function with~$\free(\conflicts \cup \conflicts', n) = \free(\conflicts, n) \cap \free(\conflicts', n)$. For example, $\conflicts_\WOIP = \conflicts_\IWOIP \cup \conflicts_\DWOIP$ is the conflict function for~$\WOIP = \IWOIP \cap \DWOIP$. Note that all the above conditions are stable by union.
\end{remark}

The above conditions suffices to guaranty that $\conflicts$-free posets induce semi-sublattices of~${(\IPos(n), \wole)}$.

\begin{proposition}
\label{prop:STMCIDConflictFunction}
For any consistent monotone semitransitive increasing (resp.~decreasing) conflict function~$\conflicts$, the set of $\conflicts$-free posets induces a meet-semi-sublattice of~$(\IPos(n), \wole, \meetT)$ (resp.~a join-semi-sublattice of~$(\IPos(n), \wole, \joinT)$).
\end{proposition}

\begin{proof}
We prove the result for increasing conflict functions, the proof being symmetric for decreasing ones.
Let~$\less, \bless$ be two $\conflicts$-free posets and~${{\dashv} \eqdef {\less} \meetST {\bless} = \tc{({\Inc{\less}} \cup {\Inc{\bless}})} \cup ({\Dec{\less}} \cap {\Dec{\bless}})}$, so that~${\less} \meetT {\bless} = {\tdd{\dashv}}$. We want to prove that~$\tdd{\dashv}$ is also $\conflicts$-free.
Assume first that~$\dashv$ is not $\conflicts$-free, and let~${\{a,c\} \in \conflicts(\dashv)}$ with~$a < c$ and~$c-a$ minimal. Since~$\conflicts$ is increasing, we have ${(a,c) \in {\Inc{\dashv}} = \tc{({\Inc{\less}} \cup {\Inc{\bless}})}}$. If~$(a,c) \notin ({\Inc{\less}} \cup {\Inc{\bless}})$, then there exists~${a = b_1 < b_2 < \dots < b_k = c}$ such that~$a = b_1 \dashv b_2 \dashv \dots \dashv b_k = c$. By minimality of~$c-a$, all~$(b_i, b_{i+1})$ are in~${\dashv} \ssm \, \conflicts(\dashv)$ while~$(a,c)$ is not, which contradicts the semitransitivity of~$\conflicts$. Therefore, $(a,c) \in ({\Inc{\less}} \cup {\Inc{\bless}})$ and we can assume without loss of generality that~$(a,c) \in {\Inc{\less}}$. Since~$\less$ is $\conflicts$-free and $\conflicts$ is consistent, we have~${(a,c) \in {\Inc{\less}} \ssm \, \conflicts(\Inc{\less})}$. Thus, since $\conflicts$ is monotone and~${\Inc{\less}} \subseteq {\dashv}$, we obtain that~${(a,c) \in {\dashv} \ssm \, \conflicts(\dashv)}$ which contradicts our assumption that~$\{a,c\} \in \conflicts(\dashv)$. We therefore obtained that~$\dashv$ is $\conflicts$-free.
Finally, since~$\conflicts$ is monotone, consistent, and increasing, and since~${{\Inc{\dashv}} = \Inc{(\tdd{\dashv})}}$, we conclude that~$\tdd{\dashv}$ is $\conflicts$-free.
\end{proof}

\begin{example}
\label{exm:IWOIPDWOIPmeetsemisublattices}
Applying Example~\ref{exm:propertiesConflictFunctionsIWOIPDWOIP} and Proposition~\ref{prop:STMCIDConflictFunction}, we obtain that the subposet of the weak order induced by~$\IWOIP(n)$ (resp.~by~$\DWOIP(n)$) is a meet-semi-sublattice of~$(\IPos(n), \wole, \meetT)$ (resp.~a join-semi-sublattice of~$(\IPos(n), \wole, \joinT)$), as already proved in Proposition~\ref{prop:IWOIPDWOIPLattices}.
\end{example}


\subsection{Intervals}
\label{subsec:intervalsSublattices}

We now consider lattice properties of the weak order on permutree interval posets~$\PIP$. This section has two main goals:
\begin{enumerate}[(i)]
\item provide a sufficient condition on~$\orientation$ for~$\PIP$ to induce a sublattice of~$(\IPos(n), \wole, \meetT, \joinT)$,
\item show that~$\PIP$ induces a sublattice of~$(\WOIP(n), \wole, \meetWOIP, \joinWOIP)$ for any orientation~$\orientation$.
\end{enumerate}
Using the notations introduced in Section~\ref{subsec:conflictFunctions}, we consider the conflict functions
\[
\begin{array}{l@{\,}l}
\conflicts_{\IPIPp}(\less) & \eqdef \, \bigset{\{a,c\}}{a \less c \text{ and } \exists \; a < b < c, \;\; b \in \orientation\positive \text{ and } a \not\less b}, \\
\conflicts_{\IPIPm}(\less) & \eqdef \, \bigset{\{a,c\}}{a \less c \text{ and } \exists \; a < b < c, \;\; b \in \orientation\negative \text{ and } b \not\less c}, \\
\conflicts_{\IPIPpm}(\less) & \eqdef \, \conflicts_{\IPIPp}(\less) \, \cup \, \conflicts_{\IPIPm}(\less), \\
\conflicts_{\IPIP}(\less) & \eqdef \, \conflicts_{\IPIPpm}(\less) \, \cup \, \conflicts_{\IWOIP}(\less), \\[.25cm]
\conflicts_{\DPIPp}(\less) & \eqdef \, \bigset{\{a,c\}}{a \more c \text{ and } \exists \; a < b < c, \;\; b \in \orientation\positive \text{ and } b \not\more c}, \\
\conflicts_{\DPIPm}(\less) & \eqdef \, \bigset{\{a,c\}}{a \more c \text{ and } \exists \; a < b < c, \;\; b \in \orientation\negative \text{ and } a \not\more b}, \\
\conflicts_{\DPIPpm}(\less) & \eqdef \, \conflicts_{\DPIPp}(\less) \, \cup \, \conflicts_{\DPIPm}(\less), \\
\conflicts_{\DPIP}(\less) & \eqdef \, \conflicts_{\DPIPpm}(\less) \, \cup \, \conflicts_{\DWOIP}(\less),
\end{array}
\]
and finally
\[
\conflicts_{\PIP}(\less) \eqdef \conflicts_{\IPIP}(\less) \cup \conflicts_{\DPIP}(\less)
\]
corresponding to the families studied in Section~\ref{subsec:PIP}. As seen in Proposition~\ref{coro:characterizationPIP}, the $\conflicts_{\PIP}$-free posets are precisely that of~$\PIP$.


\para{Covering orientations}
In the next statements, we provide a sufficient condition on the orientation~$\orientation$ for~$\PIP$ to induce a sublattice of~$(\IPos(n), \wole, \meetT, \joinT)$. We first check the conditions of Proposition~\ref{prop:STMCIDConflictFunction} to get semi-sublattices.

\begin{lemma}
\label{lem:orientationConflictFunctionProperties}
For any orientation~$\orientation$ and any~$\varepsilon \in \{\varnothing, -, +, \pm\}$, the conflict functions~$\conflicts_{\IPIPe}$ and~$\conflicts_{\DPIPe}$ are local, consistent, monotone, and semitransitive. Moreover, $\conflicts_{\IPIPe}$ is increasing while $\conflicts_{\DPIPe}$ is decreasing.
\end{lemma}

\begin{proof}
Since they are stable by union (Remark~\ref{rem:unionConflictFunctions}), and since they hold for the conflict functions~$\conflicts_\IWOIP$ and~$\conflicts_\DWOIP$ (Example~\ref{exm:propertiesConflictFunctionsIWOIPDWOIP}), it suffices to show these properties for the conflict functions~$\conflicts_{\IPIPp}$, $\conflicts_{\IPIPm}$, $\conflicts_{\DPIPp}$ and~$\conflicts_{\DPIPm}$. By symmetry, we only consider~$\conflicts_{\IPIPp}$. 
We just need to prove the semitransitivity, the other properties being immediate from the definitions. Consider~$a < b < c$ such that~$a \less b \less c$ and~$\{a,c\} \in \conflicts_{\IPIPp}(\less)$. Then there exists~$a < d < c$ such that~$d \in \orientation\positive$ and~$a \not\less d$. If~$d < b$, then $\{a,b\} \in \conflicts_{\IPIPp}(\less)$. Otherwise, $b < d$ and the transitivity of~$\less$ ensures that~$b \not\less d$, so that~$\{b,c\} \in \conflicts_{\IPIP}(\less)$. We conclude that~$\Inc{(\less \ssm \, \conflicts_{\IPIPp}(\less))}$ is transitive. Since~$\Dec{(\less \ssm \, \conflicts_{\IPIPp}(\less))} = {\Dec{\less}}$ is also transitive, we obtained that~$\conflicts_{\IPIPp}$ is semitransitive.
\end{proof}

\begin{corollary}
\label{coro:orientationConflictFunctionsSemiSublattices}
For any orientation~$\orientation$ and any~$\varepsilon \in \{\varnothing, -, +, \pm\}$, the set~$\IPIPe$ (resp.~$\DPIPe$) induces a meet-semi-sublattice of~$(\IPos(n), \wole, \meetT)$ (resp.~a join-semi-sublattice of~$(\IPos(n), \wole, \joinT)$).
\end{corollary}

\begin{proof}
Direct application of Lemma~\ref{lem:orientationConflictFunctionProperties} and Proposition~\ref{prop:STMCIDConflictFunction}.
\end{proof}

To obtain sublattices, we need an additional condition on~$\orientation$. Namely, we say that an orientation~$\orientation = (n, \orientation\positive, \orientation\negative)$ is \defn{covering} if~$\{2, \dots, n-1\} \subseteq \orientation\positive \cup \orientation\negative$. Note that we do not require a priori that~$\orientation\positive \cap \orientation\negative = \varnothing$ nor that~$\{1,n\} \subseteq \orientation\positive \cup \orientation\negative$. Observe also that when~$\orientation$ is covering, we have~$\IPIPpm = \IPIP$ and~$\DPIPpm = \DPIP$.

\begin{theorem}
\label{thm:coveringOrientationConflictFunctionsSublatticesIPos}
For any covering orientation~$\orientation$, the sets~$\IPIP$,~$\DPIP$ and~$\PIP$ all induce sublattices of~$(\IPos(n), \wole, \meetT, \joinT)$.
\end{theorem}

\begin{proof}
We only prove the result for~$\DPIP$. It then follows by symmetry for~$\IPIP$, which in turn implies the result for~$\PIP$ since~$\PIP = \IPIP \cap \DPIP$. We already know from Corollary~\ref{coro:orientationConflictFunctionsSemiSublattices} that~$\DPIP$ is stable by~$\joinT$ and it remains to show that it is stable by~$\meetT$.
We thus consider two posets~${\less,\bless} \in \DPIP$ and let~${{\dashv} \eqdef {\less} \meetST {\bless} = \tc{({\Inc{\less}} \cup {\Inc{\bless}})} \cup ({\Dec{\less}} \cap {\Dec{\bless}})}$, so that~${\less} \meetT {\bless} = {\tdd{\dashv}}$. We decompose the proof in two steps, whose detailed proofs are given in Appendix~\ref{subsec:appendixCoveringOrientationConflictFunctionsSublatticesIPos}.

\vspace{-.1cm}
\begin{claim}
\label{claim:dashvDPIP}
${\dashv}$ is in~$\DPIP$.
\end{claim}

\vspace{-.3cm}
\begin{claim}
\label{claim:tdddashvDPIP}
${\tdd{\dashv}}$ is in~$\DPIP$. \qedhere
\end{claim}
\end{proof}

\begin{corollary}
\label{coro:TOIPCOIPBOIPsublattices}
The weak order on interval posets in the Tamari lattice, in any type~$A_n$ Cambrian lattice, and in the boolean lattice are all sublattices of~$(\IPos(n), \wole, \meetT, \joinT)$.
\end{corollary}

\begin{proof}
Apply Theorem~\ref{thm:coveringOrientationConflictFunctionsSublatticesIPos} to the orientations illustrated in~\fref{fig:permutrees}: the Tamari lattice is the lattice~$\PIP[\protect{\varnothing, [n]}]$, the Cambrian lattices are the lattices~$\PIP[\orientation\positive, \orientation\negative]$ for all partitions~${\orientation\positive \sqcup \orientation\negative = [n]}$, and the boolean lattice is the lattice~$\PIP[\protect{[n], [n]}]$.
\end{proof}

\begin{remark}
\label{rem:notCoveringNotSublattice}
The covering condition is essential to the proof of Theorem~\ref{thm:coveringOrientationConflictFunctionsSublatticesIPos}. For example, Remark~\ref{rem:WOIPnotSublattice} shows that~$\WOIP(n) = \PIP[\varnothing, \varnothing]$ does not induce a sublattice of~$(\IPos(n), \wole, \meetT, \joinT)$.
\end{remark}


\para{$\PIP$ induces a sublattice of~$\WOIP(n)$}
We now consider an arbitrary orientation~$\orientation$, not necessarily covering. Although~$\PIP$ does not always induce a sublattice of~${(\IPos(n), \wole, \meetT, \joinT)}$, we show that it always induces a sublattice of~$(\WOIP(n), \wole, \meetWOIP, \joinWOIP)$.

\begin{theorem}
\label{thm:orientationConflictFunctionsSublatticesWOIP}
For any orientation~$\orientation$ and any~$\varepsilon \in \{\varnothing, -, +, \pm\}$, the set~$\IPIPe$ (resp.~$\DPIPe$) induces a sublattice of~$(\WOIP(n), \wole, \meetWOIP, \joinWOIP)$.
\end{theorem}

\begin{proof}
By symmetry, it suffices to prove the result for~$\DPIPm$. Let~${\less}, {\bless} \in \DPIPm$. We already know from Corollary~\ref{coro:orientationConflictFunctionsSemiSublattices} that~${\less} \joinT {\bless} \in \DPIPm$. Since $\conflicts_{\DPIPm}$ is a decreasing conflict function and since the $\IWOIP$ increasing deletion only deletes increasing relations, we thus obtain that
\[
{\less} \joinWOIP {\bless} = \IWOIPid{({\less} \joinT {\bless})} \in \DPIPm.
\]
It remains to prove that
\[
{\less} \meetWOIP {\bless} = \DWOIPdd{({\less} \meetT {\bless})} \in \DPIPm.
\]
For this, let us denote~${\dashv} \eqdef {\less} \meetST {\bless} = \tc{({\Inc{\less}} \cup {\Inc{\bless}})} \cup ({\Dec{\less}} \cap {\Dec{\bless}})$ and~${\dashV} \eqdef {\less} \meetWOIP {\bless}$ so that~${\dashV} = {\DWOIPdd{(\tdd{\dashv})}}$. As in the proof of Theorem~\ref{thm:coveringOrientationConflictFunctionsSublatticesIPos}, we know that~${\dashv} \in \DPIPm$. Assume now that~${\dashV} \notin\DPIPm$. Consider~$\{a,c\} \in \conflicts_{\DPIP}(\dashV)$ with~$a < c$ and~$c-a$ minimal. We therefore have~$a \Vdash c$ while there exists~$a < b < c$ with~$b \in \orientation\negative$ and~$a \not\Vdash b$. Note that since~${\dashv} \in \DPIPm$, we have~$a \vdash b$. We now distinguish two cases:
\begin{itemize}
\item If~$a \not\tdd{\vdash} b$, then there exists~$i \le b$ and~$j \ge a$ such that~$i \dashv b \dashv a \dashv j$ but~$i \not\dashv j$. From Lemma~\ref{lem:simplifytdd}, we know that there exists~$a < k < b$ such that~$a \not\tdd{\vdash} k \not\tdd{\vdash} b$.
\item If~$a \tdd{\vdash} b$, then there exists~$a < k_1 < \dots < k_\ell < b$ such that~$a \not\tdd{\vdash} k_1 \not\tdd{\vdash} \dots \not\tdd{\vdash} k_\ell \not\tdd{\vdash} b$.
\end{itemize}
In both cases, there exists~$a < k < b$ such that~$a \not\Vdash k \not\Vdash b$. Since~$\dashV \in \IWOIP$ and~$a \Vdash c$ while~$a \not\Vdash k$, we must have~$k \Vdash c$. But since~$k \not\Vdash b$, we then have~$\{k,c\} \in \conflicts_{\DPIP}(\dashV)$ contradicting the minimality of~$c-a$ in our choice of~$\{a,c\}$.
\end{proof}

\begin{corollary}
\label{coro:orientationConflictFunctionsSublatticesWOIP}
For any orientation~$\orientation$, $\PIP$ induces a sublattice of~$(\WOIP(n), \wole, \meetWOIP, \joinWOIP)$.
\end{corollary}

\begin{proof}
Immediate consequence of Theorem~\ref{thm:orientationConflictFunctionsSublatticesWOIP} as~$\PIP = \IPIP \cap \DPIP$.
\end{proof}


\subsection{Elements}
\label{subsec:elementsSublattices}

We now consider lattice properties of the weak order on permutree element posets~$\PEP$.
Similarly to the previous section, the present section has two main goals:
\begin{enumerate}[(i)]
\item provide a sufficient condition on~$\orientation$ for~$\PEP$ to induce a sublattice of~$(\IPos(n), \wole, \meetT, \joinT)$,
\item show that~$\PEP$ induces a sublattice of $(\WOIP(n), \wole, \meetWOIP, \joinWOIP)$ for any orientation~$\orientation$.
\end{enumerate}
We start with a simple observation.

\begin{proposition}
\label{prop:elementsSublatticeIntervals}
The set~$\PEP$ induces a sublattice of $(\PIP, \wole, \meetPIP, \joinPIP)$ for any orientation~$\orientation$.
\end{proposition}

\begin{proof}
We have seen in Corollary~\label{coro:PIPLattice} that the meet and join in~$\PIP$ are given by
\[
{\less_{[\tree[S],\tree[S]']}} \meetPIP {\less_{[\tree,\tree']}} = {\less_{[\tree[S] \meetO \tree, \, \tree[S]' \meetO \tree']}}
\qquad\text{and}\qquad
{\less_{[\tree[S],\tree[S]']}} \joinPIP {\less_{[\tree,\tree']}} = {\less_{[\tree[S] \joinO \tree, \, \tree[S]' \joinO \tree']}}.
\]
Therefore, for any~$\tree[S], \tree[T] \in \PEP$, we have
\[
\less_{\tree[S]} \meetPIP \less_{\tree[T]} = \less_{[\tree[S], \tree[S]]} \meetPIP \less_{[\tree[T], \tree[T]]} = \less_{[\tree[S] \meetO \tree[T], \tree[S] \meetO \tree[T]]} = \less_{\tree[S] \meetO \tree[T]}.
\qedhere
\]
\end{proof}

Proposition~\ref{prop:elementsSublatticeIntervals} enables to show Theorems~\ref{thm:coveringOrientationElementsSublatticeIPos} and~\ref{thm:nonCoveringOrientationElementsSublatticeWOIP} below.

\begin{theorem}
\label{thm:coveringOrientationElementsSublatticeIPos}
For any covering orientation~$\orientation$, $\PEP$ induces a sublattice of $(\IPos(n), \wole, \meetT, \joinT)$.
\end{theorem}

\begin{proof}
$\PEP$ induces a sublattice of $(\PIP, \wole, \meetPIP, \joinPIP)$ (by Proposition~\ref{prop:elementsSublatticeIntervals}), which in turn is a sublattice of $(\IPos(n), \wole, \meetT, \joinT)$ when~$\orientation$ is covering (by Theorem~\ref{thm:coveringOrientationConflictFunctionsSublatticesIPos}).
\end{proof}

\begin{corollary}
\label{coro:TOEPCOEPBOEPsublattices}
The Tamari lattice, any type~$A_n$ Cambrian lattice, and the boolean lattice are all sublattices of~$(\IPos(n), \wole, \meetT, \joinT)$.
\end{corollary}

\begin{proof}
Apply Theorem~\ref{thm:coveringOrientationElementsSublatticeIPos} to the orientations illustrated in~\fref{fig:permutrees}: the Tamari lattice is the lattice~$\PIP[\protect{\varnothing, [n]}]$, the Cambrian lattices are the lattices~$\PIP[\orientation\positive, \orientation\negative]$ for all partitions~${\orientation\positive \sqcup \orientation\negative = [n]}$, and the boolean lattice is the lattice~$\PIP[\protect{[n], [n]}]$.
\end{proof}

\begin{remark}
\label{rem:PEPnotSublatticePos}
Note that the covering condition in Theorem~\ref{thm:coveringOrientationElementsSublatticeIPos} is necessary in general. For example, for the orientation~$\orientation = (5, \{2\}, \{4\})$ on~$[5]$, the lattice $(\PEP, \wole, \meetPEP, \joinPEP)$ is not a sublattice of~$(\IPos(5), \wole, \meetT, \joinT)$. For example, for
\vspace{-.8cm}
\begin{align*}
{\less} = \raisebox{-1.22cm}{\scalebox{.8}{\input{relations/pep_meet1}}}
\qquad \text{and} \qquad
{\bless} = \raisebox{-1.22cm}{\scalebox{.8}{\input{relations/pep_meet2}}}
\end{align*}

\vspace{-.5cm}
\noindent
we have
\vspace{-.5cm}
\begin{align*}
{\less} \meetT {\bless} = \raisebox{-1.22cm}{\scalebox{.8}{\input{relations/pep_transitive_meet}}}
\qquad \text{while} \qquad
{\less} \meetPEP {\bless} = \raisebox{-1.22cm}{\scalebox{.8}{\input{relations/pep_woip_meet}}} .
\end{align*}
\end{remark}

\vspace{-.2cm}
However, for arbitrary orientation, we can still obtain the following weaker statement.

\begin{theorem}
\label{thm:nonCoveringOrientationElementsSublatticeWOIP}
For any orientation~$\orientation$, $\PEP$ induces a sublattice of~$(\WOIP(n), \wole, \meetWOIP, \joinWOIP)$.
\end{theorem}

\begin{proof}
$\PEP$ induces a sublattice of $(\PIP, \wole, \meetPIP, \joinPIP)$ (by Proposition~\ref{prop:elementsSublatticeIntervals}), which in turn induces a sublattice of $(\WOIP(n), \wole, \meetWOIP, \joinWOIP)$ (by Corollary~\ref{coro:orientationConflictFunctionsSublatticesWOIP}).
\end{proof}

Finally, let us give a poset proof of Proposition~\ref{prop:elementsSublatticeIntervals}.
Recall from Section~\ref{subsec:characterizationPEP} and \fref{fig:snakes} that, for an orientation~$\orientation$ of~$[n]$, an \defn{$\orientation$-snake} in a poset~$\less$ is a sequence~$x_0 < x_1 < \dots < x_k < x_{k+1}$ such that
\begin{itemize}
\item either~$x_0 \less x_1 \more x_2 \less x_3 \more \cdots$ with~$\set{x_i}{i \in [k] \text{ odd}} \subseteq \orientation\negative$ and $\set{x_i}{i \in [k] \text{ even}} \subseteq \orientation\positive$, 
\item or \quad\; $x_0 \more x_1 \less x_2 \more x_3 \less \cdots$ with~$\set{x_i}{i \in [k] \text{ odd}} \subseteq \orientation\positive$ and $\set{x_i}{i \in [k] \text{ even}} \subseteq \orientation\negative$.
\end{itemize}
Using the notations introduced in Section~\ref{subsec:conflictFunctions}, we consider the two conflict functions
\begin{gather*}
\conflicts_{\incompE}(\less) \eqdef \bigset{\{a,c\}}{\text{there is no $\orientation$-snake joining $a$ to $c$}}, \\
\conflicts_{\PEP}(\less) \eqdef \conflicts_{\PIP}(\less) \cup \conflicts_{\incompE}(\less).
\end{gather*}
As seen in Proposition~\ref{prop:characterizationPEP}, $\conflicts_{\incompE}$ corresponds to the condition characterizing~$\PEP$ in~$\PIP$, so that the $\conflicts_{\PEP}$-free posets are precisely that of~$\PEP$. We now prove that the conflict function~$\conflicts_{\incompE}$ alone induces a sublattice of the weak order on posets.

\begin{proposition}
\label{prop:orientationIncompSublatticesIPos}
For any orientation~$\orientation$ on~$[n]$, the set of $\conflicts_{\incompE}$-free posets induces a sublattice of~$(\IPos(n), \wole, \meetT, \joinT)$.
\end{proposition}

\begin{proof}
Consider two $\conflicts_{\incompE}$-free posets~$\less,\bless$ and let~${{\dashv} \eqdef {\less} \meetST {\bless} = \tc{({\Inc{\less}} \cup {\Inc{\bless}})} \cup ({\Dec{\less}} \cap {\Dec{\bless}})}$, so that~${\less} \meetT {\bless} = {\tdd{\dashv}}$. As in the proof of Theorem~\ref{thm:coveringOrientationConflictFunctionsSublatticesIPos}, we decompose the proof in two steps, whose detailed proofs are given in Appendix~\ref{subsec:appendixOrientationIncompSublatticesIPos}.

\vspace{-.1cm}
\begin{claim}
\label{claim:dashvConflictFree}
$\dashv$ is $\conflicts_{\incompE}$-free.
\end{claim}

\vspace{-.3cm}
\begin{claim}
\label{claim:tdddashvConflictFree}
$\tdd{\dashv}$ is $\conflicts_{\incompE}$-free. \qedhere
\end{claim}
\end{proof}

Note that Proposition~\ref{prop:orientationIncompSublatticesIPos} provides a proof of Proposition~\ref{prop:elementsSublatticeIntervals} on posets.
It also enables us to obtain further results for the specific orientation~$(n, \varnothing, \varnothing)$.
Indeed, Proposition~\ref{prop:characterizationWOEP} ensures that the~$\conflicts_{\incompE}$-free posets are precisely the posets of~$\WOEP(n)$. Proposition~\ref{prop:orientationIncompSublatticesIPos} therefore specializes to the following statement.

\begin{corollary}
\label{coro:WOEPsublattice}
The set~$\WOEP(n)$ induces a sublattice of the weak order~$(\IPos(n), \wole, \meetT, \joinT)$.
\end{corollary}


\subsection{Faces}
\label{subsec:facesSublattices}

In this section, we study the lattice properties of the weak order on permutree face posets~$\PFP$. We have seen in Propositions~\ref{prop:SchroderPermutrees} and~\ref{prop:weakOrderPFP} that the weak order on~$\PFP$ coincides with the Schr\"oder permutree lattice, but we have observed in Remark~\ref{rem:PFPnotSublattice} that it is not a sublattice of~$(\IPos(n), \wole, \meetT, \joinT)$, nor a sublattice of~$(\WOIP(n), \wole, \meetWOIP, \joinWOIP)$, nor a sublattice of~$(\PIP, \wole, \meetPIP, \joinPIP)$. For completeness, let us report on a method to compute the meet and join directly on the posets of~$\PFP$. For that, define the \defn{$\PFP$ increasing addition} and the \defn{$\PFP$ decreasing addition} by
\[
{\PFPia{\less}} =
\begin{cases}
{\less} & \text{if } {\less} \in \PFP \\
\PFPia{\big( \less \cup \set{(a,c)}{a < c \text{ not satisfying } (\spadesuit) \text{ nor } (\clubsuit)} \big)} & \text{otherwise}
\end{cases}
\]
and
\[
{\PFPda{\less}} =
\begin{cases}
{\less} & \text{if } {\less} \in \PFP \\
\PFPda{\big( \less \cup \set{(c,a)}{a < c \text{ not satisfying } (\spadesuit) \text{ nor } (\clubsuit)} \big)} & \text{otherwise.}
\end{cases}
\]
Experimental observations indicate that for~$\tree[S], \tree[S]' \in \PFP$,
\[
\tree[S] \meetPFP \tree[S]' = \PFPia{\big( \tree[S] \meetWOIP \tree[S]' \big)}
\qquad\text{and}\qquad
\tree[S] \joinPFP \tree[S]' = \PFPda{\big( \tree[S] \joinWOIP \tree[S]' \big)}.
\]
A complete proof of this observation would however be quite technical. It would in particular require a converging argument to prove that the $\PFP$ increasing and decreasing additions are well defined.


\section*{Acknowledgments}

\enlargethispage{.1cm}
The computation and tests needed along the research were done using the open-source mathematical software \texttt{Sage}~\cite{Sage} and its combinatorics features developed by the \texttt{Sage-combinat} community~\cite{SageCombinat}.


\bibliographystyle{alpha}
\bibliography{latticeIntegerPosets}
\label{sec:biblio}

\appendix 

\newpage
\section{Missing claims}
\label{sec:missingClaims}

This appendix gathers all missing proofs of the technical claims used in Part~\ref{part:relevantFamilies}.

\subsection{Proof of claims of Section~\ref{subsec:IWOIPDWOIP}}
\label{subsec:appendixIWOIPDWOIP}

\begin{proof}[Proof of Claim~\ref{claim:maxlePoset}]
First,~$\maxle{\less}$ is clearly antisymmetric since it is obtained from an antisymmetric relation by adding just decreasing relations between some incomparable elements. To prove that~$\maxle{\less}$ is transitive, consider~$u,v,w \in [n]$ be such that~$u \maxle{\less} v \maxle{\less} w$. We distinguish four cases:
\begin{enumerate}[(i)]
\item If~$u \not\less v \not\less w$, then we have~$w < v < u$ with~$w \not\less v \not\less u$. Our assumption thus ensures that~$w \not\less u$. Thus, either~$u \less w$ or~$u$ and~$w$ are incomparable. In both cases,~$u \maxle{\less} w$.
\item If~$u \not\less v \less w$, then we have~$v < u$ with~$v \not\less u$. We then have two cases:
	\begin{itemize}
	\item Assume that~$u < w$. Since~$v < u < w$ and~$v \not\less u$ while~$v \less w$, our assumption implies that~$u \less w$, so that~$u \maxle{\less} w$.
	\item Assume that~$w < u$. Since~$v \not\less u$ and~$v \less w$, the transitivity of~$\less$ impose that~$w \not\less u$. Thus either~$u \less w$ or~$u$ and~$w$ are incomparable. In both cases,~$u \maxle{\less} w$.
	\end{itemize}
\item If~$u \less v \not\less w$, then we have~$w < v$ with~$w \not\less v$. We then have two cases:
	\begin{itemize}
	\item Assume that~$u < w$. Since~$u < w < v$ and~$w \not\less v$ while~$u \less v$, our assumption implies that~$u \less w$ so that~$u \maxle{\less} w$.
	\item Assume that~$w < u$. Since~$w \not\less v$ and~$u \less v$, the transitivity of~$\less$ impose that~$w \not\less u$. Thus, either~$u \less w$ or~$u$ and~$w$ are incomparable. In both cases,~$u \maxle{\less} w$.
	\end{itemize}
\item If~$u \less v \less w$, then~$u \less w$ by transitivity of~$\less$ and thus~$u \maxle{\less} w$.
\end{enumerate}
We proved in all cases that~$u \maxle{\less} w$, so that~$\less$ is transitive. Since all our relations are reflexive, we conclude that~$\maxle{\less}$ is a poset.
\end{proof}

\begin{proof}[Proof of Claim~\ref{claim:IWOIPidPoset}]
First,~$\IWOIPid{\less}$ is clearly antisymmetric as it is contained in the antisymmetric relation~$\less$. To prove that it is transitive, consider~$u, v, w \in [n]$ such that~$u \IWOIPid{\less} v \IWOIPid{\less} w$. Since~${\IWOIPid{\less}} \subseteq {\less}$, we have~$u \less v \less w$, so that~$u \less w$ by transitivity of~$\less$. Assume by means of contradiction that~$u \not\IWOIPid{\less} w$. Thus, $u < w$ and there exists~$u < z_1 < \dots < z_k < w$ such that~$u \not\less z_1 \not\less \dots \not\less z_k \not\less w$. We now distinguish three cases:
\begin{enumerate}[(i)]
\item If~$v < u$, then~$v \not\IWOIPid{\less} w$ since~$v < u < z_1 < \dots < z_k < w$ and~$v \not\less u \not\less z_1 \not\less \dots \not\less z_k \not\less w$.
\item If~$u < v < w$, consider~$\ell \in [k]$ such that~$z_\ell \le v < z_{\ell + 1}$ (with~$\ell = 0$ if~$v < z_1$ and~$\ell = k$ if $z_k \le v$). Since~$z_\ell \not\less z_{\ell+1}$ and~$\less$ is transitive, we have either~$z_\ell \not\less v$ or~$v \not\less z_{\ell+1}$. In the former case, we have~$u \not\IWOIPid{\less} v$ since~$u < z_1 < \dots < z_\ell \le v$ and~$u \not\less z_1 \not\less \dots \not\less z_\ell \not\less v$. In the latter case, we have~$v \not\IWOIPid{\less} w$ since~$v < z_{\ell+1} < \dots < z_k < w$ and~$v \not\less z_{\ell+1} \not\less \dots \not\less z_k \not\less w$.
\item If~$w < v$, then~$u \not\IWOIPid{\less} v$ since~$u < z_1 < \dots < z_k < w < v$ and~$u \not\less z_1 \not\less \dots \not\less z_k \not\less w \not\less w$.
\end{enumerate}
As we obtained a contradiction in each case, we conclude that~$\IWOIPid{\less}$ is transitive. Since all our relations are reflexive, we conclude that~$\IWOIPid{\less}$ is a poset.
\end{proof}

\subsection{Proof of claims of Proposition~\ref{prop:characterizationPEP}}
\label{subsec:appendixCharacterizationPEP}

\begin{proof}[Proof of Claim~\ref{claim:snake1}]
By symmetry, we only need to prove the first statement. Note that~$u$ and~$w$ are incomparable, otherwise $u \less v$ and~$v \more w$ could not both be cover relations.
Therefore, there is a non-degenerate $\orientation$-snake~$u = x_0 < x_1 < \dots < x_k < x_{k+1} = w$ from~$u$ to~$w$. Assume first that~$x_1 < v$. If~$x_1 \in \orientation\positive$ and~$u \more x_1$, then~$u \less v$, $x_1 \in \orientation\positive$ and~${\less} \in \IPIPp$ implies that~$u \less x_1$, a contradiction. If~$x_1 \in \orientation\negative$ and~$u \less x_1$, then~$u \less v$, $x_1 \in \orientation\negative$ and~${\less} \in \IPIPp$ implies that~$x_1 \less v$ which together with~$u \less x_1$ would contradict that~$u \less v$ is a cover relation. As we reach a contradiction in both cases, we obtain that~$v \le x_1$, and by symmetry~$x_k \le v$. Therefore, we have~$x_1 = v = x_k$, so that~$u < v < w$ and~$v \in \orientation\negative$.
\end{proof}

\begin{proof}[Proof of Claim~\ref{claim:snake2}]
We work by induction on~$p$, the case~$p = 1$ being immediate. By symmetry, we can assume that~${x_0 \in \orientation\negative}$, $x_0 \more x_1$ and~$x_0 < x_1$. Let~$j$ be the first position in the path such that ${x_{j-1} \more x_j \less x_{j+1}}$ (by convention~$j = p$ if~$x_0 \more x_1 \more \dots \more x_p$). Assume that there is~$i \in [j]$ such that~$x_i \le x_0$, and assume that~$i$ is the first such index. Since~$x_i \le x_0 < x_{i-1}$, $x_i \less x_{i-1}$, $x_0 \in \orientation\negative$ and~${\less} \in \IPIPm$, we obtain~$x_0 \less x_{i-1}$, a contradiction. This shows that~$x_0 < x_i$ for~$i \in [j]$. If~$j = p$, the statement is proved. Otherwise, we consider~$x_{j-1}$, $x_j$ and~$x_{j+1}$. By Claim~\ref{claim:snake1}, we have~$x_j \in \orientation\positive$ and either~$x_{j-1} < x_j < x_{j+1}$ or~$x_{j+1} < x_j < x_{j-1}$. In the latter case, $x_0 < x_j < x_{j-1}$, $x_0 \more x_{j-1}$, $x_j \in \orientation\positive$ and~${\less} \in \DPIPp$ would imply $x_j \more x_{j-1}$, a contradiction. We thus obtain that~$x_j \in \orientation\positive$, $x_j \less x_{j+1}$ and~$x_j < x_{j+1}$. The induction hypothesis thus ensures that~$x_j < x_i$ for all~$j < i \le p$. This concludes since~$x_0 < x_j$.
\end{proof}

\subsection{Proof of claims of Proposition~\ref{prop:characterizationPFP}}
\label{subsec:appendixCharacterizationPFP}

\begin{proof}[Proof of Claim~\ref{claim:PEP}]
By Proposition~\ref{prop:characterizationPEP}, we just need prove that there is an $\orientation$-snake between any two values of~$[n]$. Otherwise, consider~$a < c$ with~$c-a$ minimal such that there is no $\orientation$-snake between~$a$ and~$c$. In particular, $a$ and~$c$ are incomparable. By~$(\spadesuit)$, we can assume for instance that there is~$a < b < c$ such that~$b \in \orientation\negative$ and~$a \not\more b \not\less c$. By minimality of~$c-a$, there is an $\orientation$-snake~$a = x_0 < x_1 < \dots < x_k < x_{k+1} = b$. Then we have either~$x_1 \in \orientation\negative$ and~$a \less x_1$, or~$x_1 \in \orientation\positive$ and~$a \more x_1$ (note that this holds even when~$x_1 = b$ since~$a \not\more b$ and~$b \in \orientation\negative$). Moreover, by minimality of~$c-a$, there is an $\orientation$-snake between $x_1$ and~$c$. Lemma~\ref{lem:orientationIncompCharacterization} thus ensures that there is as well an $\orientation$-snake between~$a$ and~$c$, contradicting our assumption.
\end{proof}

\begin{proof}[Proof of Claim~\ref{claim:decontractionCompare}]
By definition, we have~${\less} \subseteq {\bless}$. Assume now that~${\Dec{\less}} \ne {\Dec{\bless}}$, and let~$x < y$ be such that~$x \not\more y$ but~$x \bmore y$. By definition of~$\bless$, there exists a minimal path~$y = z_0, z_1, \dots, z_k = x$ such that for all~$i \in [k]$, either~$z_{i-1} \less z_i$, or~$z_{i-1} < z_i$ are incomparable in~$\less$ and do not satisfy~$(\spadesuit)$. Since~$x \not\more y$ and~$x < y$, we have~$k \ge 2$ and there exists~$i \in [k-1]$ such that~$z_{i+1} < z_{i-1}$. We distinguish three cases:
\begin{itemize}
\item If~${z_i < z_{i+1} < z_{i-1}}$, then~$z_i \more z_{i-1}$ and~$z_i \not\more z_{i+1}$, and thus~${z_{i+1} \more z_{i-1}}$ as~${{\less} \in \DWOIP(n)}$.
\item If~$z_{i+1} < z_i < z_{i-1}$, then~$z_{i+1} \more z_i \more z_{i-1}$ and thus~$z_{i+1} \more z_{i-1}$ by transitivity.
\item If~$z_{i+1} < z_{i-1} < z_i$, then~$z_{i+1} \more z_i$ and~$z_{i-1} \not\more z_i$, and thus~$z_{i+1} \more z_{i-1}$ as~${\less} \in \DWOIP(n)$.
\end{itemize}
In all cases, $z_{i+1} \more z_{i-1}$ contradicts the minimality of the path.
\end{proof}

\begin{proof}[Proof of Claim~\ref{claim:decontraction}]
We first show that~${\bless} \in \PIP$. Since~${\Dec{\less}} = {\Dec{\bless}}$ and~${\less} \in \DPIP$, we have~${\bless} \in \DPIP$ and we just need to show that~${\bless} \in \IPIP$. Consider thus~$a < b < c$ such that~$a \bless c$. By definition of~$\bless$, there exists~$a' < b < c'$ such that~$a \bless a'$, $c' \bless c$, and either~$a' \less c'$, or~$a'$ and~$c'$ are incomparable in~$\less$ and do not satisfy~$(\spadesuit)$. We now proceed in two steps:
\begin{enumerate}[(i)]
\item Our first goal is to show that either $a' \bless b$ or $b \bless c'$ which by transitivity shows that $(a,c)$ satisfies the $\WOIP$ condition. Assume that~$a' \not\bless b \not\bless c'$. The transitivity of~$\bless$ ensures that both pairs~$a',b$ and~$b,c'$ are incomparable in~$\bless$. Therefore, they are incomparable in~$\less$ and satisfy~$(\spadesuit)$. Let us focus on~$a',b$. Assume first that there is~$a' < d < b$ such that~$d \in \orientation\positive$ and~${a' \not\less d \not\more b}$. Since~${{\less} \in \IPIP}$, we cannot have~$a' < d < c'$, $d \in \orientation\positive$, $a' \not\less d$ and~$a' \less c'$. Therefore, $a'$ and~$c'$ do not satisfy~$(\spadesuit)$, which together with~$a' \not\less d$ implies that~$d \more c'$. We obtain~${d < b < c'}$ with~$d \not\more b \not\more c'$ while~$d \more c'$ contradicting that~${\less} \in \DWOIP(n)$. Assume now that there is~$a' < d < b$ such that~$d \in \orientation\negative$ and~$a' \not\more d \not\less b$. If~$a' \less c'$, then~$d \less c'$ since~${a' < d < c'}$, $d \in \orientation\negative$ and~${\less} \in \IPIP$. If~$a'$ and~$c'$ do not satisfy~$(\spadesuit)$, $d \in \orientation\negative$ and~$a' \not\more d$ imply~$d \less c'$. In both cases, we obtain~$d < b < c'$ with~$d \not\less b \not\less c'$ while~$d \less c'$ contradicting that~${{\less} \in \IWOIP(n)}$. Since we reach a contradiction in all cases, we conclude that~${a' \bless b}$~or~${b \bless c'}$.
\item We now want to check that the orientation constraint on $b$ is also satisfied. Assume~${b \in \orientation\positive}$. If~$a' \less c'$, we have~$a' \less b$ since~${\less} \in \IPIP$, and thus~$a' \bless b$ since~${\less} \subseteq {\bless}$. If~$a'$ and~$c'$ do not satisfy~$(\spadesuit)$, then~$a' \less b$ or~$b \more c'$, which implies~${a' \less b \more c'}$ by~$(\clubsuit)$, and thus~$a' \bless b$ since~${\less} \subseteq {\bless}$. We conclude that $b \in \orientation\positive \Rightarrow a' \bless b$ and by symmetry that~$b \in \orientation\negative \Rightarrow b \bless c'$.
\end{enumerate}
Since~$a \bless a'$, $c' \bless c$ and~$\bless$ is transitive, we obtain that $a \bless b$ or~$b \bless c$, and that $b \in \orientation\positive \Rightarrow a \bless b$ and~$b \in \orientation\negative \Rightarrow b \bless c$. We conclude that~${\bless} \in \PIP$.

\medskip
There is left to prove that any~$a < c$ incomparable in~$\bless$ satisfy~$(\spadesuit)$. Assume the opposite and consider~$a < c$ incomparable in~$\bless$ not satisfying~$(\spadesuit)$ with $c-a$ minimal. Since~$a$ and~$c$ are incomparable in~$\bless$, they are also incomparable in~$\less$ and satisfy~$(\spadesuit)$. Assume for example that there exists~$b \in \orientation\positive$ such that~$a \not\less b \not\more c$ (the other case is symmetric). Since~${\Dec{\less}} = {\Dec{\bless}}$, we have~$b \not\bmore c$. Since~$a$ and~$c$ do not satisfy~$(\spadesuit)$ in~$\bless$, we obtain that~$a \bless b$. We can assume that~$b$ is the maximal integer such that~$a < b < c$, $b \in \orientation\positive$ and~$a \bless b \not\bmore c$. Since~$a \bless b$ but~$a \not\bless c$, we have~$b \not\bless c$, so that~$b$ and~$c$ are incomparable in~$\bless$. By minimality of~$c-a$, we obtain that~$b$ and~$c$ satisfy~$(\spadesuit)$ in~$\bless$. We distinguish two cases:
\begin{enumerate}[(i)]
\item Assume that there exists~$b < d < c$ such that~$d \in \orientation\positive$ and~$b \not\bless d \not\bmore c$. Since~$a$ and~$c$ do not satisfy~$(\spadesuit)$ in~$\bless$, we have~$a \bless d \not\bmore c$, contradicting the maximality of~$b$.
\item Assume that there exists~$b < d < c$ such that~$d \in \orientation\negative$ and~$b \not\bmore d \not\bless c$. Since~$a$ and~$c$ do not satisfy~$(\spadesuit)$ in~$\bless$ and~$d \not\bless c$, we have~$a \bmore d$. We thus obtained~$a < b < d$ with $b \in \orientation\positive$, and $a \bmore d$ while $b \not\bmore d$ contradicting that~${\bless} \in \PIP$.
\end{enumerate}
Since we obtain a contradiction in both cases, we conclude that any~$a < c$ incomparable in~$\bless$ satisfy~$(\spadesuit)$.
\end{proof}

\begin{proof}[Proof of Claim~\ref{claim:contraction}]
We first prove that~${\less} \subseteq {\less_{\tree[S]}}$. Observe first that for a permutree~$\tree$ and a Schr\"oder permutree~$\tree[S]$ obtained from~$\tree$ by contracting a subset of edges~$E$, the poset~$\less_{\tree[S]}$ is obtained from the poset~$\less_{\tree}$ by deleting the sets
\begin{gather*}
\set{(a,d)}{a < d, \; \exists \; a \le b < c \le d, \; b \in \{a\} \cup \orientation\negative , \; c \in \{d\} \cup \orientation\positive, \; a \less_{\tree} b \less_{\tree} c \less_{\tree} d \text{ and } (b,c) \in E}, \\
\set{(d,a)}{a < d, \; \exists \; a \le b < c \le d, \; b \in \{a\} \cup \orientation\positive, \; c \in \{d\} \cup \orientation\negative, \; a \more_{\tree} b \more_{\tree} c \more_{\tree} d \text{ and } (c,b) \in E}.
\end{gather*}
Assume now that we had~${\less} \not\subseteq {\less_{\tree[S]}}$ and remember that ${\less} \subseteq {\less_{\tree[S]}}$ by construction. Since we only contract increasing edges in~$\less_{\tree}$ to obtain~$\less_{\tree[S]}$, this would imply that there exists~$a \le b < c \le d$ with~$b \in \{a\} \cup \orientation\negative$, $c \in \{d\} \cup \orientation\positive$, and such that~$a \less d$ while~$b \not\less d$. This would contradict that~${\less} \in \PIP$.

We now prove that~${\less_{\tree[S]}} \subseteq {\less}$. Observe first that~${\Dec{\less_{\tree[S]}}} \subseteq {\Dec{\bless}} = {\Dec{\less}}$. Assume now by contradiction that there exists~$a < c$ such that~$a \not\less c$ and~$a \less_{\tree[S]} c$, and choose such~$a < c$ with~$c-a$ minimal. Note that~$a$ and~$c$ are incomparable in~$\less$. We distinguish two cases:
\begin{enumerate}[(i)]
\item If~$a$ and~$c$ satisfy~$(\spadesuit)$, we can assume by symmetry that there exists~$a < b < c$ such that~$b \in \orientation\positive$ and~$a \not\less b \not\more c$. Since~$a < b < c$, $b \in \orientation\positive$, $a \less_{\tree[S]} c$ and~${\less_{\tree[S]}} \in \PIP$, we obtain that~$a \less_{\tree[S]} b$. Since~$a \not\less b$ and~$a \less_{\tree[S]} b$, this contradicts the minimality of~$c-a$.
\item If~$a$ and~$c$ do not satisfy~$(\spadesuit)$, then they satisfy~$(\clubsuit)$, and~$a \bless c$ is not a cover relation (otherwise, the relation~$a \bless c$ would have been contracted in~$\less_{\tree[S]}$). Let~$b \in [n] \ssm \{a,c\}$ be such that~$a \bless b \bless c$. If~$a < b < c$, we have~$a \not\more b \not\more c$, thus~$a \not\less b \not\less c$ by~$(\clubsuit)$, thus~$a \not\less_{\tree[S]} b \not\less_{\tree[S]} c$ by minimality of~$c-a$, contradicting that~$a \less_{\tree[S]} c$ and~${\less_{\tree[S]}} \in \WOIP(n)$. We can thus consider that~$b < a < c$ (the case~$a < c < b$ is symmetric). Note that we cannot have~$a \in \orientation\positive$ since~$b \not\bless a$ and~$b \bless c$ would contradict that~${\bless} \in \PIP$. We have~$b \more a$ (since~$b \bmore a$) and we can assume that~$b$ is maximal such that~$b < a$ and~$b \more a$. Observe that $b$ and~$c$ are incomparable in~$\less$ (indeed~$b \not\less c$ since~$b \more a$ and~$a \not\less c$, and~$b \not\more c$ since~$b \bless c$). Since~$b < a < c$ and~$b \more a$ while~$a \not\less c$, $b$ and~$c$ do not satisfy~$(\clubsuit)$, thus they satisfy~$(\spadesuit)$. We again have two cases:
\begin{itemize}
\item If there is~$b < d < c$ with~$d \in \orientation\positive$ with~$b \not\less d \not\more c$. If~$b < d < a$, we have~$b \more a$ and~${\less} \in \PIP$ implies~$d \more a$ contradicting the maximality of~$b$. Since~$d \in \orientation\positive$, we have~$d \ne a$. Finally, if~$a < d < c$, we have~$a \not\less d$ since~$d \not\more c$ and~$a$ and~$c$ satisfy~$(\clubsuit)$, and we obtain that~$a \not\less d \not\more c$ contradicting that~$a$ and~$c$ do not satisfy~$(\spadesuit)$. 
\item If there is~$b < d < c$ with~$d \in \orientation\negative$ with~$b \not\more d \not\less c$, then we have~$a < d < c$ (because~$b \not\more d$, $b \more a$ and~${\less} \in \PIP$). Since~$d \not\less c$ and~$a$ and~$c$ satisfy~$(\clubsuit)$, we obtain that~$a \not\more d \not\less c$, contradicting that~$a$ and~$c$ do not satisfy~$(\spadesuit)$.
\end{itemize}
\end{enumerate}
As we obtain a contradiction in all cases, we conclude that~${\less_{\tree[S]}} \subseteq {\less}$.
\end{proof}

\subsection{Proof of claims of Section~\ref{subsec:PIPdeletion}}
\label{subsec:appendixPIPdeletion}

\begin{proof}[Proof of Claim~\ref{claim:IPIPpmidPoset}]
First,~$\IPIPpmid{\less}$ is clearly antisymmetric as it is contained in the antisymmetric relation~$\less$. To prove that it is transitive, consider~$u, v, w \in [n]$ such that~${u \IPIPpmid{\less} v \IPIPpmid{\less} w}$. Since~${\IPIPpmid{\less}} \subseteq {\less}$, we have~$u \less v \less w$, so that~$u \less w$ by transitivity of~$\less$. Assume by means of contradiction that~$u \not\IPIPpmid{\less} w$. Thus, $u < w$ and there exists~$u \le n < p \le w$ such that~$n \in \{u\} \cup \orientation\negative$ while~$p \in \{w\} \cup \orientation\positive$ and~$n \not\less p$. We now distinguish three cases:
\begin{itemize}
\item If~$v \le n$, then~$n \not\less p$ and~$v \le n < p \le w$ contradicts our assumption that~$v \IPIPpmid{\less} w$.
\item If~$p \le v$, then~$n \not\less p$ and~$u \le n < p \le v$ contradicts our assumption that~$u \IPIPpmid{\less} v$.
\item Finally, if~$n < v < p$, then~$u \IPIPpmid{\less} v$ ensures that~$n \less v$, and~$v \IPIPpmid{\less} w$ ensures that~$v \less p$. Together with~$n \not\less v$, this contradicts the transitivity of~$\less$.
\end{itemize}
As we obtained a contradiction in each case, we conclude that~$\IPIPpmid{\less}$ is transitive. Since all our relations are reflexive, we conclude that~$\IPIPpmid{\less}$ is a poset.
\end{proof}

\begin{proof}[Proof of Claim~\ref{claim:IPIPpmidIPIPpm}]
We prove that~$\IPIPpmid{\less}$ is in~$\IPIPp$, the result follows by symmetry for~$\IPIPm$ and finally for~$\IPIPpm = \IPIPp \cap \IPIPm$. Assume that there exists~${a < b < c}$ with~$b \in \orientation\positive$ and~$a \IPIPpmid{\not\less} b$. Then there are witnesses~$a \le n < p \le b$ with~$n \in \{a\} \cup \orientation\negative$ while~$p \in \{b\} \cup \orientation\positive$ and~$n \not\less p$. Since~$b \in \orientation\positive$, we have~$p \in \orientation\positive$. Therefore, $n$ and~$p$ are also witnesses for~$a \IPIPpmid{\not\less} c$. This shows that~$\IPIPpmid{\less}$ is in~$\IPIPp$.
\end{proof}

\begin{proof}[Proof of Claim~\ref{claim:IPIPidIPIP}]
Let~${\bless} \in \IWOIP(n)$. Let~${a < b < c}$ be such that~$a \IPIPpmid{\not\bless} b \IPIPpmid{\not\bless} c$ and $a \IPIPpmid{\bless} c$. Then there exist witnesses~$a \le m < p \le b \le n < q \le c$ with~${m \in \{a\} \cup \orientation\negative}$, ${p \in \{b\} \cup \orientation\positive}$, $n \in \{b\} \cup \orientation\negative$ and~$q \in \{c\} \cup \orientation\positive$, and such that~$m \not\bless p$~and~$n \not\bless q$. If~$p \ne b$, then $p \in \orientation\positive$ and~${a \le m < p < c}$ are also witnesses for~$a \IPIPpmid{\not\bless} c$. By symmetry, we can thus assume that~$p = b = n$. Therefore, we have~$m \not\bless p=b=n \not\bless q$, which implies that~$m \not\bless q$ since~${\bless} \in \IWOIP(n)$. Since~$a \le m < q \le c$ with~$m \in \{a\} \cup \orientation\negative$, $q \in \{c\} \cup \orientation\positive$ and~$m \not\bless q$, we obtain that~$a \IPIPpmid{\not\bless} c$. We conclude that~${\bless} \in \IWOIP(n)$ implies~${\IPIPpmid{\bless}} \in \IWOIP(n)$. In particular, we obtain that~${\IPIPid{\less}} = {\IPIPpmid{(\IWOIPid{\less})}} \in \IPIP$ since~${\IWOIPid{\less}} \in \IWOIP(n)$ by Lemma~\ref{lem:IWOIPidDWOIPdd1}.
\end{proof}

\begin{proof}[Proof of Claim~\ref{claim:IPIPidDPIPdd2}]
Claim~\ref{claim:IPIPidDPIPdd2} is immediate for~$\varepsilon \in \{-,+\}$. For~$\varepsilon = \pm$, assume that there exists ${a \le n < p \le c}$ with~$n \in \{c\} \cup \orientation\negative$ and~$p \in \{b\} \cup \orientation\positive$ such that~$n \not\less p$. Then we also have~$n \not\bless p$ which implies~$n \not\bless c$ (since~${\bless} \in \IPIPp$) and~$a \not\bless p$ (since~${\bless} \in \IPIPm$), which in turn implies~$a \not\bless c$. Finally, this also implies the claim when~$\varepsilon = \varnothing$ by applying first Lemma~\ref{lem:IWOIPidDWOIPdd2} and then the claim for~$\varepsilon = \pm$.
\end{proof}

\begin{proof}[Proof of Claim~\ref{claim:blessSnakes}]
Otherwise, there exists~$a < c$ such that there is no $\orientation$-snake from~$a$ to~$c$ in~$\bless$. Choose such a pair~$a < c$ with~$c-a$ minimal. Since~$\less_\sigma$ is a total order, we have either~$a \less_\sigma c$ or~$a \more_\sigma c$. Assume for example that~$a \less_\sigma c$, the other case being symmetric. Since~$a \less_\sigma c$ while~$a \not\bless c$ (otherwise $a \bless c$ is an $\orientation$-snake), there exists~$a \le n < p \le c$ such that~$n \in \{a\} \cup \orientation\negative$ while~$p \in \{c\} \cup \orientation\positive$ and~$n \not\less_\sigma p$. Since~$\less_\sigma$ is a total order, we get~$n \more_\sigma p$. Moreover, we have either~$a \less_\sigma n$ or~$a \more_\sigma n$. In the latter case, we get by transitivity of~$\less_\sigma$ that~$a \more_\sigma p$. Therefore, up to forcing~$a = n$, we can assume that~$a \less_\sigma n$ and similarly up to forcing~$p = c$, we can assume that~$p \less_\sigma c$. It follows that~$a \less_\sigma n \more_\sigma p \less_\sigma c$ is an $\orientation$-snake from~$a$ to~$c$ in~$\less_\sigma$, where either~$a \ne n$ or~$p \ne c$ (because~$a \less_\sigma c$). By minimality of~$c-a$ in our choice of~$a < c$, there exists an $\orientation$-snake from~$a$ to~$n$, from~$n$ to~$p$, and from~$p$ to~$c$ in~$\bless$. Since~$n \in \{a\} \cup \orientation\negative$ and~$p \in \{c\} \cup \orientation\positive$, it is straightforward to construct from these snakes an $\orientation$-snake from~$a$ to~$c$ in~$\bless$, contradicting our assumption.
\end{proof}

\begin{proof}[Proof of Claim~\ref{claim:blessSpadeClub}]
Assume that there exists~$a < c$ incomparable in~$\bless$ that do not satisfy~$(\clubsuit)$ in~$\bless$. We choose such a pair~$a < c$ with~$c-a$ minimal. By symmetry, we can assume that there exists~$a < b < c$ such that~$a \bless b \not\bmore c$ and that~$b$ is maximal for this property. Since~${\bless} \subseteq {\less}$, we have~$a \less b$. We distinguish three cases:
\begin{enumerate}[(i)]
\item Assume first that~$a$ and~$c$ are incomparable in~$\less$. Since~${\less} \in \WOFP(n)$, Proposition~\ref{prop:characterizationWOFP} and~$a \less b$ imply that~$a \less b \more c$. Since~$b \more c$ while~$b \not\bmore c$, there is~$b \le p < n \le c$ with~$p \in \{b\} \cup \orientation\positive$ while~$n \in \{c\} \cup \orientation\negative$ and~$p \not\more n$. We again have two cases:
    \begin{itemize}
    \item If~$b \ne p$, we have~$p \not\bmore c$ (since otherwise~$p \not\more n$ would contradict that~${{\bless} \in \DPIP}$) and thus~$a \not\bless p$ (by maximality of~$b$). We thus obtained~$a < p < c$ with~$p \in \orientation\positive$ and~$a \not\bless p \not\bmore c$, so that~$a < c$ satisfy~$(\spadesuit)$ in~$\bless$.
    \item If~$b = p$, then~$n \ne c$. We have~$a \not\more n$ (since otherwise~$p \not\more n$ would contradict that~${{\bless} \in \PIP}$). Moreover, by minimality of~$c-a$, we have~$b = p$ and~$c$ satisfy~$(\clubsuit)$ in~$\bless$, so that~$p \not\bmore n$ implies that~$n \not\bless c$. We obtained that~$a < n < c$ with~$n \in \orientation\negative$ and~$a \not\bmore n \not\bless c$, so that~$a < c$ satisfy~$(\spadesuit)$ in~$\bless$.
    \end{itemize}
\item Assume now that~$a \less c$. Since~$a \less c$ while~$a \not\bless c$, there is~$a \le n < p \le c$ with~$n \in \{a\} \cup \orientation\negative$ while~$p \in \{c\} \cup \orientation\positive$ and~$n \not\less p$. Since~${\bless} \in \PIP$ and~$n \not\bless p$, we must have~$a \not\bless p$ and~$n \not\bless c$. Assume that~$a$ and~$c$ do not satisfy~$(\spadesuit)$ in~$\bless$. This implies that~$p \bmore c$ and~$a \bmore n$. Since~$a \not\bmore c$ and~$p \bmore c$, we obtain by transitivity of~$\bless$ that~$a$ and~$p$ are incomparable in~$\bless$. By minimality of~$c-a$, we obtain that~$a$ and~$p$ satisfy~$(\clubsuit)$. We now consider two cases:
	\begin{itemize}
	\item If~$b < p$, then~$a \bless b$ implies that~$b \bmore p$, which together with~$p \bmore c$ and~$b \not\bmore c$ contradicts the transitivity of~$\bless$.
	\item If~$p \le b$, then we have~$a \le n < p \le b$ with~$n \in \{a\} \cup \orientation\negative$ while~$p \in \{c\} \cup \orientation\positive$ and~$n \not\less p$, which contradicts that~$a \bless b$.
	\end{itemize}
Since we obtained a contradiction in both cases, we conclude that~$a$ and~$c$ satisfy~$(\spadesuit)$ in~$\bless$.
\item Assume finally that~$a \more c$. Then~$a \not\more b$ and~${\less} \in \DWOIP(n)$ implies that~$a \less b \more c$ and we are back to case~(i). \qedhere
\end{enumerate}
\end{proof}

\subsection{Proof of claims of Theorem~\ref{thm:coveringOrientationConflictFunctionsSublatticesIPos}}
\label{subsec:appendixCoveringOrientationConflictFunctionsSublatticesIPos}

\begin{proof}[Proof of Claim~\ref{claim:dashvDPIP}]
As $\conflicts_{\DPIP}$ is decreasing, we only consider~$(a,c) \in {\Dec{\dashv}}$. Since~${{\Dec{\dashv}} = {\Dec{\less}} \cap {\Dec{\bless}}}$, we have~$a \more c$ and~$a \bmore c$. Since both~${\less, \bless} \in \DPIP$, for any~$a < b < c$, if~$b \in \orientation\negative$ then~$a \more b$ and~$a \bmore b$ so that~$a \vdash b$, while if~$b \in \orientation\positive$ then~$b \more c$ and~$b \bmore c$ so that~$b \vdash c$. Note that the important point here is that the behavior of~$b$ is the same in~$\less$ and~$\bless$ as it is dictated by the orientation of~$b$.
\end{proof}

\begin{proof}[Proof of Claim~\ref{claim:tdddashvDPIP}]
Assume now that~${\tdd{\dashv}} \notin \DPIP$. Consider~$\{a,c\} \in \conflicts_{\DPIP}(\tdd{\dashv})$ with~$a < c$ and~$c-a$ minimal. Since~$\conflicts_{\DPIP}(\tdd{\dashv})$ is decreasing, we have~$a \tdd{\vdash} c$. Assume for the moment that there exists~$a < b < c$ such that~$b \in \orientation\negative$ and $a \not\tdd{\vdash} b$, and choose such~$b$ with~$b-a$ minimal. Since~${\dashv} \in \DPIP$, we have~$a \vdash b$ while~$a \not\tdd{\vdash} b$. By definition of~$\tdd{\dashv}$, there exists~$i \le b$ and~$j \ge a$ such that~$i \dashv b \dashv a \dashv j$ but~$i \not\dashv j$. From Lemma~\ref{lem:simplifytdd}, we know that either~$i \ne b$ or~$j \ne a$. We thus distinguish two cases.
\begin{enumerate}[(i)]
\item Assume that~$i \ne b$. Again by Lemma~\ref{lem:simplifytdd}, there exists~$a < k < b$ such that~$i \dashv k \dashv b$. Thus, we have~$k \not\vdash b$ (since~$\dashv$ is antisymmetric) while~$a \vdash b$ and~$a < k < b$, so $k \notin \orientation\positive$ (since~$\dashv$ if $\conflicts_{\DPIP}$-free). Since~$\orientation$ is covering, we therefore obtain that~$k \in \orientation\negative$. By minimality of~$b-a$ in our choice of~$b$, we obtain that~$a \tdd{\vdash} k$. But~$i \dashv k \dashv a \dashv j$ and~$a \tdd{\vdash} k$ implies that~$i \dashv j$, a contradiction to our assumption on~$i$ and~$j$.
\item Assume now that~$j \ne a$. Again by Lemma~\ref{lem:simplifytdd}, there exists~$a < k < b$ such that~${a \dashv k \dashv j}$. Thus, we have~$a \not\vdash k$ (since~$\vdash$ is antisymmetric) while~$a \vdash b$ and~$a < k < b$, so ${k \notin \orientation\negative}$ (since~$\less$ if $\conflicts_{\DPIP}$-free). Since~$\orientation$ is covering, we therefore obtain that~$k \in \orientation\positive$. Since~${{\dashv} \in \DPIP}$, $k \in \orientation\positive$ and~$a \vdash c$, we have~$k \vdash c$. We claim that~$k \tdd{\vdash} c$. Otherwise we could find~$i' \le c$ and~$j' \ge k$ such that~$i' \dashv c \dashv k \dashv j'$ while~$i' \not \dashv j'$. Since~$a \dashv k \dashv j'$ and~$\dashv$ is semitransitive, we would also have~$i' \dashv c \dashv a \dashv j'$ while~$i' \not \dashv j'$, contradicting the fact that~$a \tdd{\vdash} c$. Now by minimality of~$c-a$ in our choice of~$(a,c)$, we obtain that~$(c,k) \in {\tdd{\dashv}} \ssm \, \conflicts_{\DPIP}(\tdd{\dashv})$. Therefore, since~$b \in \orientation\negative$, we have~$k \tdd{\vdash} b$. But~$i \dashv b \dashv k \dashv j$ and~$k \tdd{\vdash} b$ implies that~$i \dashv j$, a contradiction to our assumption on~$i$~and~$j$.
\end{enumerate}
We therefore proved that~$a \tdd{\vdash} b$ for all~$a < b < c$ with~$b \in \orientation\negative$. The case of~$b \in \orientation\positive$ is symmetric and left to the reader. This concludes the proof.
\end{proof}

\subsection{Proof of claims of Proposition~\ref{prop:orientationIncompSublatticesIPos}}
\label{subsec:appendixOrientationIncompSublatticesIPos}

\begin{proof}[Proof of Claim~\ref{claim:dashvConflictFree}]
Assume that~$\dashv$ is not $\conflicts_{\incompE}$-free and let~$\{a,c\} \in \conflicts_{\incompE}(\dashv)$ with~$a < c$ and~$c-a$ minimal. Since~$a \not\dashv c$, we have~$a \not\less c$ and~$a \not\bless c$. Since~$a \not\vdash c$, we have~$a \not\more c$ or~$a \not\bmore c$. We can thus assume without loss of generality that~$a$ and~$c$ are incomparable in~$\less$. Since~$\less$ is $\conflicts_{\incompE}$-free, there exists~$a < b_1 < \dots < b_k < c$ such that either~$b_{2i} \in \orientation\positive$, $b_{2i+1} \in \orientation\negative$ and~$a \less b_1 \more b_2 \less b_3 \more \cdots$, or~$b_{2i} \in \orientation\negative$, $b_{2i+1} \in \orientation\positive$ and~$a \more b_1 \less b_2 \more b_3 \less \cdots$. We distinguish these two cases:
\begin{enumerate}[(i)]
\item In the former case, we obtain~$b_1 \in \orientation\negative$ and~$a \dashv b_1$ (since~${\Inc{\less}} \subseteq {\dashv}$).
\item In the latter case, we distinguish three cases according to the order of~$a$ and~$b_1$ in~$\bless$:
	\begin{itemize}
	\item if~$a \bless b_1$, then~$a \dashv b_1 \dashv b_2$ (since~${\Inc{\less}} \cup {\Inc{\bless}} \subseteq {\dashv}$) so that we obtain~$b_2 \in \orientation\negative$ and~$a \dashv b_2$.
	\item if~$a \bmore b_1$, we obtain~$b_1 \in \orientation\positive$ and~$a \vdash b_1$ (since~${\Dec{\less}} \cap {\Dec{\bless}} \subseteq {\dashv}$).
	\item if~$a$ and~$b_1$ are incomparable in~$\bless$, then they are also incomparable in~$\dashv$. By minimality of~$c-a$ in our choice of~$(a,c)$, we have~$\{a,b_1\} \notin \conflicts_{\incompE}(\dashv)$. Lemma~\ref{lem:orientationIncompCharacterization}\,(iii) thus ensures the existence of~$a < b < b_1$ such that either~$b \in \orientation\negative$ and $a \dashv b$, or $b \in \orientation\positive$ and~$a \vdash b$.
	\end{itemize}
\end{enumerate}
In all situations, we have found~$a < b < c$ such that either~$b \in \orientation\negative$ and $a \dashv b$, or $b \in \orientation\positive$ and~$a \vdash b$. Since~$\{b,c\} \notin \conflicts_{\incompE}(\dashv)$ by minimality of~$c-a$ in our choice of~$(a,c)$, Lemma~\ref{lem:orientationIncompCharacterization}\,(iii) thus contradicts that~$\{a,c\} \in \conflicts_{\incompE}(\dashv)$.
\end{proof}

\begin{proof}[Proof of Claim~\ref{claim:tdddashvConflictFree}]
Assume that~$\tdd{\dashv}$ is not $\conflicts_{\incompE}$-free and let~${\{a,c\} \in \conflicts_{\incompE}(\tdd{\dashv})}$ with~$a < c$ and~$c-a$ minimal. We distinguish two cases:
\begin{enumerate}[(i)]
\item If~$a \not\vdash c$, since~$\dashv$ is~$\conflicts_{\incompE}$-free, Lemma~\ref{lem:orientationIncompCharacterization}\,(iii) ensures that there exists~$a < b < c$ such that~$b \in \orientation\negative$ and $a \dashv b$, or $b \in \orientation\positive$ and~$a \vdash b$. In the former case, we also have~$a \tdd{\dashv} b$. In the latter case, we either have~$a \tdd{\vdash} b$ or $a$ and~$b$ are incomparable in~$\tdd{\dashv}$. By minimality of~$c-a$ in our choice of~$(a,c)$, we have~$\{a,b\} \notin \conflicts_{\incompE}(\tdd{\dashv})$. We thus obtain by Lemma~\ref{lem:orientationIncompCharacterization}\,(iii) that there exists~$a < b' < b$ such that~$b' \in \orientation\negative$ and $a \tdd{\dashv} b'$, or $b' \in \orientation\positive$ and~$a \tdd{\vdash} b'$.
\item If~$a \vdash c$, then there exists~$i \le c$ and~$j \ge a$ such that~$i \dashv c \dashv a \dashv j$ while~$i \not\dashv j$. From Lemma~\ref{lem:simplifytdd}, we can assume for example that~$i \ne c$ so that there exists~$a < k < c$ with~${i \dashv k \dashv c}$ (the proof when~$j \ne a$ is similar). Note that~$a \not\tdd{\dashv} k$ since otherwise~$a \tdd{\dashv} k \tdd{\dashv} c$ and~$a \not\tdd{\dashv} c$ would contradict the transitivity of~$\tdd{\dashv}$. Moreover, $a \not\tdd{\vdash} k$ since either we already have~$a \not\vdash k$, or $i \le c$ and~$j \ge a$ still satisfy~$i \dashv k \dashv a \dashv j$ while~$i \not\dashv j$. Therefore, $a$ and~$k$ are incomparable in~$\dashv$. By minimality of~$c-a$ in our choice of~$(a,c)$, we have~$\{a,k\} \notin \conflicts_{\incompE}(\tdd{\dashv})$. We thus obtain by Lemma~\ref{lem:orientationIncompCharacterization}\,(iii) that there exists~$a < b < k$ such that~$b \in \orientation\negative$ and $a \tdd{\dashv} b$, or $b \in \orientation\positive$ and~$a \tdd{\vdash} b$.
\end{enumerate}
In all situations, we have found~$a < b < c$ such that either~$b \in \orientation\negative$ and $a \tdd{\dashv} b$, or $b \in \orientation\positive$ and~$a \tdd{\vdash} b$. Since~$\{b,c\} \notin \conflicts_{\incompE}(\tdd{\dashv})$ by minimality of~$c-a$ in our choice of~$(a,c)$, Lemma~\ref{lem:orientationIncompCharacterization}\,(iii) thus contradicts that~$\{a,c\} \in \conflicts_{\incompE}(\tdd{\dashv})$.
\end{proof}

\end{document}